\documentclass[12pt]{article}

\usepackage{amsmath}
\usepackage{amsfonts}
\usepackage{amssymb}
\usepackage{extarrows}
\usepackage{amscd}
\usepackage{amsthm}
\usepackage{textcomp}
\usepackage{bbm}
\usepackage{color}
\usepackage[linktocpage=true,plainpages=false,pdfpagelabels=false]{hyperref}
\usepackage{euscript}

\input xypic

\usepackage[matrix,arrow,curve]{xy}

\usepackage{color}

\newcommand{\quash}[1]{}


\numberwithin{equation}{section}

\newtheorem{defin}{Definition}[section]
\newtheorem{prop}{Proposition}[section]
\newtheorem{nt}{Remark}[section]
\newtheorem{Th}{Theorem}[section]
\newtheorem{lemma}{Lemma}[section]

\newtheorem{defin-prop}{Definition-proposition}[section]

\newfont{\ssdbl}{msbm8}
\newfont{\sdbl}{msbm9}
\newfont{\dbl}{msbm10 at 12pt}

\newcommand{\oo}{{\cal O}}

\newcommand{\g}{{\cal G}}

\newcommand{\Hom}{\mathop {\rm Hom}}
\newcommand{\Homb}{{\Hom\nolimits^{\blacklozenge}}}

\newcommand{\tr}{\mathop {\rm tr}}

\newcommand{\dz}{\mathbb{Z}}
\newcommand{\dc}{\mathbb{C}}

\newcommand{\dr}{\mathbb{R}}

\newcommand{\Z}{\dz}

\newcommand{\sdc}{\mathbb{C}}

\newcommand{\Ker}{{\rm Ker}\:}

\newcommand{\lrto}{\longrightarrow}

\newcommand{\dn}{\mathbb{N}}

\def\Z{{\mathbb Z}}
\def\Q{{\mathbb Q}}
\def\N{{\mathbb N}}
\def\C{{\mathbb C}}

\newcommand{\Lie}{\mathop {\rm Lie}}

\newcommand{\ve}{\varepsilon}

\newcommand{\h}{{{\mathcal{H}}}}

\newcommand{\DS}{{\mathop{\rm Diff}}^+_{\rm hol}(S^1)}

\newcommand{\WDS}{\widehat{\DS}}

\newcommand{\G}{{\mathcal G}}

\newcommand{\xlrto}{\xlongrightarrow}

\begin{document}

\author{
Denis V. Osipov
}

\title{Analytic diffeomorphisms of the circle and topological Riemann-Roch theorem for circle fibrations  \thanks{The work was supported by the HSE University Basic Research Program.}}
\date{}

\maketitle

\begin{abstract}
We consider the group $\mathcal G$ which is the semidirect product of the group of analytic functions with values in ${\mathbb C}^*$ on the circle and the group of analytic diffeomorphisms of the circle that preserve the orientation.
Then we construct the central extensions of the group $\mathcal G$ by the group ${\mathbb C}^*$. The first central extension,  so-called the determinant central extension, is constructed by means of determinants of linear operators acting in infinite-dimensional locally convex topological $\C$-vector spaces. Other central extensions are constructed by $\cup$-products of group $1$-cocycles with the application to them the map related with algebraic $K$-theory.
We prove in the second cohomology group, i.e. modulo of a group $2$-coboundary, the equality of the $12$th power of the $2$-cocycle constructed by the first central extension and the product of integer powers of the $2$-cocycles constructed above by means of \linebreak $\cup$-products (in multiplicative notation).
As an application of this result we obtain a new topological Riemann-Roch theorem for a complex line bundle $L$ on a smooth manifold $M$,
where $\pi :M \to B$ is a fibration in oriented circles. More precisely, we prove that in the group $H^3(B, {\mathbb Z})$ the element $12 \, [ {\mathcal Det}  (L)]$ is equal to the element  $6 \, \pi_* ( c_1(L)  \cup c_1(L))$,
where $[{\mathcal Det}  (L)]$ is the class of the determinant gerbe on $B$ constructed by $L$ and the determinant central extension.

\end{abstract}

\tableofcontents

\section{Introduction}
\subsection{Infinite-dimensional Lie groups}

In this paper we consider the group
$$
\g =  \h^*  \rtimes \DS  \, \mbox{,}
$$
where $\h^*$ is the group of analytic functions on the circle $S^1$ with values in $\C^*$, where $\C^* = \C \setminus 0$, and $\DS$ is the group of analytic diffeomorphisms of $S^1$ that preserve the orientation.
Clearly, the group $\DS$ acts on the group $\h^*$.

We will consider also
$$S^1  = \{z \in \mathbb{C} \, \mid  \,  |z|  =1   \} $$
inside of the Riemann sphere $\widehat{\C} = \C \cup \infty$.

The groups $\h^*$, $\DS$ and $\g$ are infinite-dimensional smooth Lie groups modelled on locally convex topological vector spaces defined over the fields $\dr$ or $\C$, see Section~\ref{Expl_descr} below.
More exactly, these locally convex topological vector spaces are so-called Silva spaces, where a Silva space is a countable inductive limit (in the category of locally convex topological vector spaces) of Banach spaces with compact inclusion maps between Banach spaces which have successive indices in inductive system, see more in Remark~\ref{Silva} below.  (Recall that a linear map between Banach spaces is compact when it maps bounded subsets
 to relatively compact subsets, i.e. to subsets with compact closure.)

\subsection{First central extension}

We construct the determinant central extension of the smooth Lie group
$$
 1\lrto \C^*  \lrto \widetilde{\g}  \lrto \g \lrto 1  \, \mbox{.}
$$
Moreover, there is a smooth (non-group) section $\g  \to \widetilde{\g}$ of the map ${\widetilde{\g}  \to \g }$. Therefore this central extension can be described by a smooth group $2$-cocycle on $\g$ with coefficients in $\C^*$.

 In particularly, consider the smooth Lie subgroup $\g^0 \subset \g$ (which is the identity component)
$$
\g^0 =  \h^*_0  \rtimes \DS  \, \mbox{,}
$$
where $\h^*_0$ is the subgroup of the group $\h^*$ that consists of functions with zero winding number.
Then the determinant central extension restricted to the subgroup $\g^0$ is given by the following explicit smooth $2$-cocycle $D$:
$$
D(g_1, g_2) = \det\left((g_1)_{++} (g_2)_{++} ((g_1g_2)_{++})^{-1}\right)  \, \in \, \C^*  \, \mbox{,} \qquad g_1 \, \mbox{,} g_2 \in \g^0  \, \mbox{,}
$$
 and this formula is explained as follows. The group $\g$ naturally acts on the  locally convex topological $\C$-vector space $\h$ of analytic functions on $S^1$ with values in $\C$.
There is a natural decomposition $\h = \h_- \oplus \h_+$, where $\h_+$ is the space of functions from $\h$ that can be extended to holomorphic functions inside $S^1$ on $\C$,
and $\h_-$ is  the space of functions from $\h$ that can be extended to holomorphic functions outside $S^1$ on $\widehat{\C}$ and vanish at $z = \infty$. The spaces $\h_+$, $\h_-$, $\h$ are Silva spaces. For any $g \in \G$
consider the linear operator $g_+ = \mathop{\rm pr}_+ \circ \, g  |_{\h_+}: \h_+  \to \h_+ $, where $\mathop{\rm pr}_+$ is the projection from $\h$ to $\h_+$ with respect to the above decomposition. Then the determinant in the expression for the $2$-cocycle $D$ is well-defined since it is given by a series
$$
\det(1 + T) = \sum_{l \ge 0} \tr(\Lambda^l T)    \, \mbox{,}
$$
where the linear operator $T : \h_+ \to \h_+$ has the trace, see Section~\ref{big-trace} below.

The explicit expression for the $2$-cocycle $D$ is the same as the expression for the $2$-cocycle on the open subset of the group $GL_{\rm res}$, where the group $GL_{\rm res}$ is a special group which acts on the Hilbert space
$L^2(S^1, \C)$, see~\cite[Prop.~6.6.4]{PS}. The group ${\rm Diff}^+ (S^1)$ of orientation preserving smooth diffeomorphisms of $S^1$ is embedded into the group $GL_{\rm res}$, but this emebedding is not continuous (this emebdding induces the discrete topology on the group ${\rm Diff}^+ (S^1)$), see~\cite[\S~6.8]{PS}. Therefore this gives the difficulties to work with the central extension of the group ${\rm Diff}^+ (S^1)$ induced by the central extension of the group $GL_{\rm res}$ by the group $\C^*$, constructed in~\cite[\S~6.6]{PS}.

Therefore we develop an approach through the Lie group $\DS$ of analytic diffeomorphisms and the Lie group $\h^*$ of $\C^*$-valued analytic functions on $S^1$, which naturally act on the Silva space $\h$.
More generally, we construct also
groups ${\mathop{ \rm GL}}^0_{\rm res}(\h)$, ${\mathop{ \rm GL}}^+_{\rm res}(\h)$ and ${\mathop{ \rm GL}}^-_{\rm res}(\h) $ such that
the  group ${\mathop{ \rm GL}}^0_{\rm res}(\h)$ is a subgroup of groups ${\mathop{ \rm GL}}^+_{\rm res}(\h)$ and ${\mathop{ \rm GL}}^-_{\rm res}(\h) $, and the last two groups    are versions of the group $GL^0_{\rm res}$ (which is the identity component of the group $GL_{\rm res}$). But  the groups  ${\mathop{ \rm GL}}^+_{\rm res}(\h)$ and ${\mathop{ \rm GL}}^-_{\rm res}(\h) $ naturally act on the Silva space $\h$ (we don't use the Hilbert space $L^2(S^1, \C)$), see Section~\ref{group-gl} below. Then we construct the central extensions of these groups by the group $\C^*$ such that the determinant central extension of the group $\g^0$ is the pullback of these central extensions under the natural embedding of the group $\g^0$ into  the group ${\mathop{ \rm GL}}^0_{\rm res}(\h)$.

Another advantage of our approach is that we obtain the natural smooth $2$-cocycle $D$ on the group $\g^0$ with values in the group $\C^*$ which describes the determinant central extension.
For this goal we use the decomposition of elements of the group $\h_0^*$ into the product of ``$+$'' and ``$-$'' parts, and similarly use the decomposition for the elements of the group $\DS$, what is called the conformal welding and is based on~\cite{K},
see Sections~\ref{funct} and~\ref{weld}   and also Theorem~\ref{th-block}  below.

We have the decomposition into a semi-direct product $\g = \g^0 \rtimes \dz$, and the determinant central extension of $\g^0$ given by the $2$-cocycle $D$ can be uniquely extended to the central extension of the Lie group $\g$, which we also call the determinant central extension, see Section~\ref{ext-det} and Theorem~\ref{Th5} below. Besides, from the construction of the extended central extension it follows that there is a non-group smooth section $\g  \to \widetilde{\g}$
which extends the smooth section over $\g^0$ that was used to construct the $2$-cocycle $D$.

\subsection{Formal analog and linear operators with the trace}

There is the formal, purely algebraic,  analog of the determinant central extension of the Lie group $\g$, see~\cite{O1,O2}.
In this case the group $\h^*$ is replaced by the group of invertible functions on the formal punctured disk, and the group $\DS$ is replaced by the group of automorphisms of the formal punctured disc.

More exactly, for every commutative ring $A$ we consider the group of invertible elements $A((t))^*$ of the $A$-algebra of the  Laurent series $A((t))= A[[t]][t^{-1}]$, and the group ${\mathcal A}ut^{\rm c, alg}({\mathcal L})(A)$ of the continuous $A$-algebra automorphisms of $A((t))$, where we consider the $t$-adic topology on $A((t))$.   Now we have the following formal analog of the above Lie group $\g$:
$$
\g(A)= A((t))^* \rtimes {\mathcal A}ut^{\rm c, alg}({\mathcal L})(A)  \, \mbox{.}
$$
(More precisely, we have to speak on the functor $A \mapsto \g(A)$ from the category of commutative rings to the category of groups, and this functor is represented by the group ind-scheme, which is the formal analog of the Lie group $\g$.)

Now the construction of the determinant central extension of the Lie group $\g$ is analogous to the construction in the formal case. But in the formal case the formula for the analog of the group $2$-cocycle $D$ uses the determinant of the operator from $A[[t]]$ to $A[[t]]$ of type $1 + \widetilde{T}$, where there is $l > 0$ such that $\widetilde{T} |_{t^l A[[t]]} = 0$.

In our case of the smooth Lie group $\g$ we have in this place the absolutely convergent infinite series (see the formula for $\det(1 +T)$ in the previous section). Besides,
the space of continuous linear operators is given by a formula
$$
\Hom (\h_+, \h_+) = \varprojlim_{n \in \mathbb{N}} \varinjlim_{m \in \mathbb{N}}  \Hom(C_n, C_m)     \, \mbox{,} \qquad   \mbox{where} \quad \h_+= \varinjlim_{n \in \mathbb{N}} C_{n}
$$
as the Silva space with the Banach spaces $C_n$, see Remarks~\ref{Groth} and~\ref{Groth2} below.
Now the linear operators from a subspace
$$
{\Homb}  ({ \h_+}, { \h_+})  = \varinjlim_{m \in \mathbb{N}}  \varprojlim_{n \in \mathbb{N}}  \Hom(C_n, C_m)   \; \,
\subset \; \,
 {\Hom}  ({\h_+}, {\h_+})
$$
have the trace. (We just  interchanged  the order of projective and inductive limits $\varprojlim$
and  $\varinjlim$ in the definition of  ${\Homb}(\h_+, \h_+) $ with respect to $\Hom(\h_+, \h_+)$.)

More precisely,  the elements of ${\Homb}  ({ \h_+}, { \h_+})$ are in one-to-one correspondence with  the holomorphic functions $f(z,w)$ defined on an open subset $W_1 \times W_2  \subset  \C \times \widehat{\C}$, where the open set $W_1$ (which depends on the function $f(z,w)$) contains the unit closed disk and $W_2$ is the complement in $\widehat{\C}$ to the closed unit disk, such that $f(z,w)$ vanishes on $W_1 \times \infty$, see
Propositions~\ref{holom} and~\ref{tilde-holom} below.
 The linear operator from ${\Homb}  ({ \h_+}, { \h_+})$  is given as
$$
h(z)  \longmapsto  \oint_{| w | = 1+ \ve}   f(z, w) h(w) dw  \, \mbox{,}
$$
where $\ve > 0$ is small enough.
Then the trace of this linear operator is defined as
$
\oint_{| v | = 1 + \sigma } f (v,v)  dv
$, where $\sigma > 0$ is small enough, see Section~\ref{trace} below.

Note that the determinant central extension constructed from a group acting on a locally convex topological $\C$-vector space $H$, which is not a Hilbert space,  was considered also in~\cite[\S~1.II)]{ADKP}. But this differs from our case, because
$H = H_- \oplus H_+ $, where  $H_+$ is a Silva space and $H_-$ is not a Silva space, but it is topological (or continuous) dual to $H_+$. Besides, the group $\DS$ we are interested in does not act on $H$.

\subsection{Other central extensions}
Another central extensions of the Lie group $\g$ by the Lie group $\C^*$ can be obtained by the following explicit procedure.

Let $C^{\infty}(S^1, \C^*)$ be the group of smooth functions from
$S^1 $  to $\C^*$.
There is a bumultiplicative and antisymmetric pairing
$$
{\mathbb T} \, : \,   C^{\infty}(S^1, \C^*) \times C^{\infty}(S^1, \C^*) \lrto \C^*  \, \mbox{,}  \qquad
{\mathbb T} (f,g) = \exp \left( \frac{1}{2 \pi i} \int_{x_0}^{x_0}  \log f \,  \frac{dg}{g} \right) g(x_0)^{-\nu(f)}  \, \mbox{,}
$$
where the integer
$\nu(f) = \frac{1}{2 \pi i} \oint_{S^1} \frac{df}{f}$ is the winding number,
$x_0$ is any point on $S^1$, $\log f $ is any branch of the logarithm on $S^1 \setminus x_0$, and the integral is on $S^1$ from $x_0$ to $x_0$ in the counterclockwise direction.
The pairing $\mathbb T$ is invariant under the diagonal action of the group ${\mathop{\rm Diff}}^+(S^1)$ of orientation preserving smooth diffeomorphisms  of $S^1$ on  ${C^{\infty}(S^1, \C^*) \times C^{\infty}(S^1, \C^*)}$.

The pairing $\mathbb T$ was constructed by A.~A.~Beilinson and P.~Deligne, it is related with the $\cup$-product in Deligne cohomology, it factors through the Milnor $K_2$-group
$K_2^M(C^{\infty}(S^1, \C^*))$, see Sections~\ref{bimult} and~\ref{Del-coh}   below. There is the formal analog of pairing~$\mathbb T$, called the  Contou-Carr\`{e}re symbol, see~\cite[\S~2.9]{D2},   \cite{CC1}, \cite[\S~2]{OZ}.

The natural homomorphism of Lie groups $\G \to \DS$  gives the action of $\G$ on~$\h^*$. Let $\lambda_1$ and $\lambda_2$ be any group smooth $1$-cocycles on $\g$ with coefficients in the
$\g$-module $\h^*$. We construct the following smooth $2$-cocycle on $\g$ with coefficients in $\C^*$ (see more in Section~\ref{cup-sect}):
$$
\langle \lambda_1, \lambda_2   \rangle  = {\mathbb T} \circ (\lambda_1 \cup \lambda_2)  \, \mbox{.}
$$

 There are distinct smooth $1$-cocycles $\Lambda$ and $\Omega$ on $\g$ with coefficients in $\h^*$:
$$
\Lambda((d,f))= d   \qquad \mbox{and}  \qquad \Omega((d,f)) = (f^{\circ -1})'   \, \mbox{,}
$$
where $(d,f) \in \g = \h^* \rtimes \DS$,
$f^{\circ -1} $ denotes the diffeomorphism which  is   inverse to the diffeomorphism  $f$, and ${}'$  is the derivative with respect to the complex variable $z$.

Thus,  we have the following  smooth $2$-cocycles on the group $\g$ with coefficients in the group $\C^*$ (where $\g$ acts trivially on $\C^*$):
$$
\langle \Lambda, \Lambda  \rangle  \, \mbox{,} \qquad
\langle   \Lambda , \Omega                 \rangle  \, \mbox{,} \qquad
\langle   \Omega, \Omega                                \rangle     \, \mbox{.}
$$

By the standard procedure, each of these $2$-cocycles defines the central extension of the Lie group $\g$ by the Lie group $\C^*$ with the property that    this central extension admits a (non-group) smooth section from $\g$, see Remark~\ref{Rem-coh}.

\subsection{Comparison of central extensions and the local Deligne-Riemann-Roch isomorphism}
Consider the group $H_{\rm sm}^2(\g, \C^*)$ that classifies the equivalence classes of central extensions of the Lie group $\g$  by the Lie group $\C^*$ that admit smooth, in general non-group, sections from  $\g$. We will use the multiplicative notation for the group law in  $H_{\rm sm}^2(\g, \C^*)$.

In Theorem~\ref{Th-eq}   we prove that in the group $H_{\rm sm}^2(\g, \C^*)$  the class of
the $12$th power of the determinant central extension  is equal to the class of the  central extension given by the following $2$-cocycle:
$$
  \langle \Lambda, \Lambda \rangle^6 \cdot    \langle \Lambda, \Omega \rangle^{-6} \cdot \langle \Omega, \Omega \rangle  \, \mbox{.}
$$

The formal analog (after restriction to $\Q$-algebras $A$)  of Theorem~\ref{Th-eq} was proved in~\cite[Theorem~7]{O2}. This is the local Deligne--Riemann--Roch isomorphism for line bundles, where the original Deligne--Riemann--Roch isomorphism is stated as follows (see~\cite{D1}  and also explanations in~\cite[\S~1.1]{O2}).

Let $p : X \to S$ be  a  family of smooth projective curves over a scheme $S$ (i.e. a smooth proper morphism of relative dimension~$1$ between schemes $X$ and $S$) with connected geometric fibres.
Let $\omega = \Omega^1_{X/S}$ be an invertible sheaf of relative differential $1$-forms on $X$, and $\mathcal L$ be any invertible sheaf on $X$.
The Deligne--Riemann--Roch isomorphism is an isomorphism of invertible sheaves  on $S$:
$$
(\det R \, p_* {\mathcal L})^{\otimes 12} \simeq \langle {\mathcal L}, {\mathcal L } \rangle^{\otimes 6}  \otimes \langle {\mathcal L}, \omega  \rangle^{\otimes -6} \otimes \langle  \omega, \omega \rangle  \mbox{,}
$$
where for any invertible sheaves ${\mathcal L}_1$ and ${\mathcal L}_2$ on $X$ the Deligne bracket $\langle {\mathcal L}_1 , {\mathcal L}_2 \rangle$ is an invertible sheaf on $S$.
This isomorphism is functorial in $\mathcal L$ and compatible with base change.

\subsection{The topological  Riemann-Roch theorem}
The Theorem~\ref{Th-eq} on equivalence of central extensions of $\G$ by $\C^*$ has the following application.

Let $\pi : M \to B$ be a fibration in oriented circles (see more in Sections~\ref{Gysin} and~\ref{tRR}), where $M$ and $B$
 are finite-dimensional smooth manifolds.
 Let $L$ be a complex line bundle on  $M$.
Then $L$ can be considered as a locally trivial fibration over  $B$ with  transition functions with values in the group $\G$. These transition functions define a principal $\g$-bundle ${\mathcal P}_L$ over $B$.
This principal bundle and the determinant central extension of $\g$ define the determinant gerbe $ {\mathcal D}et (L)$  over $B$ which is the lifting gerbe, or the gerbe  of local lifts of the principal $\g$-bundle ${\mathcal P}_L$ to  principal $\widetilde{\G}$-bundles,
see more in Section~\ref{sec-last}. Denote the class of this gerbe in $H^3(B, \dz)$ by $[{\mathcal D}et (L) ]$. In particularly,  the element $[{\mathcal D}et (L) ]$ is an obstruction to find a principal $\widetilde{\G}$-bundle
$\widetilde{{\mathcal P}_L}$ over $B$ such that the principal $\g$-bundles $\widetilde{{\mathcal P}_L} / {\C}^*$ and ${\mathcal P}_L$  are isomorphic.

In Theorem~\ref{th-last} we prove  in the group $H^3(B, \dz)$  an equality
$$
12 \, [{\mathcal Det}(L)] = 6  \, \pi_* (c_1(L) \cup c_1(L))  \, \mbox{,}
$$
where $c_1(L) \in H^2(M, \dz)$ is the first Chern class of $L$, and $\pi_*  :  H^4(M, \dz)  \to H^3(B, \dz)$ is the Gysin map. In this equality $12 \, [{\mathcal Det}(L)]$ corresponds to the $12$th power of the determinant central extension of $\g$, $6  \, \pi_* (c_1(L) \cup c_1(L))$ corresponds to the central extension $\langle \Lambda, \Lambda \rangle^6$, and the central extensions $\langle \Lambda, \Omega \rangle^{-6}$ and $  \langle \Omega, \Omega \rangle$ give zero impact in $H^3(B, \dz)$, see the proof of Theorem~\ref{th-last}.

We call this equality  the topological Riemann-Roch theorem, since the fibres of $\pi$ are circles, i.e. $\pi$ is not  a morphism of complex manifolds, and besides, the equality is in $H^3(B, \dz)$.

Note that in the group $H^3(B, \C)$ the equality $2 \, [{\mathcal Det}(L)] =  \, \pi_* (c_1(L) \cup c_1(L))$ (that also follows from our equality using the map $H^3(B, \dz) \to H^3(B, \C)$) was proved by  P.~Bressler,  M.~Kapranov, B.~Tsygan and E.~Vasserot   in~\cite{BKTV}  by another methods.

\subsection{Organization of the paper}
The paper is organized as follows.

In Section~\ref{sec-two} we study spaces of analytic functions on a closed analytic curve $\Gamma$ in $\C$, explain that  these spaces are Silva spaces, and  describe the Grothendieck--K\"othe--Sebasti\~ao e Silva duality when $\Gamma = S^1$.

In Section~\ref{Expl_descr}  we consider the groups $\h^*$ and $\DS$ as  Lie groups modelled on Silva spaces. Besides, we prove various group decompositions for the group $\h^*$.

In Section~\ref{weld} we prove decompositions for elements of the group $\DS$ using conformal welding. We give also the application of this result to the action of elements of $\g^0$ on $\h$.

In Section~\ref{sect-cont}  we study the spaces of continuous linear operators $\Hom( {\mathcal V}_1, {\mathcal V}_2)$ and their subspaces $\Homb( {\mathcal V}_1, {\mathcal V}_2)$, where ${\mathcal V_1}$ and ${\mathcal V_2}$ belong to the set
  $\left\{\h_- ,  \h_+  \right\}$. We define also the trace  and the determinant for operators from $\Homb( \h_+, \h_+)$ and $1 + \Homb( \h_+, \h_+)$ correspondingly, and study the properties.

In Section~\ref{group-gl} we define the groups  ${\mathop{ \rm GL}}^+_{\rm res}(\h)$ and ${\mathop{ \rm GL}}^-_{\rm res}(\h)$ acting on $\h$,  and define their subgroup ${\mathop{ \rm GL}}^0_{\rm res}(\h)$.
Then we define the determinant central extensions of these groups.

In Section~\ref{det-centr-g}   we prove that $\G^0$ is a subgroup of the group ${\mathop{ \rm GL}}^0_{\rm res}(\h)$ and define the determinant central extension of $\G^0$ as the pullback of the determinant central extension of ${\mathop{ \rm GL}}^0_{\rm res}(\h)$. We construct the explicit smooth  $2$-cocycle $D$ for the determinant central extension of $\G^0$. We study the Lie algebra  $\mathop{\rm Lie} \g^0$ of the Lie group $\g$ and the complex Lie algebra $\Lie_{\C} \g^0$. This leads to the construction of a unique extension of the determinant central extension from the Lie subgroup $\g^0$ to the Lie group $\g$.

In Section~\ref{expl-2-coc} we describe the pairing $\mathbb T$ and then apply it to the construction of the explicit smooth  $2$-cocycles
$
\langle \Lambda, \Lambda  \rangle
$,
$
\langle   \Lambda , \Omega                 \rangle  $,
$
\langle   \Omega, \Omega                                \rangle
$ on the group $\g$ with coefficients in the group $\C^*$.

In Section~\ref{dec-th}   we decompose  in the group $H_{\rm sm}^2(\g, \C^*)$
the class of the $12$th power of the determinant central extension
into the product on integer powers of the explicit $2$-cocycles constructed in Section~\ref{expl-2-coc}. For this goal we use the corresponding decomposition for the Lie algebra $2$-cocycles.

In Section~\ref{DelGys}   we recall the Deligne cohomology on smooth manifolds. We also prove Theorem~\ref{th-9} that relates the Gysin map for a  fibration  in oriented circles applied to the $\cup$-product in singular cohomology with integer coefficients and
the map (or pairing) $\mathbb T$.

In Section~\ref{sec-last} we relate the group cohomology, the \v{C}ech cohomology and  gerbes. We prove also the topological Riemann-Roch theorem for complex line bundles on  fibrations in oriented circles.


\bigskip

\noindent {\bf \large Acknowledgments}

\medskip

I am grateful to the referee of this paper,  who pointed out to me some references, see Remark~\ref{referee} below.

\section{Analytic functions on a closed curve in $\sdc$}
\label{sec-two}

\subsection{A closed curve $\Gamma$ in $\C$}
\label{curve}

Let $\widehat{\C}= \mathbb{C}  \cup \infty = \C P^1 $ be the Riemann sphere (or, in other words, the extended complex plane).

For any domain $W \subset \widehat{\C}$ denote by ${\h{\left(W \right)} = H^0(W, \oo)}$ the $\C$-vector space of holomorphic functions on $W$.
If $\infty \in W$, then by $\h_0(W)  \subset \h(W)$ denote the $\C$-vector subspace of functions that vanish at $\infty$.

Let  $S^1 = \{z \in \mathbb{C} \, \mid  \,  |z|  =1   \}$ be the unit circle in  $\C$.

Let $\Gamma$ be a closed curve in $\C$ such that $\Gamma$ is the image of $S^1$ under a holomorphic univalent map defined in a neighbourhood of $S^1$.

By $\h(\Gamma)$ {\em denote} the $\C$-vector space of  $\C$-valued functions on $\Gamma$ that can be extended to holomorphic functions in a neighbourhood of $\Gamma$ (this neighbourhood depends on a function).

The closed curve $\Gamma$ defines the disjoint union
\begin{equation}  \label{disj_union}
\widehat{\C} = W_{\Gamma,+} \sqcup  \Gamma \sqcup W_{\Gamma,-}  \, \mbox{,}
\end{equation}
where the connected domain $W_{\Gamma,+}$ is in  $\C$, and the connected domain $W_{\Gamma,-}$ contains~$\infty$.

Let $\h_+(\Gamma) \subset \h(\Gamma)$  be the $\C$-vector subspace of  functions that can be extended to
holomorphic functions on $W_{\Gamma,+}$. {(Recall that any holomorphic function on a connected domain that contains $\Gamma$ is uniquely defined by its restriction to $\Gamma$.)}

Let
$\h_-(\Gamma) \subset \h(\Gamma)$ be the $\C$-vector subspace of  functions that can be extended to holomorphic functions on $W_{\Gamma,-}$ that vanish at $z= \infty$.

\begin{prop}  \label{dec-gamma}
There is a canonical decomposition
\begin{equation}  \label{hol-dec}
\h(\Gamma) = \h_-(\Gamma) \oplus \h_+(\Gamma)  \, \mbox{.}
\end{equation}
\end{prop}
\begin{proof}
Consider an open covering $\widehat{\C} = U \cup V $, where the Riemann surfaces $U$, $V$ and $U \cap V$ are connected,  $\Gamma \subset U \cap V$,   $U \subset \C$ and $\infty \in V$.

Recall that any connected noncompact Riemann surface $W$ is a Stein manifold, and hence   $H^1(W, \oo)=0$. Therefore $H^1(U, \oo)= H^1(V, \oo) =0$.
Therefore the \v{C}ech complex  for this covering gives an exact sequence:
$$
0 \lrto H^0(\widehat{\C}, \oo)  \lrto H^0(V, \oo)  \oplus H^0(U, \oo)  \lrto H^0(V \cap U, \oo)  \lrto H^1(\widehat{\C}, \oo) \lrto 0 \, \mbox{.}
$$
Since $H^0(\widehat{\C}, \oo)= \C$ and $H^1(\widehat{\C}, \oo)=0$, this sequence gives a decomposition
\begin{equation}  \label{hol-dec-wl}
\h(V \cap U)= \h_0(V) \oplus \h(U)  \, \mbox{.}
\end{equation}

Taking the inductive limit of~\eqref{hol-dec-wl}  by considering smaller and smaller $V \cap U  \supset \Gamma$, we obtain~\eqref{hol-dec}.
\end{proof}

\subsection{When $\Gamma$ is the unit circle}  \label{unit_circle}

In case of $\Gamma = S^1$, using that $\dr/\dz  \simeq S^1$ via the map $ x \mapsto \exp(2 \pi i x)$, the space $\h \left(S^1 \right)$ can be identified with the space of $\C$-valued analytic functions on $\dr$ with period $1$.
We will use also {\em the following notation}
$$
\h = \h \left(S_1 \right) \, \mbox{,} \qquad \h_-= \h_- \left(S^1 \right)  \, \mbox{,}  \qquad \h_+= \h_+ \left(S^1 \right)  \, \mbox{.}
$$

For any real number  $r > 0 $ let the open disk $D_r =\{z \in \C \mid  |z| <r   \} $, the closed disk $\overline{D}_r$ be the closure of the open disk $D_r$, and for any real number $0 < r < 1$ let the annulus $$A_r= \{ z \in \C \mid  1-r < |z| < 1+ r \} \, \mbox{.}$$

We note that
\begin{equation}  \label{indlim}
\h = \varinjlim_{n \in \N} \tilde{\h} \left(A_{1/n} \right) \, \mbox{,} \qquad  \h_-=  \varinjlim_{n \in \N} \tilde{\h}_0 \left(\widehat{\C} \setminus \overline{D}_{1 - 1/n} \right)  \, \mbox{,}
\qquad
\h_+ = \varinjlim_{n \in \N} \tilde{\h} \left( D_{1 + 1/n } \right)   \, \mbox{,}
\end{equation}
where $\tilde{}$ in formula~\eqref{indlim} means that in the corresponding spaces of holomorphic functions $\h(\cdot)$ or $\h_0(\cdot)$  we take
$\C$-vector subspaces of functions that can be extended to continuous functions  on the closure of the domain.

In formula~\eqref{indlim} every vector space under the limit is a Banach space with the supremum  norm  (by the maximum principle it is enough to consider only the supremum  norm on the boundary of the domain).

Then we take the inductive limit topology on $\h$, $\h_+$ and $\h_-$ in the category of locally convex topological vector spaces. This topology makes each of these  spaces into a so-called Silva space (see, e.~g., \cite[\S~I.1]{Neeb1}, \cite[Chapter~1, \S~5]{Mor}  and Remark~\ref{Silva} below with more explanations on Silva spaces).

We note that decomposition~\eqref{hol-dec} in case of $\Gamma = S^1$ is
\begin{equation}  \label{H-decomp}
\h = \h_- \oplus \h_+  \, \mbox{.}
\end{equation}
Other way to obtain  decomposition~\eqref{H-decomp} is to take the  minus and nonnegative part (with respect to powers of $z$) of the Laurent series decomposition.
From  Cauchy's integral formulas for coefficients of the Laurent series we have the estimates for these coefficients, and hence  it is easy to see that the topology on $\h$ described above is the product topology from $\h_-$, $\h_+$ with respect to decomposition~\eqref{H-decomp}.

\begin{nt}  \label{Silva}  \em
Recall that,  {\em by definition},  a locally convex topological vector space  $F$ is called a Silva space if $F$ is  an inductive limit  of a sequence of Banach
spaces $F_{n}$, where $n \in \mathbb{N}$, in the category of locally convex topological vector spaces:
$$
F = \varinjlim_{n \in \mathbb{N}} F_{n} \, \mbox{,}
$$
and  where all the linking linear maps $F_n \to F_{n+1}$ are  compact inclusions, see, e.~g.,~\cite[\S~I.1]{Neeb1} and especially survey in~\cite[Appendix~A, \S~5]{Mor} (but where the Silva spaces  are called $DFS$ spaces).

We point out now  some properties of  Silva spaces, see~\cite[Appendix~A, \S~5]{Mor} (and, except for the  last property, see also~\cite[\S~25, \S~26]{FW}  where a little bit more general situation  is considered).

\begin{itemize}
\item A Silva space  is  a complete Hausdorff locally convex  topological vector space.

\item A sequence $\{x_l\}$ of elements from a Silva space $F$
 converges to an element $x$
 from $F$
 if and only if there exists $n \in \mathbb{N}$ such that
 all the elements $x_l$, $x$  belongs to  $F_{n}$
 and $\{ x_l \}$
 converges to $x$
 in $F_{n}$.

  Moreover, a subset $K$ of the Silva space $F$  is bounded if and only if  there exists $n \in \mathbb{N}$ such that $K$ is a bounded subset of  $ F_n$.

\item Any closed vector subspace of a Silva space  is again a Silva space.

\item If $\vartheta $ is a continuous bijective linear map between Silva spaces, then $\vartheta^{-1}$ is continuous. (Indeed, the graph of $\vartheta$ is closed. Therefore the graph of $\vartheta^{-1}$ is closed. Hence by the closed graph theorem, see~\cite[Corollary~A.6.4]{Mor}, the map $\vartheta^{-1}$  is continuous. Note that the closed graph theorem follows also in our case  from the result of A.~Grothendieck, see~\cite[Introduction~IV.4]{Gro0}.)
\end{itemize}

\end{nt}

\subsection{Grothendieck--K\"othe--Sebasti\~ao e Silva duality}  \label{Grot-dual}

The $\C$-vector space $\h(D_1)$ carries a natural structure of a  Fr\'echet space, and the topology is defined by the countable system of norms $\parallel  \cdot   \parallel_n$, where for any integer $n > 0 $
$$
\parallel  f   \parallel_n  = \max_{|z| = 1- 1/n} \left|f(z) \right|  \, \mbox{.}
$$
This topology coincides with the topology of uniform convergence on compact sets.

The Grothendieck--K\"othe--Sebasti\~ao e Silva duality (see~\cite{dual} and references therein,   \cite[Chapter~2, \S~1]{Mor}, \cite{LR} and Remark~\ref{dual-top} below) gives that the complete locally convex topological vector spaces $\h(D_1)$
and $\h_-$
are dual to each other, i.e., each of them is topologically isomorphic to the continuous dual of other,  where the continuous  dual topological vector  space consists of the continuous linear functionals and has the strong topology. This duality is given by the following pairing
\begin{equation}  \label{exp-pair}
( f, \phi ) =  \oint_{|z| =1 - \gamma} f(z) \phi(z)  dz  \, \mbox{,}
\end{equation}
for $f \in \h \left(D_1 \right)$, $\phi \in \tilde{\h}_0 \left(\widehat{\C} \setminus \overline{D}_{1 - 1/n } \right)$  (see formula~\eqref{indlim}), and where $\gamma $ is any real number such that ${0 < \gamma < 1/n}$. The pairing does not depend on the choice of such $\gamma$.

By considering the change of variable $z \mapsto z^{-1}$ on $\widehat{\C}$,  the analogous duality holds
between the Fr\'echet space $\h_0(\widehat{\C} \setminus \overline{D}_1)$
and the Silva space $\h_+$. Besides, since
\begin{equation}  \label{change_z}
f(z^{-1}) \phi(z^{-1})  d(z^{-1})  = - (z^{-1}f(z^{-1})) (z^{-1}\phi(z^{-1})) dz  \, \mbox{,}
\end{equation}
 the  pairing~\eqref{exp-pair} goes to the  similar pairing, where we have to consider the integral over the curve $|z|= (1 - \gamma)^{-1}$. (The minus
 in formula~\eqref{change_z}   will disappear in the integral, since the change of variable $z \mapsto z^{-1}$ sends the  anticlockwise  direction to the clockwise direction of the integral contour.)

\begin{nt} \em \label{dual-top}
More generally, in the category of locally convex topological vector spaces we have (see, e.g., \cite[Chapter~1, \S~4]{Mor})
$$
\h(D_1)  = \varprojlim_{n \in \mathbb{N}} \tilde{\h} \left( D_{1 - 1/n } \right)  \, \mbox{.}
$$
And the linking linear maps $\tilde{\h}\left( D_{1 - 1/(n+1) } \right)  \to \tilde{\h}\left( D_{1 - 1/n }\right)$ between the  Banach spaces are compact. This is an example of
$FS$ space (see~\cite[Appendix A, \S~4]{Mor}) which is, by definition,  a countable projective limit (in the category of locally convex topological vector spaces) of Banach spaces with compact linear linking maps. And $FS$-spaces and Silva spaces are strongly dual to each other, see~\cite[Appendix A, \S~6]{Mor} or \cite[\S~26]{FW}. Pairing~\eqref{exp-pair} gives an explicit example of such duality.
\end{nt}

\section{Functions from $S^1$ to $\C^*$ and diffeomorphisms of~$S^1$}  \label{Expl_descr}

\subsection{Lie groups of functions from $S^1$ to $\C^*$}   \label{funct}
Let $C^{\infty}(S^1, \C^*)$ be the group of all smooth functions from  $S^1 $  to $\C^*$, where $\C^* = \C \setminus 0$ and we consider $S^1$ in $\C$ as in Section~\ref{curve}.

Recall that for any
$g$ from $C^{\infty}(S^1, \C^*)$
  the winding number $\nu(g)$ is
\begin{equation}\label{wind}
\nu(g) = \frac{1}{2 \pi i} \oint_{S^1} \frac{dg}{g}   \,\,  \in  \,\,  \Z  \mbox{.}
\end{equation}
The integer $\nu(g)$ equals also the degree of the map $\mathop{\rm arg} (g) $ from $S^1$ to $S^1$.

The map $\nu$ is a homomorphism from the group $C^{\infty}(S^1, \C^*)$ to the group  $\dz$.

\bigskip

By $\h^* $ {\em denote} the group of invertible elements of the ring $\h = \h(S^1)$. The group $\h^*$ consists of functions from the space $\h$ that take the values in $\C^*$.

By $\h^*_0 $  {\em denote} the subgroup of the group $\h^*$ that consists of functions with zero winding number.

By $\h_{+,1}^*$ {\em denote} the subgroup of the group $\h^*_0$ that consists of functions that can be extended to holomorphic functions
on $W_{S^1,+}$ (see formula~\eqref{disj_union}) taking the values in $\C^*$ and equal to $1$ at $z =0$.

By $\h_{-,1}^*$ {\em denote} the subgroup of the group $\h^*_0$ that consists of functions that can be extended to holomorphic functions
on $W_{S^1,-}$ (see formula~\eqref{disj_union}) taking the values in $\C^*$ and equal to $1$ at $z =\infty$.

Since $\nu(z) =1$, for any  $f \in \h^* $ we have $f= z^{\nu(f)} \cdot \left( f z^{-\nu(f)}  \right)$. This gives a canonical group  decomposition
\begin{equation}  \label{dec1}
\h^* = \Z \times \h^*_0  \, \mbox{.}
\end{equation}

\begin{prop}  \label{dec-funct}
There are  canonical group decompositions
\begin{equation}  \label{dec2}
\h^*_0 = \C^* \times \h_{-,1}^* \times \h_{+,1}^*  \qquad \mbox{and}  \qquad
\h^* = \Z \times \C^* \times \h_{-,1}^* \times \h_{+,1}^*  \, \mbox{.}
\end{equation}
\end{prop}
\begin{proof}
Because of formula~\eqref{dec1}, it is enough to obtain the first decomposition.

If $f \in \h^*$ and $\nu(f) =0$, then there is a branch $\log f \in \h$. Therefore by formula~\eqref{H-decomp} we have the unique decomposition
$\log (f) = g_- + g_+$. Hence we have the unique decomposition
$$
f = \exp(g_+)(0) \cdot \left(\exp(g_+) \exp(g_+)(0)^{-1} \right)  \cdot \exp(g_-)  \, \mbox{,}
$$
where $\exp(g_+)(0)  \in \C^*$, $\exp(g_+) \exp(g_+)(0)^{-1}   \in \h_{+,1}^* $ and $\exp(g_-) \in \h_{-,1}^* $.
\end{proof}

\bigskip

The homomorphism $\exp$ gives the isomorphism between $\h_-$ and $\h_{-,1}^*$, and also between the closed $\C$-vector subspace of $\h_+$ (consisting of functions that vanish at $z=0$) and $\h_{+,1}^*$.

Using these isomorphisms, the direct product topology in decompositions~\eqref{dec2} and the discrete topology on the group $\Z$, we obtain the structure of infinite-dimensional Lie group on each of the  groups:  $\h_{-,1}^*$,
$\h_{+,1}^*$, $\h_0^*$, $\h^*$. These Lie groups are modelled on locally convex topological vector spaces, or,  more exactly, on Silva spaces. (See the recent survey on  Lie groups modelled on locally convex topological vector spaces  in~\cite{Neeb1}.)

Besides, since any topological vector space is contractible,  we have
$$
\pi_1 \left(\h_{-,1}^* \times \h_{+,1}^* \right) = \{0\} \, \mbox{,} \qquad
\pi_1 \left(\h^* \right)= \pi_1 \left(\h_0^* \right) = \pi_1 \left(\C^* \right) = \Z  \,  \mbox{.}
$$

\subsection{Lie group of analytic diffeomorphisms of $S^1$}   \label{anal_diff}

By $\DS$ denote the group of analytic diffeomorphisms of $S^1$ that preserve the orientation. (Note that sometimes the group of real analytic diffeomorphisms of a real analytic manifold $\Upsilon$ is denoted by $\mathop{\rm Diff}^{\omega}(\Upsilon) $, but we will not use this notation.)

Recall that we consider $S^1$ as the unit circle in $\C$.
It is easy to see that a bijective function ${f : S^1 \to S^1}$ belongs to $\DS$ if and only if
$f$ can be extended to the univalent holomorphic function defined in a neighbourhood of $S^1$ (this extension is always unique). Further in the article we will use {\em this description} of  $\DS$.

We will describe the structure of infinite-dimensional Lie group on $\DS$. This Lie group is modelled  on locally convex topological vector spaces, or,  more exactly, on Silva spaces.

We give the description of the universal covering group of the group  $\DS$ and its topological structure. This description is similar to the well-known description of the group of smooth diffeomorphisms of $S^1$, but in case of group of analytic diffeomorphisms the Silva spaces appear.

Consider the group  $\WDS$ that consists of real analytic functions $\phi \, : \, \dr \to \dr$ such that $\phi(x+1)= \phi(x) +1$ and $\phi'(x) > 0$ for any $x \in \dr$, and the group structure is given by the composition of functions.  Using that $\dr/\dz \simeq S^1$ via $x \mapsto \exp(2 \pi i x)$, we obtain the central extension of groups
\begin{equation}  \label{z-ext}
0 \lrto \dz \stackrel{\varrho}{\lrto} \WDS  \stackrel{\varsigma}{\lrto} \DS  \lrto 1  \, \mbox{,}
\end{equation}
where $\varrho(n)(x) = x +n $ for $n \in \dz$, $x \in \dr$, and $\varsigma(\phi)= \exp(2 \pi i \phi)$ for $\phi \in \WDS$.

The map $\phi \mapsto \phi -x$ gives the isomorphism between the set $\WDS$ and the set $\mathcal L$ of periodic real analytic functions $\varphi \, : \, \dr \to \dr$ with period $1$ and the property $\varphi'(x) > -1$ for any $x \in \dr$.
The set $\mathcal L$ is a subset in the $\dr$-vector space $\h_{\dr}$ of real analytic functions from $S^1 \simeq \dr/\dz$ to $\dr$. It is easy to see that $\h_{\dr}$ is a closed $\dr$-vector subspace of $\h = \h(S^1)$. Clearly, a closed vector subspace of a Silva space is again a Silva space (see also Remark~\ref{Silva}). Therefore the $\dr$-vector space  $\h_{\dr}$ is a Silva space.

Now the condition $\max_{x \in \dr} \varphi'(x) > -1$ gives the open subset in the Silva space  $\h_{\dr}$. Indeed,  let $\tilde{\varphi}(z) = \varphi(x)$, where $z = \exp(2 \pi i x)$.  Then
$$
\frac{d}{ dx } {\varphi} = 2 \pi i z \frac{d}{d z} \tilde{\varphi} \, \mbox{.}
$$
Therefore the condition $\min_{x \in \dr} \varphi'(x) > -1$ is equivalent to the condition
$$
\quash{\min_{z \in S^1} 2 \pi i z \frac{d}{dz} \tilde{\varphi}(z)  =}
\min_{z \in S^1} \mathop{\rm Re} \left( 2 \pi i z \frac{d}{dz} \tilde{\varphi}(z) \right)  \,   >  \,
-1  \, \mbox{.}
$$
 The last condition  defines the open subset in $\h$ and correspondingly in $\h_{\dr}$, since an upper estimate of the supremum norm of  a holomorphic function  on the closure of an open neighbourhood of $S^1$  entails an upper estimate of the supremum norm on $S^1$ of the derivative of the function due to Cauchy's integral formula for the derivative.

 This description of $\mathcal L$ as the open subset of the Silva space $\h_{\dr}$ gives the structure of smooth infinite-dimensional Lie group on $\WDS$ (to see that it is a smooth group we have to use the Taylor formula and Caushy's estimates for higher derivatives). Therefore, using central extension~\eqref{z-ext}, we obtain the structure of smooth infinite-dimensional Lie group (modelled on Silva spaces) on~$\DS$.

The topological space $\mathcal L$ is contractible via the homotopy $(\varphi, \lambda)  \mapsto \lambda \varphi$, where $\varphi \in \mathcal L$, $\lambda \in [0,1]$, between the identity map and the constant map to the element $0 \in {\mathcal L}$.

Hence, $\pi_1 \left(\WDS \right) = \{ 1 \}$, and by central extenion~\eqref{z-ext}, $\WDS$ is the universal covering group of $\DS$. Besides, the natural embedding of the unitary group ${\mathrm U}(1)$ to the group  $\DS$ gives the isomorphism
$$
\dz \, \simeq \, \pi_1 \left({\mathrm U}(1) \right) \, \simeq \, \pi_1 \left(\DS \right)  \, \mbox{.}
$$

\section{Conformal welding and block matrix}
\label{weld}
\subsection{Conformal welding}  \label{Conf-weld}
Here we describe the conformal welding for elements of $\DS$ (see also~\cite[\S~5.1]{AST} which is based on the case of smooth diffeomorphisms given in~\cite[\S~2]{K}).

For elements of~$\DS$ we describe decompositions that are similar to decompositions for elements of $\h_0^*$ given in  Proposition~\ref{dec-funct}.

Let a function $f$ belongs to the group $\DS$. Then there is a real number $r$ such that ${ 0 < r < 1}$ and the function $f$ is a univalent holomorphic function on an open neighbourhood of the closure of the  annulus $A_{r}$ (see notation in Section~\ref{unit_circle}). Consider the bounded open domain $ U$  in $\C$ such that its boundary $\overline{U} \setminus U$  is the image of the circle $\{z \in \mathbb{C} \, \mid  \,  |z|  =1 -r  \}$
 under the map~$f$.  Taking the smaller $r$ if necessary, we  suppose that $0 \in U$ and $\infty \in \widehat{\C} \setminus \overline{U}$.
Here $\overline{{U}}$ is the closure of ${U}$.

We glue the open disk $D_{1+r}$ with the open domain $\widehat{\C} \setminus \overline{{U}}$ by identification of the annulus $A_r$ with its image  under the map $f$. We obtain the compact Riemann surface that is biholomorphic to $\widehat{\C}$. (Indeed,
the $2$-sphere $S^2$ has a unique Riemann surface structure, since by the uniformization theorem, any simply connected Riemann surface is biholomorphic either to the open unit disk, or to $\C$, or to $\widehat{\C}$.)

Therefore there are univalent holomorphic maps $f_+ \, : \, D_{1+r}  \to \C$ and $f_-  \, : \,  \widehat{\C} \setminus \overline{U}  \to \widehat{\C}$  such that $f_+ = f_- \circ f$ on the annulus $A_r$,
where $\circ$ means the composition of functions.
Since
$$\mathop{\rm Aut} \widehat{\C} = \mathop{\rm Aut} \C P^1 = \mathop{\rm PSL}(2, \C)  \, \mbox{,}$$
we can demand
$f_+(0)=0$, $f_-(\infty) = \infty$, and  also  $ (f_+)'(0)=1 $ (or
 $(f_-)'(\infty) =1$ instead of the last condition). {\em Here and further in the article we will use notation $'$ for $\frac{d}{dz}$ applied to a holomorphic function on a domain in $\widehat{\C}$, and in particularly, applied to elements of $\DS$.}

 Thus, we have the following proposition, which we will use in Section~\ref{App}.

 \begin{prop}  \label{diff-decomp}
 For any $f \in \DS$ there are univalent holomorphic maps $f_+$  from an open neighbourhood of $\overline{D}_1$ to $\C$ and $f_-$ from an open  neighbourhood
 of $\widehat{\C} \setminus D_1 $ (in $\widehat{\C}$) to $\widehat{\C}$  such that
 \begin{equation} \label{de-diff}
   f_+ = f_- \circ f  \quad \mbox{on} \quad S^1  \, \mbox{,} \qquad f_+(0) =0  \, \mbox{,}  \qquad f_-(\infty) = \infty \, \mbox{.}
 \end{equation}
 Moreover, it is possible to demand $ (f_+)'(0)=1 $ (or
 $(f_-)'(\infty) =1$). Then, after this additional condition,  decomposition~\eqref{de-diff} will be  unique.
 \end{prop}
In particularly, we have $f_+(S^1) = f_-(S^1) \subset \C$, and $z =0$ is inside the open domain whose boundary is $f_+(S^1)$.

\subsection{Applications and block matrix}  \label{App}

We will give now applications of conformal welding described in Section~\ref{Conf-weld}, which we will need further.

\begin{prop}  \label{wind-diff}
For any $f \in \DS$ the winding number
$$
\nu \left(  f' \right) = 0  \, \mbox{.}
$$
\end{prop}
\begin{proof}
By Proposition~\ref{diff-decomp} we have
$$
f_+ = f_- \circ f  \, \mbox{.}
$$
Hence we have
$$
  (f_+)' = \left( (f_-)'   \circ f \right) \cdot   f'  \, \mbox{.}
$$
Therefore
$$
\nu \left((f_+)' \right) = \nu \left( (f_-)'   \circ f \right) + \nu \left(   f'  \right) = \nu \left( (f_-)' \right) +  \nu \left(   f'  \right) \, \mbox{.}
$$
Now $\nu \left((f_+)' \right) =0$, since in the definition given by formula~\eqref{wind} we can contract the integral contour to the point $z=0$. Analogously, $\nu \left( (f_-)' \right) =0$, since in the definition given by formula~\eqref{wind} we can contract the integral contour to the point $z = \infty$.
\end{proof}

\bigskip

The Lie group $\DS$ acts smoothly  on the $\C$-vector space (the Silva space) $\h = \h(S^1)$ by the rule $f \diamond h = h \circ f^{\circ -1}$, where $h \in \h$, $f \in \DS$ and  $f^{\circ -1} $ denotes the diffeomorphism which  is   inverse to the diffeomorphism  $f$. It gives the smooth action of the Lie group $\DS$ on the Lie group $\h^*$, and this action preserves the Lie subgroup $\h_0^*$.

\begin{defin}  \label{d1}
Define a Lie group
$$
\G = \h^*  \rtimes \DS     \qquad \mbox{and its Lie subgroup} \qquad \G^0 = \h_0^*  \rtimes \DS   \, \mbox{.}
$$
\end{defin}

So, the group $\g$ consists of pairs $(d,f)$, where $d \in \h^*$, $f \in \DS$, with the multiplication law
$$
(d_1, f_1) (d_2, f_2) = (d_1 \cdot (f_1 \diamond d_2), f_1  \circ f_2)  \, \mbox{.}
$$

The Lie group $\G$ acts smoothly on the $\C$-vector space (the Silva space) $\h$ by the rule $(d,f) \diamond h = d \cdot (f \diamond h)$, where  $d \in \h^*$, $f \in \DS$, $h \in \h$.

For any element $g  \in \G$ consider the block matrix
\begin{equation}  \label{block-mat}
\begin{pmatrix}
  g_{--} & g_{+-} \\
  g_{-+} & g_{++}
\end{pmatrix}
\end{equation}
for the action of $g$ on $\h$ with respect to the decomposition $\h = \h_- \oplus \h_+$ (see formula~\eqref{H-decomp}), where the linear operators
$$
g_{--} : \h_- \lrto \h_-  \, \mbox{,} \qquad g_{++} : \h_+ \lrto \h_+  \, \mbox{,}  \qquad g_{+ -} : \h_+ \lrto \h_-  \, \mbox{,}  \qquad
g_{-+} : \h_- \lrto \h_+ \, \mbox{.}
$$

\medskip

\begin{defin} \label{def-1}
For any Silva space $\mathcal V$ over the field $\C$ by $\mathop{ \rm GL}({\mathcal V})$ denote the group of all continuous  $\C$-linear automorphisms of $\mathcal V$. (Note that by Remark~\ref{Silva} it is enough to consider only bijective continuous linear maps $g : {\mathcal V} \to {\mathcal V}$, then $g^{-1}$ will be also continuous.)
\end{defin}

\medskip

\begin{Th}  \label{th-block}
For any $g \in \G^0$ the linear operators $g_{++}$ and  $g_{--}$  belongs to the groups $\mathop{ \rm GL}(\h_+)$ and $\mathop{ \rm GL}(\h_-)$ correspondingly.
\end{Th}
\begin{proof}
First, we prove the statement of the theorem about $g_{++}$.

Let $g = (d , f)$,
where $d  \in \h_0^*$ and $f \in \DS$.

By Proposition~\ref{dec-funct}, we have $d = d_- d_+ $, where $d_{-} \in \C^* \times \h_{-,1}^*$  and $d_+ \in \h_{+,1}^*$.
With respect to the decomposition $\h = \h_- \oplus \h_+$ the element $d_-$ acts on $\h$ via the block matrix  of type
$
\left(
\begin{matrix}
  * & * \\
  0 & *
\end{matrix}
\right)
$.
And the linear operators on the diagonal in this block matrix are invertible.

By Proposition~\ref{diff-decomp}, we have $f = (f_-)^{\circ -1}   \circ f_+$, where $(f_-)^{\circ -1} $ means the inverse of $f_-$. Besides, $f_+(S^1) = f_-(S^1) = \Gamma\subset \C$.
With respect to the decompositions $$\h = \h_- \oplus \h_+  \qquad   \mbox{and}  \qquad \h(\Gamma) = \h_- (\Gamma) \oplus \h_+ (\Gamma)$$
 (see Proposition~\ref{dec-gamma})
the element $(f_-)^{\circ -1}$ acts from $\h(\Gamma)$  to  $\h$  (by the rule $s \mapsto s \circ f_-$) via the block matrix  of type
$
\left(
\begin{matrix}
  * & * \\
  0 & *
\end{matrix}
\right)
$, and the element $f_+$ acts from $\h$  to $\h(\Gamma)$  (by the rule $h \mapsto h \circ (f_+)^{\circ -1}$) via the block matrix  of type~$
\left(
\begin{matrix}
  * & 0 \\
  * & *
\end{matrix}
\right)
$.
And the linear operators on the diagonal in these block matrices are invertible.

Now  consider the element $w =d_+ \circ (f_-)^{\circ -1}$ from the group of invertible elements of the ring $\h(\Gamma)$. Since the winding number $\nu(d_+) =0$, there is a branch
$\log w$, which belongs to  $\h(\Gamma)$. By Proposition~\ref{dec-gamma}   we have decomposition $\log w = r_- + r_+$. Now $w =   w_- w_+ $, where $w_+ = \exp(r_+)$ is the holomorphic invertible function inside of $\Gamma$ (and in the neighborhood of $\Gamma$), and $w_- =\exp(r_-)$ is
the holomorphic invertible function outside of $\Gamma$ (and in the neighborhood of $\Gamma$) in $\widehat{\C}$, compare with the proof of Proposition~\ref{dec-funct}.

With respect to the decomposition $\h(\Gamma) = \h_-(\Gamma) \oplus \h_+(\Gamma)$ the element $w_-$ acts on $\h(\Gamma)$ (by multiplication) via the block matrix  of type
$
\left(
\begin{matrix}
  * & * \\
  0 & *
\end{matrix}
\right)
$, and the element $w_+$ acts on $\h(\Gamma)$ (by multiplication) via the block matrix  of type~$
\left(
\begin{matrix}
  * & 0 \\
  * & *
\end{matrix}
\right)
$.
And the linear operators on the diagonal in these block matrices are invertible.

It is easy to see that the action of element $g$ defines the operator, which is the composition of the following linear operators:
\quash{$$
\h  \xlongrightarrow{f_+} \h(\Gamma)  \xlrto{\times w_+}  \h(\Gamma)  \xlrto{\times w_-}  \h(\Gamma)  \xlrto{(f_-)^{\circ -1}}  \h  \xlrto{\times d_-}  \h
$$}
\begin{equation} \label{arr}
\xymatrix{
\h  \ar[rr]^{f_+} && \h(\Gamma)  \ar[rr]^{\times w_+} &&  \h(\Gamma)  \ar[rr]^{\times w_-} && \h(\Gamma)  \ar[rr]^{(f_-)^{\circ -1}} && \h  \ar[rr]^{\times d_-} && \h  \, \mbox{,} }
\end{equation}
where $\times$ means the operator of multiplication by the corresponding function.

The linear operators given by the first two arrows in~\eqref{arr} are given by block matrices of type
$
\left(
\begin{matrix}
  * & 0 \\
  * & *
\end{matrix}
\right)
$ with invertible operators on the diagonal. Therefore their composition will be again the matrix of such type.

The linear operators given by the last three arrows in~\eqref{arr} are given by block matrices of type
$
\left(
\begin{matrix}
  * & * \\
  0 & *
\end{matrix}
\right)
$ with invertible operators on the diagonal. Therefore their composition will be again the matrix of such type.

Therefore we obtain that the linear operator given by the action of   $g$ is written as the product of invertible linear operators of type
$$
\left(
\begin{matrix}
  * & * \\
  0 & *
\end{matrix}
\right)
\left(
\begin{matrix}
  * & 0 \\
  * & *
\end{matrix}
\right)
$$
and hence $g_{++}$ is an invertible operator.

Since the topology on $\h$ is the product topology from $\h_-$ and $\h_+$ (see Section~\ref{unit_circle}), the operator $g_{++}$ is continuous. Hence $g_{++}$ belongs to $\mathop{ \rm GL}(\h_+)$.

The proof that $g_{--} \in \mathop{ \rm GL}(\h_-)$ is analogous. But one has to take $d = d_+ d_-$ and $f^{\circ -1} = ((f^{\circ -1})_-)^{\circ -1}  \circ (f^{\circ -1})_+  $,
hence $f = ((f^{\circ -1})_+)^{\circ -1}  \circ (f^{\circ -1})_-$. Then by similar reasonings as above it is possible to write the operator for the action of $g$ as the composition of five invertible linear operators similar to~\eqref{arr}, but
where the first two operators will be given by the matrices of type  $
\left(
\begin{matrix}
  * & * \\
  0 & *
\end{matrix}
\right)
$
and the last three operators will be given by the matrices of type  $
\left(
\begin{matrix}
  * & 0 \\
  * & *
\end{matrix}
\right)
$. Therefore, at the end,  the linear operator given by the action of $g$ can be written as the product of invertible linear operators of type
$$
\left(
\begin{matrix}
  * & 0 \\
  * & *
\end{matrix}
\right)
\left(
\begin{matrix}
  * & * \\
  0 & *
\end{matrix}
\right)
$$
and hence $g_{--}$ is an invertible operator.

\end{proof}

\section{Spaces of continuous linear operators $\Hom(\cdot , \cdot)$ and  ${\Homb}(\cdot ,  \cdot ) $}
\label{sect-cont}
\subsection{Continuous linear operators}

\begin{defin}
For any two Silva spaces ${\mathcal V}_1$ and ${\mathcal V}_2$ over the field $\C$ by $\Hom({\mathcal V}_1 , {\mathcal V}_2)$ denote the $\C$-vector space of continuous linear operators from ${\mathcal V}_1$ to ${\mathcal V}_2$.
\end{defin}

\begin{prop}   \label{holom}
\label{oper}
There is the following explicit description of spaces of linear continuous operators.
\begin{enumerate}
\item  \label{it-1} The space $\Hom(\h_+, \h_-)$ is identified with the space of  functions $f(z, w)$ which are holomorphic on an open set (depending on $f$) of $\widehat{\C} \times \widehat{\C}$ containing $(\widehat{\C} \setminus {D}_1) \times (\widehat{\C} \setminus \overline{D}_1) $ and vanishing on
$(\infty \times \widehat{\C})  \cup (\widehat{\C} \times \infty)$. The action is
\begin{equation}  \label{emb1}
h(z)  \longmapsto  \oint_{| w | = 1+ \ve}   f(z, w) h(w) dw  \, \mbox{,}
\end{equation}
where a real number $\ve > 0$ is small enough such that the function   $h(w)$ is holomorphic in a neighbourhood of  the circle $| w | = 1+ \ve$ (the result does not depend on the choice of~$\ve$).
\item  \label{it-2}
The space $\Hom(\h_-, \h_+)$ is identified with the space functions $f(z, w)$ which are holomorphic on an open set (depending on $f$) of ${\C} \times {\C}$ containing $\overline{D}_1 \times  {D}_1 $. The action is
$$
h(z)  \longmapsto  \oint_{| w | = 1- \ve}   f(z, w) h(w) dw  \, \mbox{,}
$$
where a real number $\ve > 0$ is small enough such that the function   $h(w)$ is holomorphic in a neighbourhood of  the circle $ | w | = 1- \ve$ (the result does not depend on the choice of~$\ve$).
\item  \label{it-3}  The space $\Hom(\h_+, \h_+)$ is identified with the space of functions $f(z, w)$ which are  holomorphic  on an open set (depending on $f$) of ${\C} \times \widehat{\C}$ containing $\overline{D}_1 \times (\widehat{\C} \setminus \overline{D}_1) $ and vanishing on
${\C} \times \infty$. The action is
$$
h(z)  \longmapsto  \oint_{| w | = 1+ \ve}   f(z, w) h(w) dw  \, \mbox{,}
$$
where a real number $\ve > 0$ is small enough such that the function   $h(w)$ is holomorphic in a neighbourhood of  the circle $| w | = 1+ \ve$ (the result does not depend on the choice of~$\ve$).
\item The space $\Hom(\h_-, \h_-)$ is identified with the space of functions $f(z, w)$ which are holomorphic on an open set (depending on $f$) of $\widehat{\C} \times \C$ containing $(\widehat{\C} \setminus {D}_1) \times D_1 $ and vanishing on
$\infty \times \C$. The action is
$$
h(z)  \longmapsto  \oint_{| w | = 1- \ve}   f(z, w) h(w) dw  \, \mbox{,}
$$
where a real number $\ve > 0$ is small enough such that the function   $h(w)$ is holomorphic in a neighbourhood of  the circle $| w | = 1- \ve$ (the result does not depend on the choice of~$\ve$).
\end{enumerate}
\end{prop}
\begin{proof}
The proof is based on the Grothendieck--K\"othe--Sebasti\~ao e Silva duality from Section~\ref{Grot-dual} and on the generalized Hartogs' theorem on separate analyticity (or, in \linebreak other words, on  Hartogs' fundamental theorem). We will give the proof of item~\ref{it-1} of the proposition. The other items are proved analogously.

It is easy to see that formula~\eqref{emb1} gives the map of the space of corresponding holomorphic functions to the space  $\Hom(\h_+, \h_-)$. We construct the inverse map to this map.

Suppose  $T \in \Hom(\h_+, \h_-)$. For every integer $k \ge 0 $ we consider  $f_k = T(z^k) $. Then the function $f_k(z)$ is from $\h_-$.

Now the function
\begin{equation}\label{fun_f}
f(z,w) = \frac{1}{2 \pi i} \left( \sum_{k \ge 0} f_k(z) w^{- k -1}  \right)
\end{equation}
will be the function that we need for formula~\eqref{emb1}. We will prove it.

First, we check that the function $f(z,w)$ given by series~\eqref{fun_f} is defined for any ${(z, w) \in (\widehat{\C} \setminus {D}_1) \times (\widehat{\C} \setminus \overline{D}_1)  }$  and also that $f(z,w)$ is a holomorphic function with respect to the variable $w \in \widehat{\C} \setminus \overline{D}_1$ under the fixed variable  $z = a \in \widehat{\C} \setminus {D}_1 $.

Indeed, consider the continuous linear functional $\chi_{a}$ from $\h_+$ to $\C$ given as  ${\chi_{a}(h)= T(h)(a)}$. Then by Section~\ref{Grot-dual}, the functional $\chi_{a}$ comes from the holomorphic function
$g(z) = \sum_{k \ge 0} c_k z^{- k -1}$ on $\widehat{\C} \setminus \overline{D}_1$. Calculating the functional  $\chi_{a}$ on various functions $z^k$, where $k \ge 0$, we obtain that $c_k = f_k(a) / (2 \pi i)$. Therefore
the function
\begin{equation}  \label{form_f}
f(a, w) = \sum_{k \ge 0} c_k w^{- k -1}
\end{equation}
is holomorphic with respect to the variable $w$  from  $\widehat{\C} \setminus \overline{D}_1$.

Second, we fix $s $ and  $r$  from $\dr$  such that $1 < s < r$. Consider any fixed $b \in \C$ such that $| b | > r$.
Consider the function (that depends on $b$) of the variable $z$
$$
g_{b}(z)= \sum_{k \ge 0}  b^{-k -1} z^k  \, \mbox{.}
$$

Consider the set $M_r$ that consists of all functions $g_b(z)$ with $| b | > r$. It is easy to see by direct calculation  that the set $M_r$   is a bounded subset in the Banach space
$\tilde{\h} \left( D_{s}  \right)$, where this Banach space  consists of functions from $\h \left( D_{s}  \right)$ which can be extended to continuous functions on the closed  disk $\overline{D}_s$.
Therefore the set $M_r$ is a bounded subset in the space $\h_+$.

Hence the set $T(M_r)$ is a bounded subset in the space $\h_-$. Hence and by Remark~\ref{Silva} there is  an integer $n > 0$ such that
$$
T(W_r)  \; \subset \; \tilde{\h}_0 \left(\widehat{\C} \setminus \overline{D}_{1 - 1/n} \right)  \, \mbox{,}
$$
see formula~\eqref{indlim} for the notation.

Besides, for any function $g_b(z)$ from the set $M_r$  we have in $\h_-$ the following equalities
$$
T(  g_b(z)) = T \left( \lim_{l \to \infty} \sum_{k \ge 0}^l  b^{-k -1} z^k  \right) =  \lim_{l \to \infty}  T \left( \sum_{k \ge 0}^l  b^{-k -1} z^k  \right) =  \lim_{l \to \infty} \sum_{k \ge 0}^l  f_k(z) b^{-k -1} = f(z,b)  \, \mbox{.}
$$
We put also $f(z, \infty)=0$ (that corresponds to formula~\eqref{form_f} in this case).

Thus, we proved that the function $f(z,w)$ given by series~\eqref{fun_f}  is   defined on
an open subset $U$  of $\widehat{\C} \times \widehat{\C}$ containing $(\widehat{\C} \setminus {D}_1) \times (\widehat{\C} \setminus \overline{D}_1) $ and this function vanishes on
${(\infty \times \widehat{\C})  \cup (\widehat{\C} \times \infty)}$. Besides, $f(z,w)$ is a holomorphic function of one variable $w$ under fixed $z = a \in \widehat{\C} \setminus {D}_1 $  such that $(a,w) \in U$, and $f(z,w)$ is a holomorphic function of one variable $z$  under fixed
$w = b \in \widehat{\C} \setminus \overline{D}_1$ such that $(z,b) \in U$.

Hence, by the generalized Hartogs' theorem on separate analyticity (see~\cite{H} and the exposition in the textbook~\cite[Chapter~7, \S~5]{BM}) the function $f(z,w)$ is a holomorphic function on $U$.

\end{proof}

\begin{nt}  \em  \label{Groth}
Suppose the Silva spaces are given as
\begin{equation}
\label{lim}
 {\mathcal V_1} = \varinjlim_{n \in \mathbb{N}} A_{n}      \qquad  \mbox{and} \qquad  {\mathcal V_2}  = \varinjlim_{m \in \mathbb{N}} C_{m}  \, \mbox{,}
\end{equation}
 and a linear operator $T$ belongs to  $\Hom ({\mathcal V_1}, {\mathcal V_2})$. Then for any $n \in \mathbb{N}$  there is $m \in \mathbb{N}$ such that
 $$T(A_n) \, \subset  \, C_m      \qquad    \mbox{and}  \qquad   T |_{A_n}   \, \in  \, \Hom (A_n, C_m)   \, \mbox{,} $$
 see~\cite[\S~3.9, \S~5.5]{F}   (this is the result of A.~Grothendieck, see~\cite[Introduction~IV.4]{Gro0}).

 Note that in case when ${\mathcal V_1}$ and ${\mathcal V_2}$ belong to  the set
  $\left\{\h_- ,  \h_+ , \h \right\}$, this statement easily follows from Proposition~\ref{holom}. Indeed, for example, in formula~\eqref{emb1},   if the variable $w$  belongs the  circle $| w | = 1 + \ve$, where a real number $\ve > 0$ is  fixed,  then the function  $f(z,w)$ (and hence the function in the left hand side of formula~\eqref{emb1}) will be a holomorphic function with respect to the variable  $z \in  \widehat{\C} \setminus \overline{D}_{1 - \delta}$, where a real number $\delta > 0$ depends on $\ve$, since the circle $| w | = 1 + \ve$ is compact.
\end{nt}

\begin{defin}
For any two Silva spaces ${\mathcal V_1}$ and ${\mathcal V_2}$ with presentations by inductive limits as in formula~\eqref{lim}, by ${\Homb}  ({\mathcal V_1}, {\mathcal V_2})$
denote the linear subspace of operators $T$ from ${\Hom}  ({\mathcal V_1}, {\mathcal V_2})$ such that $T$ belongs to $\Hom({\mathcal V_1}, C_m)$, where  $m \in {\mathbb N}$ (and depends on~$T$).
\end{defin}

\begin{nt} \label{Groth2}  \em
For any two Silva spaces ${\mathcal V_1}$ and ${\mathcal V_2}$ with presentations by inductive limits as in~\eqref{lim}, it follows
from Remark~\ref{Groth} that
$$
{\Homb}  ({\mathcal V_1}, {\mathcal V_2})  = \varinjlim_{m \in \mathbb{N}}  \varprojlim_{n \in \mathbb{N}}  \Hom(A_n, C_m)   \; \,
\subset \; \,
\varprojlim_{n \in \mathbb{N}} \varinjlim_{m \in \mathbb{N}}  \Hom(A_n, C_m) = {\Hom}  ({\mathcal V_1}, {\mathcal V_2})   \, \mbox{.}
$$
Thus, from this point of view, the definition of ${\Homb}  ({\mathcal V_1}, {\mathcal V_2})$ is just the interchanging  the order of (projective and inductive) limits $\varprojlim$
and  $\varinjlim$ in the definition of ${\Hom}  ({\mathcal V_1}, {\mathcal V_2})$.
 \end{nt}

\medskip

From Remark~\ref{Groth} it follows that
\begin{equation}  \label{inv-comp}
\Hom  ({\mathcal V_1}, {\mathcal V_2}) \times  {\Homb}  ({\mathcal V_2}, {\mathcal V_3})  \times   {\Hom}  ({\mathcal V_3}, {\mathcal V_4}) \; \, \subset  \; \,  {\Homb}  ({\mathcal V_1}, {\mathcal V_4})
\end{equation}
for any four Silva spaces $\mathcal V_i$, $1 \le i \le 4$.

\begin{prop}  \label{tilde-holom}
An operator  $T$ from  ${\Hom}  ({\mathcal \h_+}, {\mathcal \h_+})$ belongs to  ${\Homb}  ({\mathcal \h_+}, {\mathcal \h_+})$ if and only if there is a real number $\delta > 0$  (which  depends on $T$) such that  the function $f(z,w)$ corresponding to $T $  by item~\ref{it-3} of Proposition~\ref{holom}   will be holomorphic on an open set
 ${D_{1+ \delta} \times (\widehat{\C} \setminus \overline{D}_1) }$.

 The analogous property (in view of Proposition~\ref{holom}) is valid for any $T$ from \linebreak ${\Homb}  ({\mathcal V_1}, {\mathcal V_2})$, where ${\mathcal V_1}$ and ${\mathcal V_2}$ belong to the set
  $\left\{\h_- ,  \h_+  \right\}$.
\end{prop}
\begin{proof}
It is evident that if an operator $T$ has  the function $f(z,w)$ as in the condition of the  proposition, then $T$ belongs to  ${\Homb}  ({\mathcal \h_+}, {\mathcal \h_+})$.

The proof in the opposite direction is just the repetition of the proof of Proposition~\ref{holom}, where we used  the Grothendieck--K\"othe--Sebasti\~ao e Silva duality   and the generalized Hartogs' theorem on separate analyticity (moreover, in the second part of this proof the reasonings about bounded subsets can be omitted, since we consider ${\Homb}(\cdot, \cdot)$).
\end{proof}

\subsection{Compositions, convolutions,  traces, determinants and inverse operators}  \label{big-trace}
\subsubsection{Composition and convolution}
Consider two continuous linear operators
$$T  \;  :  \;  {\mathcal V_1}  \lrto {\mathcal V_2}
\qquad \mbox{and}  \qquad
P \;  :  \;  {\mathcal V_2}  \lrto {\mathcal V_3}  \, \mbox{,}
$$
where $ {\mathcal V_1}$, ${\mathcal V_2}$ and ${\mathcal V_3}$ belong to the set  $\left\{\h_- ,  \h_+ \right\}$.
Then the composition $P \circ T$ can be written, in view of Proposition~~\ref{oper}, as the convolution (see also the description in similar situation in~\cite[p.~11]{ADKP}).

For example, consider ${\mathcal V_1} = {\mathcal V_3} = \h_+$ and
 ${\mathcal V_2} = \h_-$. Let the operator  $T$ correspond to the function $t$ and the operator $P$ correspond to the function $p$ as in  Proposition~\ref{oper}. Then the operator
 $P \circ T \; : \; \h_+  \to \h_+$ corresponds to the function $p \ast t $, where $p \ast t$ is the convolution
 \begin{equation}  \label{conv}
 (p \ast t) (z, w) = \oint_{| v | = 1 - \ve_w } p(z,v) t(v,w) dv  \, \mbox{,}
 \end{equation}
 where a real number $\ve_w >0$ depends on $w$ (and on the function $t$) and is small enough such that the function $t(v,w)$ is holomorphic in the neighbourhood of the circle ${ | v | = 1 - \ve_w }$ under fixed $w$ (the result does not depend on the choice of such $\ve_w$).

 This description of compositions of linear operators as convolutions of holomorphic functions allows to write the composition of linear continuous operators from $\h$ to $\h$. For this goal we have to use
  the product of block matrices  with respect to the decomposition $\h = \h_- \oplus \h_+$, and the product of separate blocks
 is written by convolutions of the corresponding functions.

\subsubsection{Trace}  \label{trace}

Suppose that $T$ belongs to  ${\Homb}  ({\mathcal \h_+}, {\mathcal \h_+})$.
Let  $f(z,w)$ be the function that corresponds to $T$ via item~\ref{it-3} of Proposition~\ref{holom} and Proposition~\ref{tilde-holom}.
The operator $T$ has the trace in $\C$ given by the formula
\begin{equation}  \label{integr}
\tr(T) = \oint_{| v | = 1 + \sigma } f (v,v)  dv  \, \mbox{,}
\end{equation}
where a real number  $\sigma >0$  is small enough such that the function $f(v,v)$ is holomorphic in the neighbourhood of the circle $| v | = 1 + \sigma$  (the result does not depend on the choice of  $\sigma$).

Note that the family of functions from $\h_+$
\begin{equation}  \label{family}
  1  \,   \mbox{,}  \,  \;  z \,    \mbox{,}  \,  \;  z^2  \,  \mbox{,}  \,  \;  z^3  \, \mbox{,}   \, \;  z^4    \,   \mbox{,}  \ldots \,
    \mbox{,}
    \;  z^j  \, \mbox{,}   \ldots
\end{equation}
has the property that every element from $\h_+$ can be uniquely written  as the convergent series $\sum_{j \ge 0} d_j z^j$ with $d_j \in \C$
in the topology of $\h_+$. Thus the family of functions~\eqref{family} is the generalization of basis for the topological vector space $\h_+$.

 Then we have
\begin{equation}  \label{ser}
\tr(T) = \sum_{j \ge 0} a_{j,j}  \, \mbox{,}
\end{equation}
where $\left( a_{j,k} \right)_{j,k \ge 0}$ is the matrix of the operator $T$ with respect to the family of functions~\eqref{family}. Besides, the function $f(z,w)$ has the power series decomposition
$$ f(z,w) = \frac{1}{2 \pi i} \left( \sum_{j \ge 0, k \ge 0 }  a_{j, k} z^j w^{-k -1}  \right)$$
  in the point $0 \times \infty   \in \C \times \widehat{\C}$.

Since the  radius of convergence of a power series depends only on the absolute values of coefficients of this series, after changing the coefficients $a_{j,k}$ of the power series to the coefficients $| a_{j,k}  |$
we obtain that the resulting  new function will be holomorphic at $(z,w)= (v,v)$, where $| v | = 1 + \sigma$, $\sigma > 0$ (see formula~\eqref{integr}). Now, by considering analog of formula~\eqref{integr} for the new function,  we obtain that the series~\eqref{ser} is absolutely convergent.

Fix a real number  $\delta  > 0$. Consider the space of linear operators  $ T \in {\Homb} ({\mathcal \h_+}, {\mathcal \h_+})$
 such that they correspond to holomorphic functions $f(z,w)$ on the open set \linebreak
 ${D_{1+ \delta} \times (\widehat{\C} \setminus \overline{D}_1) }$.
 We consider on the space of these holomorphic functions the topology of uniform convergence on compact sets, and thus we define the topology on the space of corresponding linear operators, obtaining a Fr\'echet space. From formula~\eqref{integr} it is easy to see that the trace is a continuous, linear  and hence  smooth function on this Fr\'echet space of continuous linear operators $T$ (which depends on the number $\delta$).

 The space  ${\Homb} ({\mathcal \h_+}, {\mathcal \h_+})$
 is the inductive limit of described Fr\'echet spaces over all $\delta$ (and it is enough to take the countable system of indices  $\delta = 1/n$, $n \in \dn$.) We consider the inductive limit topology on the space ${\Homb} ({\mathcal \h_+}, {\mathcal \h_+})$  in the category of Hausdorff locally convex topological vector spaces, i.e. we go to the  Hausdorff quotient vector spaces if it is necessary.
 Now we obtain that the trace is a smooth linear function on the space ${\Homb}  ({\mathcal \h_+}, {\mathcal \h_+})$.

\subsubsection{Determinant and inverse operator}  \label{determin}

Suppose that $T$ belongs to ${\Homb}  ({\mathcal \h_+}, {\mathcal \h_+})$. Then the operator $1+ T$ has  the determinant from $\C$ given by the formula
\begin{equation}  \label{deter}
\det(1 + T) = \sum_{l \ge 0} \tr(\Lambda^l T)   \, \mbox{,}
\end{equation}
where $\Lambda^l T$ are the exterior powers of operator $T$.

Let us show that every term in the right hand side of formula~\eqref{deter} makes sense and that this series is absolutely convergent.

Let  $f(z,w)$ be the function that corresponds to $T$ via item~\ref{it-3} of Proposition~\ref{holom} and Proposition~\ref{tilde-holom}.
It is not difficult to see (see also the nice exposition in the case of Hilbert spaces, but which works also in our case, in~\cite[Theorem~3.10]{Si}, and see~\cite[Chap.~III, \S~2]{Gro}) that
\begin{equation}  \label{trace-wedge}
 \tr(\Lambda^l T)  = \frac{1}{l!} \oint_{|v_1| = 1 + \sigma}  \ldots \oint_{|v_l| = 1 + \sigma}  f
 \begin{pmatrix}
   v_1 & \ldots & v_l\\
    v_1 & \ldots & v_l
 \end{pmatrix}  dv_1 \ldots dv_l  \, \mbox{,}
\end{equation}
where
\begin{equation}  \label{det-form}
f
 \begin{pmatrix}
   z_1 & \ldots & z_l\\
    w_1 & \ldots & w_l
    \end{pmatrix}
=
\det \left( ( f(z_j, w_k))_{1 \le j , k \le l}   \right)
\end{equation}
 and a real number $\sigma >0$  is small enough such that the function $f(z,w)$ is holomorphic in the neighbourhood of the product of two  circles $| z | = 1 + \sigma$ and $| w | = 1 + \sigma$  (the result does not depend on the choice of  $\sigma$).

Now by the   Hadamard's theorem on determinants (or Hadamar's inequality)
$$
\max_{| z | = |w| = 1 + \sigma} \left| \det \left( (f(z_j, w_k))_{1 \le j , k \le l}   \right)  \right|  \, \le  \,  b^l l^{l/2}   \, \mbox{,} \quad \mbox{where}  \quad
b \ge \max_{| z | = |w| = 1 + \sigma}  |f(z,w)  |  \, \mbox{.}
$$

Using the Stirling's formula for $l!$ we obtain
$$
\frac{b^l l^{l/2}}{l!} \, \le  \, c \left( \frac{b e}{\sqrt{l}} \right)^l  \, \mbox{,}
$$
where $c > 0$ is a real constant.

From these inequalities we obtain that the series~\eqref{deter} is absolutely convergent. (See also the similar inequalities in~\cite[Chap.~II, \S~2, Chap.~III, \S~2]{Gro}.)

For any $P$ from  ${\Homb}  ({\mathcal \h_+}, {\mathcal \h_+})$ it is possible to write the Taylor decomposition for $\det(1+T +P)$ using the formula for $\tr \Lambda^l(T +P)$ that comes from the right hand side
of formula~\eqref{det-form} (see~\cite[Chap.~I, \S~4]{Gro}). Hence and again using the Hadamar's inequality and the Stirling's formula
 we obtain that
the determinant is a continuous and smooth function on
the space ${\Homb}  ({\mathcal \h_+}, {\mathcal \h_+})$ (where we consider the inductive limit topology as at the end of Section~\ref{trace}).

Hence (and we can  use~\cite[Chap.~I, \S~3]{Gro}, since the set of  operators of finite rank is dense, because there is family~\eqref{family}) we obtain that for any $P$ and $T$ from  ${\Homb}  ({\mathcal \h_+}, {\mathcal \h_+})$:
\begin{equation}  \label{invert-mult}
\det (1+P)  \det(1+T) = \det \left( (1+P)(1+T)    \right) = \det (1+P +T +PT)  \, \mbox{.}
\end{equation}

If $\det(1 + T) \ne 0$ for $T$ from  ${\Homb}  ({\mathcal \h_+}, {\mathcal \h_+})$, then there is the inverse operator  $(1+T)^{-1}$ from ${1+ \Homb}  ({\mathcal \h_+}, {\mathcal \h_+})$ given by a formula
\begin{equation}  \label{invert}
(1+T)^{-1}= 1 - \det(1+T)^{-1}R \mbox{,}
\end{equation}
where $R$  from   ${\Homb}  ({\mathcal \h_+}, {\mathcal \h_+})$  corresponds to the following holomorphic function of variables $z$ and $w$ via Proposition~\ref{holom} and Proposition~\ref{tilde-holom} (see~\cite[Chap.~I, \S~6, Chap.~III, \S~2]{Gro}):
\begin{equation}  \label{inverse}
\sum_{l \ge 0} \frac{1}{l!} \oint_{|v_1| = 1 + \sigma} \ldots \oint_{|v_l| = 1 + \sigma} f
 \begin{pmatrix}
   v_1 & \ldots & v_l &z \\
    v_1 & \ldots & v_l & w
 \end{pmatrix}  dv_1 \ldots dv_l
\end{equation}

Moreover, the map $T \mapsto (1+T)^{-1}$ is continuous and smooth on the open subset of linear operators $T$ from ${\Homb}  ({\mathcal \h_+}, {\mathcal \h_+})$ such that $\det(1 +T) \ne 0$, since it follows again from
the Hadamar's inequality, the Stirling's formula, formula~\eqref{inverse} and the formula
$$
(1 + T + P)^{-1}= (1+T)^{-1} \left(1 + (1+T)^{-1}P \right)^{-1}  \, \mbox{.}
$$

Besides, for $T$ from ${\Homb}  ({\mathcal \h_+}, {\mathcal \h_+})$ the operator $1+T$ is invertible in $\Hom  ({\mathcal \h_+}, {\mathcal \h_+})$ and consequently $(1+T)^{-1}$ belongs to  ${1+ {\Homb}  ({\mathcal \h_+}, {\mathcal \h_+})}$ if and only if $\det(1 +T) \ne 0$. (Indeed, we use formula~\eqref{invert}, and also if $(1+T)O =1$, then $O = 1 -TO$ and $TO$ belongs to ${\Homb}  ({\mathcal \h_+}, {\mathcal \h_+})$ by formula~\eqref{inv-comp}, and we can use formula~\eqref{invert-mult}.)

\begin{nt}  \em
Analogs of all the reasonings in  Sections~\ref{trace}--\ref{determin} are applicable also to~$\h_-$.
\end{nt}

\section{The determinant central extensions of groups \linebreak ${\mathop{ \rm GL}}^+_{\rm res}(\h)$, ${\mathop{ \rm GL}}^-_{\rm res}(\h)$ and ${\mathop{ \rm GL}}^0_{\rm res}(\h)$}
\label{group-gl}

\subsection{The groups ${\mathop{ \rm GL}}^+_{\rm res}(\h)$ and ${\mathop{ \rm GL}}^-_{\rm res}(\h)$ and their central extensions}

For any $A$ from  ${\mathop{ \rm GL}}(\h)$ consider the block matrix
\begin{equation}  \label{block-mat-2}
\begin{pmatrix}
  A_{--} & A_{+-} \\
  A_{-+} & A_{++}
\end{pmatrix}
\end{equation}
with respect to the decomposition $\h = \h_- \oplus \h_+$ as in~\eqref{block-mat}.

\begin{defin}  \label{GL+} Define the set
$
{\mathop{ \rm GL}}^+_{\rm res}(\h) $
as the set of operators
  $ A \in \mathop{{\rm GL}}(\h) $ such that  the operators  $A_{+-}$ and $(A^{-1})_{+-}$  are from   ${\Homb}  ({\mathcal \h_+}, {\mathcal \h_-}) $ and there is
  $r \in {\rm GL }(\h_+)$ (which depends on $A$) such that $A_{++} -r $ belongs to ${\Homb}  ({\mathcal \h_+}, {\mathcal \h_+})$.
\end{defin}

\begin{prop}  \label{group}
The set $
{\mathop{ \rm GL}}^+_{\rm res}(\h) $ is a subgroup of the group $\mathop{{\rm GL}}(\h) $, and there is a natural exact sequence of groups
\begin{equation}  \label{exact-seq}
1 \lrto {\mathcal T}  \lrto {\mathcal E} \stackrel{\theta}{\lrto} {\mathop{ \rm GL}}^+_{\rm res}(\h)  \lrto 1  \, \mbox{,}
\end{equation}
where the group ${\mathcal E}$ consists of all pairs $(A, r) \in {\mathop{ \rm GL}}^+_{\rm res}(\h) \times {\rm GL }(\h_+)$ such that $A_{++} -r $ belongs to ${\Homb}  ({\mathcal \h_+}, {\mathcal \h_+})$,
and $\theta$ is the projecton to the first group in the direct product.
\end{prop}
\begin{proof}
Clearly, if  $A$ and $B$ from $\mathop{{\rm GL}}(\h) $ such that $A_{+-}$ and   $B_{+-}$   are from   ${\Homb}  ({\mathcal \h_+}, {\mathcal \h_-}) $, then
$(AB)_{+-}$ is from ${\Homb}  ({\mathcal \h_+}, {\mathcal \h_-}) $, see formula~\eqref{inv-comp}.

Let us check that the set $\mathcal E$ is a group.

Suppose that $(A_1, r_1)$ and $(A_2, r_2)$ are from ${\mathcal E}$. Then $(A_1A_2, r_1 r_2)$ is from ${\mathcal E}$, since
\begin{multline*}
(A_1 A_2)_{++} - r_1 r_2 = (A_1)_{-+} (A_2)_{+-} + (A_1)_{++} (A_2)_{++} - r_1 r_2 =  \\
= (A_1)_{-+} (A_2)_{+-} + \left((A_1)_{++} - r_1 \right) (A_2)_{++} + r_1 \left((A_2)_{++} - r_2 \right)  \, \mbox{,}
\end{multline*}
and we  use formula~\eqref{inv-comp}.

Let $q = (A_1)_{++} - r_1$. Analogously we obtain that $(A_1^{-1}, r_1^{-1})$ is from ${\mathcal E}$, since
\begin{multline*}
(A_1^{-1})_{++} - r_1^{-1}  = \left((A_1^{-1})_{++} \, r_1 - 1 \right)r_1^{-1} = \left( (A_1^{-1})_{++} \, \left((A_1)_{++} - q \right) -1   \right)r_1^{-1} =  \\ = \left(1 - (A_1^{-1})_{-+} (A_1)_{+-} - (A_1^{-1})_{++} \, q -1  \right) r_1^{-1} =
\left( - (A_1^{-1})_{-+} (A_1)_{+-} - (A_1^{-1})_{++} \, q \right) r_1^{-1}  \, \mbox{,}
\end{multline*}
and we  use formula~\eqref{inv-comp} again.
\end{proof}

\medskip

We note that the group ${\mathcal T}$ consists of the pairs $(1, r) \in {\mathop{ \rm GL}}^+_{\rm res}(\h) \times {\rm GL }(\h_+)$ such that $q = 1 -r $ belongs to ${\Homb}  ({\mathcal \h_+}, {\mathcal \h_+})$.
Then by Section~\ref{determin}  the following map $\Psi$ is a well-defined homomorphism from the group ${\mathcal T}$ to the group $\C^*$:
$$
{\mathcal T}  \ni (1,r) \,   \stackrel{\Psi}{\longmapsto}  \, \det r = \det (1 -q)  \in \C^*  \, \mbox{.}
 $$

For any $s \in {\rm GL }(\h_+)$ we have $$\det(1 - sqs^{-1}) = \det(1-q)  \, \mbox{,}$$
since by formula~\eqref{deter} it is enough to see
$$\tr(\Lambda^l(sqs^{-1})) = \tr(\Lambda^l(s))  \, \mbox{,}$$
 and this follows easily from formulas~\eqref{conv}, \eqref{trace-wedge}    and~\eqref{det-form}, because we have just to change the order of the integration in these formulas.

Therefore $\Ker \Psi$ is a normal subgroup in the group ${\mathcal E}$.  We denote the group
$$\widetilde{{\mathop{ \rm GL}}^+_{\rm res}(\h) }  = {\mathcal E} / \Ker \Psi  \mbox{.}$$
From~\eqref{exact-seq} we obtain the following  central extension, which we will call {\em the determinant central extension} of the group ${\mathop{ \rm GL}}^+_{\rm res}(\h) $:
$$
1 \lrto \C^* \lrto \widetilde{{\mathop{ \rm GL}}^+_{\rm res}(\h) }  \lrto {\mathop{ \rm GL}}^+_{\rm res}(\h)  \lrto 1  \, \mbox{.}
$$

\begin{nt}  \em
Similar to the group ${\mathop{ \rm GL}}^+_{\rm res}(\h) $, one defines the group ${\mathop{ \rm GL}}^-_{\rm res}(\h)$, but in
Definition~\ref{GL+}
one has to demand that $A_{-+}$ and $(A^{-1})_{-+}$  are from   ${\Homb}  ({\mathcal \h_-}, {\mathcal \h_+}) $ instead of  that $A_{+-}$ and $(A^{-1})_{+-}$  are from   ${\Homb}  ({\mathcal \h_+}, {\mathcal \h_-}) $.
Then, similarly to the Proposition~\ref{group},
 starting from the subgroup of the group ${\mathop{ \rm GL}}^-_{\rm res}(\h) \times {\rm GL }(\h_+)$ one  defines {\em the determinant central extension} of the group ${\mathop{ \rm GL}}^-_{\rm res}(\h) $.
\end{nt}

\subsection{The subgroup ${\mathop{ \rm GL}}^0_{\rm res}(\h)$ of the group  ${\mathop{ \rm GL}}^+_{\rm res}(\h)$}

It is possible to define the subgroup ${\mathop{ \rm GL}}^0_{\rm res}(\h)$ of the group  ${\mathop{ \rm GL}}^+_{\rm res}(\h)$ such that it will looks more symmetric and is similar to the case of Hilbert spaces described
in~\cite[\S~6.2, \S~6.6]{PS}.

Denote by $J \in \mathop{{\rm GL}}(\h)$ the linear operator given  by the block  matrix
$$
J = \begin{pmatrix}
  -1 & 0 \\
  0 & 1
\end{pmatrix}
$$
with respect
to the decomposition $\h = \h_- \oplus \h_+$, see formula~\eqref{block-mat-2}.

\begin{defin} Define the set
$
{\mathop{ \rm GL}}^0_{\rm res}(\h) $
as the set of operators
  $ A \in \mathop{{\rm GL}}(\h) $ such that   $[J, A] \in {\Homb}  ({\mathcal \h}, {\mathcal \h}) $ and there is
  $r \in {\rm GL }(\h_+)$ (which depends on $A$) such that $A_{++} -r $ belongs to ${\Homb}  ({\mathcal \h_+}, {\mathcal \h_+})$.
\end{defin}

Using notation~\eqref{block-mat-2}, we note that
$$
[J, A] = \begin{pmatrix}
  0 & -2 A_{+-} \\
  2 A_{-+} & 0
\end{pmatrix}  \, \mbox{.}
$$
Therefore the condition $[J, A] \in {\Homb}  ({\mathcal \h}, {\mathcal \h})$ is equivalent to the following condition:
\begin{equation}  \label{plmi}
A_{+-}  \in {\Homb}  ({\mathcal \h_+}, {\mathcal \h_-}) \qquad  \mbox{and}
\qquad
A_{-+}  \in {\Homb}  ({\mathcal \h_-}, {\mathcal \h_+})  \, \mbox{.}
\end{equation}

Since $[J, \, \cdot \,]$ is a derivation of the ring $\Hom(\h, \h)$, we obtain from formula~\eqref{inv-comp} by considering $[J, A A^{-1}]$ and $[J, AB]$ that if $[J, A]$ and $[J, B]$ belong to ${\Homb}  ({\mathcal \h}, {\mathcal \h}) $, then $[J, A^{-1}]$ and $[J, AB]$
belong to ${\Homb}  ({\mathcal \h}, {\mathcal \h}) $.

Therefore, by repeating the proof of Proposition~\ref{group} we obtain that the set $
{\mathop{ \rm GL}}^0_{\rm res}(\h) $ is a group. And it is a subgroup in the following  groups:
$${\mathop{ \rm GL}}^0_{\rm res}(\h) \hookrightarrow {\mathop{ \rm GL}}^+_{\rm res}(\h) \, \mbox{,}  \qquad \qquad {\mathop{ \rm GL}}^0_{\rm res}(\h) \hookrightarrow {\mathop{ \rm GL}}^-_{\rm res}(\h)  \, \mbox{.}$$

\medskip

By restricting the determinant central extension of  the group ${\mathop{ \rm GL}}^+_{\rm res}(\h) $ to the subgroup ${\mathop{ \rm GL}}^0_{\rm res}(\h) $  (or, in other words, by taking the pullback of the determinant central extension via the embedding ${\mathop{ \rm GL}}^0_{\rm res}(\h) \hookrightarrow{\mathop{ \rm GL}}^+_{\rm res}(\h) $) we obtain the central extension of the group ${\mathop{ \rm GL}}^0_{\rm res}(\h) $.

We call this constructed central extension {\em the determinant central extension} of the group ${\mathop{ \rm GL}}^0_{\rm res}(\h) $.

From the construction it follows also that the pullback of the determinant central extension of the group  ${\mathop{ \rm GL}}^-_{\rm res}(\h) $ via the embedding ${\mathop{ \rm GL}}^0_{\rm res}(\h) \hookrightarrow{\mathop{ \rm GL}}^-_{\rm res}(\h) $ gives the same central extension of  the group ${\mathop{ \rm GL}}^0_{\rm res}(\h) $, i.e. the determinant central extension of  the group ${\mathop{ \rm GL}}^0_{\rm res}(\h) $.

\section{The determinant central extensions of Lie groups $\g^0$ and $\g$}
\label{det-centr-g}
\subsection{The determinant central extension of Lie group $\g^0$}  \label{det-g0}
\begin{Th}  \label{th-plmi}
For any $g \in \G$ the linear operators $g_{+-}$ and  $g_{-+}$ (see formula~\eqref{block-mat})  belong to  $\Homb  ({\mathcal \h_+}, {\mathcal \h_-}) $ and ${\Homb}  ({\mathcal \h_-}, {\mathcal \h_+})$ correspondingly.
\end{Th}
\begin{proof}
Consider  $g \in \G$.
We prove that $g_{+-}$ belongs to ${\Homb}  ({\mathcal \h_+}, {\mathcal \h_-}) $.
Let $h$ be from~$\h_+$.

It is easy to see that  there is an open set $V = \widehat{\C} \setminus \overline{D}_{1 - 1/n} $ (with $n > 1$) that depends only on $g$, and there is an open set $U =  D_{1 + 1/m} $ (with $m>1$) that depends on $g$ and $h$ such that
$$
g \diamond h    \, \in \, \h(U \cap V)  \, \mbox{.}
$$

We have $\widehat{\C} = U \cup V$. Therefore  by formula~\eqref{hol-dec-wl} from the proof of Proposition~\ref{dec-gamma}
we have the unique decomposition  $g \diamond h = w_+ + w_-$, where $w_+ \in \h(U)$ and $w_- \in \h_0(V)$. Since $V$ does not depend on $h$, the operator $g_{+-}$ belongs to ${\Homb}  ({\mathcal \h_+}, {\mathcal \h_-}) $.

\quash{\h = \varinjlim_{n \in \N} \tilde{\h} \left(A_{1/n} \right) \, \mbox{,} \qquad  \h_-=  \varinjlim_{n \in \N} \tilde{\h}_0 \left(\widehat{\C} \setminus \overline{D}_{1 - 1/n} \right)  \, \mbox{,}
\qquad
\h_+ = \varinjlim_{n \in \N} \tilde{\h} \left( D_{1 + 1/n } \right)   \, \mbox{,}}

The case of $g_{-+}$ is proved analogously.

\end{proof}

\begin{prop}
The group $\g^0$ is a subgroup of the group $
{\mathop{ \rm GL}}^0_{\rm res}(\h) $.
\end{prop}
\begin{proof}
This follows from Theorem~\ref{th-block} and Theorem~\ref{th-plmi}, and from the construction of the group $
{\mathop{ \rm GL}}^0_{\rm res}(\h) $ (see also formula~\eqref{plmi}).
\end{proof}

\begin{defin}  \label{centr-det-g0}
The pullback of the determinant central extension via the embedding ${\g^0  \hookrightarrow {\mathop{ \rm GL}}^0_{\rm res}(\h)}$ is called the determinant central extension of the group $\g^0$. We denote this central extension as
\begin{equation}  \label{g-det}
1 \lrto \C^*  \lrto \widetilde{\g^0}  \lrto \g^0 \lrto 1  \, \mbox{.}
\end{equation}
\end{defin}

We give the properties of this central extension.
\begin{Th}  \label{th-3}
There is a natural (non-group) section $\sigma : \g^0  \to \widetilde{\g^0}$ of the map
$\widetilde{\g^0} \to \g^0 $
in sequence~\eqref{g-det}. For any $g_1$ and  $g_2$ from $\g^0$ we have
\begin{gather}
\sigma(g_1) \sigma(g_2) = D(g_1, g_2) \sigma(g_1 g_2)  \, \mbox{,} \nonumber \\  \label{2-coc}
\mbox{where} \quad D(g_1, g_2) = \det\left((g_1)_{++} (g_2)_{++} ((g_1g_2)_{++})^{-1}\right)  \, \in \, \C^*  \, \mbox{.}
\end{gather}
The determinant in~\eqref{2-coc}  is well-defined, since it is applied to the element from the set ${1 + {\Homb}  ({\mathcal \h_+}, {\mathcal \h_+}) }$. Besides, $D(\cdot, \cdot)$ is a smooth  $2$-cocycle on the group $\g^0$ with values in the trivial $\g^0$-module $\C^*$. Therefore $\widetilde{\g^0}$ is a smooth Lie group modelled on Silva spaces, whose topological space (without the group structure) is isomorphic to $\C^* \times \g^0$.
\end{Th}
\begin{proof}
Consider the group $\g^0$ as the subgroup of the group ${\mathop{ \rm GL}}^+_{\rm res}(\h)$.
The determinant central extension was constructed from the exact sequence of groups~\eqref{exact-seq}. By Theorem~\ref{th-block},
over the subgroup $\g^0$ the map $\theta$ from~\eqref{exact-seq} has a natural (non-group) section
$$
\g^0 \ni g \, \longmapsto  \, (g,g_{++})  \in {\mathcal E}  \, \mbox{.}
$$
This section defines the (non-group) section $\sigma : \g^0  \to \widetilde{\g^0}$ of the map
$\widetilde{\g^0} \to \g^0 $
in sequence~\eqref{g-det}. By construction, the section $\sigma$ gives the  $2$-cocycle $D$
on the group $\g^0$ with values in the trivial $\g^0$-module $\C^*$
 given by the explicit formula~\eqref{2-coc}.

Now consider the $2$-cocycle $D$ in more detail. We note that
$$
(g_1 g_2)_{++} = (g_1)_{-+} (g_2)_{+-} + (g_1)_{++} (g_2)_{++}  \, \mbox{.}
$$
Therefore we obtain
\begin{multline} \label{form-2-coc}
(g_1)_{++} (g_2)_{++} ((g_1g_2)_{++})^{-1} = \left(  \left((g_1)_{-+} (g_2)_{+-} + (g_1)_{++} (g_2)_{++} \right) \left( (g_1)_{++} (g_2)_{++} \right)^{-1}    \right)^{-1} =   \\  =
\left( 1 + (g_1)_{-+} (g_2)_{+-} \left(  (g_1)_{++} (g_2)_{++}  \right)^{-1}  \right)^{-1} =  \\  =
\left( 1 + (g_1)_{-+} (g_2)_{+-} \left( (g_2)_{++}  \right)^{-1} \left((g_1)_{++}\right)^{-1}   \right)^{-1}
   \mbox{.}
\end{multline}
Hence, by Theorem~\ref{th-plmi}, formula\eqref{inv-comp} and Section~\ref{determin}, the determinant in formula~\eqref{2-coc} is well-defined.

For any $g \in \g^0$
 we have
 \begin{gather}
 g_{-+} (g^{-1})_{+-} + g_{++} (g^{-1})_{++}  =1  \, \mbox{,} \nonumber \\   \label{form-inverse} \mbox{and hence}  \quad
(g_{++})^{-1} = (g^{-1})_{++} \left(1 - g_{-+} (g^{-1})_{+-} \right)^{-1}  \, \mbox{.}
\end{gather}

\medskip

Now we will obtain from formulas~\eqref{form-2-coc} and~\eqref{form-inverse} that the $2$-cocycle $D$ is a smooth function
as the composition of the smooth functions.

Indeed, let us first note that the map $1+T  \mapsto (1+T)^{-1}$ is a smooth function when ${\det(1 +T )  \ne 0}$ (see Section~\ref{determin}).
Now consider on the   spaces ${\Homb}  ({\mathcal \h_+}, {\mathcal \h_-}) $ and
${\Homb}  ({\mathcal \h_-}, {\mathcal \h_+}) $ the inductive limit topology in the category of Hausdorff locally convex topological vector spaces as in
 Section~\ref{trace} for ${\Homb}  ({\mathcal \h_+}, {\mathcal \h_+}) $. Besides,
 in the category of Hausdorff locally convex topological vector spaces consider
 on the space  ${\Hom}  ({\mathcal \h_+}, {\mathcal \h_+})$
  the  inductive limit topology when the elements of ${\Hom}  ({\mathcal \h_+}, {\mathcal \h_+})$
 are considered as functions $f(z,w)$ by  item~\ref{it-3} of Proposition~\ref{holom},
 and where the index (directed) set of the corresponding  inductive (direct) system is the set of open subsets of $\C \times \widehat{\C}$ containing
 $\overline{D}_1 \times (\widehat{\C} \setminus \overline{D}_1) $, and the spaces in the direct system are the Fr\'echet spaces of holomorphic functions  on these open subsets with the condition as in item~\ref{it-3} of Proposition~\ref{holom} and with the topology of uniform convergence on compact
sets.

Further,
 from the second part of the proof of Proposition~\ref{holom} it follows that
the maps $g \mapsto g_{+-}$, $g \mapsto g_{-+}$,  $g \mapsto g_{++}$from $\g^0$ to  ${\Homb}  ({\mathcal \h_+}, {\mathcal \h_-}) $,
${\Homb}  ({\mathcal \h_-}, {\mathcal \h_+}) $  and ${\Hom}  ({\mathcal \h_+}, {\mathcal \h_+})$ are smooth functions correspondingly (see also the method of the proof of Theorem~\ref{th-plmi}). Finally, the determinant is a smooth function when we apply it to the right hand part of formula~\eqref{form-2-coc} (with the help of formula~\eqref{form-inverse}). Indeed,  we   first consider and apply these formulas to the spaces from  the direct  systems above, and the determinant is a smooth function when these spaces are inserted in formulas~\eqref{form-2-coc} and~\eqref{form-inverse}, since we use the continuous property of the composition (or the convolution) of these spaces and the property from Section~\ref{determin} that the determinant is a smooth function.

\end{proof}

\subsection{The Lie algebra $\mathop{\rm Lie} \g^0$ and the complex Lie algebra $\Lie_{\C} \g^0$}

Using Section~\ref{Expl_descr}, we have the explicit description of the tangent space at the identity element of the Lie group $\g^0$ and, in the standard way, the description of the corresponding Lie algebra
$\Lie \g^0$. This leads to the following propositions (see also the corresponding description in the formal case in~\cite[Prop.~2]{O2}).

\begin{prop}  \label{prop-Lie}
We have a canonical  isomorphism
$$
\Lie \g^0 \simeq \h \rtimes \mathop{\rm Vect}\nolimits_{\rm hol}(S^1) \, \mbox{,}
$$
where  $\h = \h(S^1) $ is an Abelian Lie algebra and  $\mathop{\rm Vect}\nolimits_{\rm hol}(S^1) $ is the Lie algebra of real analytic vector fields on $S^1$. Besides, there is a natural action of the Lie algebra
$\Lie \g^0$ on the space $\h$.
\end{prop}

\begin{nt} \em
The Lie group $\g^0$ is the identity component of the Lie group $\g$. Therefore $\Lie \g = \Lie \g^0$.
\end{nt}

Consider the complexification $\mathop{\rm Vect}\nolimits_{\rm hol}(S^1)_{\C} = \mathop{\rm Vect}\nolimits_{\rm hol}(S^1) \otimes_{\dr} {\C}$ of the Lie algebra $\mathop{\rm Vect}\nolimits_{\rm hol}(S^1)$. We  consider a natural complex Lie algebra
$$
\Lie\nolimits_{\C} \g^0 = \h \rtimes \mathop{\rm Vect}\nolimits_{\rm hol}(S^1)_{\C}  \, \mbox{.}
$$

From Proposition~\ref{prop-Lie} and using that $\h = \h_{\dr} \otimes_{\dr}  \C$  (see  Section~\ref{anal_diff}) and $\frac{d}{dx} = 2 \pi i z \frac{d}{dz}$ when $z = \exp(2 \pi i x)$,  we immediately obtain the following corresponding properties for~$\Lie\nolimits_{\C} \g^0$.

\begin{prop} \label{Lie_alg}
The Lie algebra  $\mathop{\rm Vect}\nolimits_{\rm hol}(S^1)_{\C}$ is naturally isomorphic to the Lie algebra of derivations of type $- r \frac{d}{dz}$, where $r \in \h$, for the commutative associative algebra $\h$.
The Lie bracket in $\Lie\nolimits_{\C} \g^0$ is written as
\begin{equation}  \label{fo1}
\left[ s_1 - r_1   \frac{d}{dz}, s_2 - r_2 \frac{d}{d z} \right]  = \left(r_1 s_2' - r_2 s_1' \right) -  \left(r_1 r_2' - r_2 r_1' \right)\frac{d}{d z}  \, \mbox{,}
\end{equation}
where $s_i, r_i  \in \h$. Besides, the natural action  of the Lie algebra $\Lie\nolimits_{\C} \g^0$ on $\h$ is
\begin{equation}  \label{fo2}
\left(s - r   \frac{d}{d z} \right) \left(h \right) = sh + rh' \, \mbox{,} \qquad \mbox{where} \quad s,r,h \in \h  \, \mbox{.}
\end{equation}
\end{prop}

\begin{nt} \label{rem-sign} \em
Formulas~\eqref{fo1}-\eqref{fo2} differ by  signs from analogous formulas obtained in formal case in~\cite[Prop.~2]{O2} (but the Lie algebra itself will be isomorphic to the new Lie algebra obtained after we change these signs). The reason is that in~\cite{O2} we considered the automorphisms of the ring of functions on the formal punctured disk, and now we consider the diffeomorphisms of $S^1$ which act on functions by formula given before Definition~\ref{d1}.
\end{nt}

We will need a proposition.
\begin{prop}
The Lie algebra $\Lie\nolimits_{\C} \g^0$ is perfect.
\end{prop}
\begin{proof}
We note that $\left[  \Lie\nolimits_{\C} \g^0 , \Lie\nolimits_{\C} \g^0    \right]$ is an $\h$-module, because
$$
h \left[ s_1 - r_1   \frac{d}{d z}, s_2 - r_2 \frac{d}{d z} \right] =
\left[ - h r_1  \frac{d}{d z},  s_2 - \frac{1}{2} r_2 \frac{d}{dz}   \right] +
\left[   s_1 -  \frac{1}{2} r_1   \frac{d}{d z}, - h r_2 \frac{d}{d z}    \right]   \, \mbox{,}
$$
where $h, r_i, s_i \in \h$. Therefore it is enough to show that elements $1$ and $\frac{d}{d z}$ belong to
$\left[  \Lie\nolimits_{\C} \g^0 ,   \Lie\nolimits_{\C} \g^0   \right]$. We have
$$
\frac{d}{d z} = \left [ - \frac{d}{d z}, z \frac{d}{d z}         \right]  \qquad \mbox{and}
\qquad
1 = \left [  -\frac{d}{d z}, z        \right]  \, \mbox{.}
$$

Hence we have
$$
\left[  \Lie\nolimits_{\C} \g^0 , \Lie\nolimits_{\C} \g^0     \right] = \Lie\nolimits_{\C} \g^0  \, \mbox{.}
$$
\end{proof}

Let $\overline{\g^0}$ be the universal covering Lie group of the Lie group $\g^0$. Clearly, ${\Lie \overline{\g^0} = \Lie {\g^0} }$.

\begin{Th}  \label{smooth}
Any smooth homomorphism from the  Lie group $\g^0$ or from the Lie group $\overline{\g^0}$ to the Lie group $\C^*$ is trivial.
\end{Th}
\begin{proof}
Since the natural homomorphism $\overline{\g^0}  \to \g^0$ is surjective, it is enough to prove the statement of the theorem for $\overline{\g^0}$.

Let $\tau$ be a smooth homomorphism from $\overline{\g^0}$ to $\C^*$.

For any point $v \in \overline{\g^0}$ we will prove that the differential $(d \tau)_v$ is zero as the map from the tangent space $T_v \overline{\g^0}$
to the tangent space $T_{\tau(v)} \C^*$, where we identify the last tangent space with $\C$ via embedding $\C^* \subset \C$.

If $v =e$ is the identity element of $\overline{\g^0}$, then $(d \tau )_e$ is the homomorphism from the Lie algebra $\Lie \g^0$ to the Abelian Lie algebra $\C$. The homomorphism $(d \tau )_e$ restricted to
the Lie algebra  $\mathop{\rm Vect}\nolimits_{\rm hol}(S^1)$ can be uniquely extended by $\C$-linearity to the Lie algebra  $\mathop{\rm Vect}\nolimits_{\rm hol}(S^1)_{\C}$.
With this extension, the  homomorphism $(d \tau )_e$ can be uniquely extended to the $\dr$-linear homomorphism from the Lie algebra $\Lie_{\C} \g^0$ to the Lie algebra $\C$. But the last homomorphism will be zero, since $
\left[  \Lie\nolimits_{\C} \g^0 , \Lie\nolimits_{\C} \g^0     \right] = \Lie\nolimits_{\C} \g^0
$  and $\left[ \C, \C  \right] = 0$. Hence $(d \tau )_e$ is zero.

Now for any $v \in \overline{\g^0}$ we have $(d \tau)_v = 0$, since $\tau$ is the homomorphism and
 therefore the following diagram is commutative
 $$
 {\xymatrix{
   {T_e \overline{\g^0}} \ar[d]_{(d \tau)_e}  \ar[rr]^{\left(d r_v\right)_e} &&  {T_v \overline{\g^0}} \ar[d]^{(d \tau)_v}  \\
   T_{1} {\C^*} \ar[rr]_{\left(d r_{\tau(v)}\right)_1}   && T_{\tau(v)} {\C^*}
 }}
 $$
 where $r_v$ or $r_{\tau(v)}$ is the multiplication on the right by the element $v$ or $\tau(v)$,
 and the horizontal arrows are isomorphisms.

Hence $\tau$ is a locally constant map, and since $\overline{\g^0}$ is connected and $\tau(e) =1$, we have that $\tau $ is trivial.

\end{proof}

\subsection{The determinant central extension of Lie group $\g$}
\label{ext-det}

\begin{prop}  \label{semi}
There is a canonical group decomposition
$$
\g = \g^0 \rtimes \Z  \, \mbox{,}
$$
where  we consider the composition of embeddings $\Z \hookrightarrow \h^*  \hookrightarrow \g$, in which the first embedding is given by  decomposition~\eqref{dec1}.
\end{prop}
\begin{proof}
Consider $(d,f) \in \g^0$, where $d \in \h_0^*$, $f \in \DS$. In the group $\g$ we have
$$
 (z,1) (d, f) (z, 1)^{-1} =  (z \, d, f)(z^{-1}, 1)  = (z \, d  \, (f^{\circ -1})^{-1}, f)
\, \mbox{,}
$$
where $(z \, d  \, (f^{\circ -1})^{-1}, f) \in \g^0$, since $\nu(z \, d  \, (f^{\circ -1})^{-1})=0$.
\end{proof}

\bigskip

We will construct the determinant central extension of Lie group $\g$  by $\C^* $ which, restricted to its connected component $\g^0$, coincides with the determinant central extension of  $\g^0$ constructed in Section~\ref{det-g0}.

We will always suppose that a central extension of a Lie group by another Lie group is a principal bundle.

We recall (cf.~\cite[1.7. Construction]{BD}) that   a central extension $\widetilde{V}$ of a Lie group  ${V = V_1 \rtimes V_2}$
by  a Lie group  $Q$ is equivalent to the following data:
\begin{itemize}
\item[1)]   a central extension of $V_2$ by $Q$;
\item[2)]   a central extension $\widetilde{V_1}$ of $V_1$ by $Q$;
\item[3)] an action of the Lie group $V_2$ on the Lie  group  $\widetilde{V_1}$, lifting the action of $V_2$ on $V_1$ and  trivial on $Q$.
\end{itemize}

\medskip

We have $\g = \g^0 \rtimes \Z$. Clearly, any central extension of the discrete group $\Z$ is trivial. Therefore, to
extend the determinant central extension from $\g^0$ to $\g$, it is enough to  lift the action of $1 \in \dz$   on $\g^0$
to  an action on   $\widetilde{{\g^0}}$.
 The action of $1 \in \dz  $ on $\h$
is the multiplication by the element $z$, and the action on $\g^0 $ is by means of the conjugation by $z$ (see also the proof of Proposition~\ref{semi}).

Now we consider the following map on  the subgroup $\theta^{-1}(\g^0)$ of the group $\mathcal E$ in sequence~\eqref{exact-seq}:
\begin{equation}  \label{lift}
(g,r)  \longmapsto (z g z^{-1}, r_z ) \, \mbox{,}
\end{equation}
where $r_z \, :  \, \h_+  \to  \h_+$ equals to ${\rm id}  \oplus z r z^{-1} $ for the decomposition  $\h_+ = \C \oplus z \h_+$.

The map~\eqref{lift} is an injective endomorphism of the group $\theta^{-1}(\g^0)$, but this endomorphism is not surjective. But it is easy to see that the induced endomorphism of the group $ \widetilde{{\g^0}}$ is an automorphism.
Moreover, from the description of the Lie group  $ \widetilde{{\g^0}}$ given in Theorem~\ref{th-3} it follows that this automorphism is a Lie group automorphism.
This automorphism is the lift of the action of  $1 \in   \dz$ to an action on the Lie group $\widetilde{{\g^0}}$.

We call this central extension {\em the determinant central extension} of the group $\g$:
\begin{equation}  \label{det-centr}
1 \lrto \C^*  \lrto \widetilde{\g}  \lrto \g \lrto 1  \, \mbox{.}
\end{equation}

We have the following uniqueness result.

\begin{Th}  \label{Th5}
The determinant central extension~\eqref{det-centr}
is a unique (up to isomorphism) Lie group central extension of $\g$ by $\C^*$ such that its restriction to $\g^0$ coincides with the central extension from Definition~\ref{centr-det-g0}.
\end{Th}
\begin{proof}
We claim that a lift of any automorphism of the Lie group $\g^0$ to the Lie group $\widetilde{{\g^0}}$ with the identity action on $\C^*$ is unique (after the discussion before the theorem,  it is enough to prove only this statement). Indeed,  any two possible such lifts will differ by the Lie group
automorphism of $\widetilde{{\g^0}}$ that induces the identity action on $\C^*$ and  $\g^0$. But any such automorphism is identified with the smooth
 homomorphism from $\g^0$ to $\C^*$, and by Theorem~\ref{smooth} this homomorphism is trivial.

Now the theorem follows
from~Proposition~\ref{semi} and the above discussion on central extensions of semidirect products.
\end{proof}

\begin{nt} \label{sect-g} \em
Form the construction it follows that the topological space of $\widetilde{\g}$ (without the group structure) is isomorphic to $\C^* \times \g$.
\end{nt}

\section{Explicit $2$-cocycles on $\g$}
\label{expl-2-coc}

\subsection{Bimultiplicative pairing}  \label{bimult}

Following the papers of A.~A.~Beilinson and P.~Deligne, see~\cite{Be} and~\cite[\S~2.7]{D2}, we consider the following pairing.

For any two  functions $f$ and  $g$ from
from the group $C^{\infty}(S^1, \C^*)$ of smooth functions from
$S^1 $ (where $S^1$ is considering in $\C$ as in Section~\ref{curve}) to $\C^*$, we consider  ${\mathbb T}(f,g)$ from $\C^*$:
\begin{equation}  \label{expli}
{\mathbb T} (f,g) = \exp \left( \frac{1}{2 \pi i} \int_{x_0}^{x_0}  \log f \,  \frac{dg}{g} \right) g(x_0)^{-\nu(f)}  \, \mbox{,}
\end{equation}
where $x_0$ is any point on $S^1$, $\log f $ is any branch of the logarithm on $S^1 \setminus x_0$, and the integral is on $S^1$ from $x_0$ to $x_0$ in the counterclockwise direction.
The complex number ${\mathbb T}(f,g)$
does not depend on the choice of $x_0$ and the branch $\log f$ of the logarithm of $f$.

The pairing
$${\mathbb T}  \;  : \; C^{\infty}(S^1, \C^*) \times C^{\infty}(S^1, \C^*) \lrto \C^*$$
is bimultiplicative and antisymmetric. Moreover, from formula~\eqref{expli} it follows that the pairing $\mathbb T$ is invariant under the diagonal action of the group ${\mathop{\rm Diff}}^+(S^1)$ of orientation preserving  smooth diffeomorphisms  of $S^1$ on  $C^{\infty}(S^1, \C^*) \times C^{\infty}(S^1, \C^*)$.

\begin{nt}  \em
For any $f$ from $C^{\infty}(S^1, \C^*)$ such that $1 -f$ is from $C^{\infty}(S^1, \C^*)$ the Steinberg relation is satisfied:
$$
{\mathbb T}(f,1-f) =1  \, \mbox{.}
$$
Indeed, one of the definitions of the pairing $\mathbb T$ is that the complex number  ${\mathbb T}(f,g)$ is the image of the element $1 \in \dz = \pi_1(S_1)$ (considered in the counterclockwise direction) under the  monodromy of the connection on the line bundle on $S^1$ constructed by ${f, g \in C^{\infty}(S^1, \C^*)}$, see  more in Section~\ref{Del-coh}  below. Then such a line bundle with  a connection constructed by $f$ and $1-f$, where $f, 1-f \in C^{\infty}(S^1, \C^*)$, will have a non-vanishing section which is horizontal with respect to the connection, see~\cite[Exemp.~3.5]{D2}
and~\cite[Prop.~(1.13), Corol.~(1.15)]{Bl} (in the last reference $f$ and $g$ are holomorphic functions on a Riemann surface, but the reasoning  we need work in our case too.)

Therefore the pairing $\mathbb T$ defines the homomorphism from the group $K_2^M(C^{\infty}(S^1, \C^*))$ to the group  $\C^*$, where the Milnor $K_2$-group $K_2^M(A)$ of any commutative ring $A$ is defined as
$$
K_2^M(A) = A^* \otimes_{\dz} A^* / St  \, \mbox{,}
$$
where $A^*$ is the group of invertible elements of the ring $A$, and  the subgroup of Steinberg relations $St \subset A^* \otimes_{\dz} A^*$ is generated by all elements $a \otimes (1-a)$ with $a$ and $1-a$ from~$A^*$.
\end{nt}

\begin{nt}  \em
Restrict the pairing $\mathbb T$ to $\h^*  \times \h^*$. Using decomposition~\eqref{dec1}, we see that this restriction is uniquely defined by the  bimultiplicative, antisymmetric and the following two properties
\begin{equation}  \label{form-CC}
{\mathbb T}(z,z)= -1 \, \mbox{,}  \qquad  {\mathbb T}(f,g)= \exp \frac{1}{2 \pi i} \oint_{S^1} \log f \, \frac{dg}{g} \, \mbox{,} \quad \mbox{where} \quad f \in \h_0^* \, \mbox{,} \quad g \in \h^* \, \mbox{.}
\end{equation}

Clearly, ${\mathbb T} \, : \, \h^* \times \h^*  \to \C^*$ is a smooth map, where the topology on $\h^* \times \h^*$ is the product topology.
\end{nt}

\begin{nt} \em
There is the formal (in algebraic terms) analog of the pairing $\mathbb T$, called the Contou-Carr\`{e}re symbol, see~\cite[\S~2.9]{D2},   \cite{CC1}, \cite[\S~2]{OZ}. There are also the higher-dimensional generalizations of
the Contou-Carr\`{e}re symbol, see~\cite{OZ}, \cite{GO}.
\end{nt}

\subsection{$\cup$-products and $2$-cocycles}   \label{cup-sect}
We construct explicit smooth  group $2$-cocycles on the group $\g$ with coefficients in the group $\C^*$.

Let $\Gamma$ be a group, and  $N$  be a $\Gamma$-module. We will consider the multiplicative notation for the group law in $N$.

Recall (see, e.g., \cite[ch.~IV-V]{Bro}) that
a function $\lambda$ from $\Gamma$ to $N$ is called
 a $1$-cocycle~if
 $$
 \lambda (b_1 b_2) = \lambda(b_1) \, b_1 (\lambda(b_2))  \, \mbox{,} \quad \mbox{where} \quad b_1, b_2 \in \Gamma  \, \mbox{.}
 $$

For any two $1$-cocycles $\lambda_1$ and $\lambda_2$ from $\Gamma$ to $N$, a $2$-cocycle $\lambda_1 \cup \lambda_2$ on $\Gamma$ with coefficients in $N \otimes_{\dz} N$ is defined,
$$
\lambda_1  \cup \lambda_2   \, : \, \Gamma \times \Gamma \lrto N \otimes_{\dz} N  \, \mbox{,}
$$
where  $\Gamma$  acts diagonally on $N \otimes_{\dz} N$.
For any elements $b_1 $ and $b_2$ from $\Gamma$, the value of the $2$-cocycle $\lambda_1  \cup \lambda_2$ on the pair $(b_1, b_2)$ is given by the rule
$$
(\lambda_1  \cup \lambda_2)(b_1, b_2)= \lambda_1(b_1)  \otimes b_1 (\lambda_2(b_2))   \, \mbox{.}
$$

The $\cup$-product induces the well-defined homomorphism of cohomology groups (see also more on group cohomology in Section~\ref{gerbes} below):
$$
\cup \; : \; H^1(\Gamma, N) \otimes_{\dz} H^1(\Gamma, N)  \lrto H^2(\Gamma, N \otimes_{\dz} N)  \, \mbox{.}
$$

\bigskip

Recall that the group $\h^*$ is a $\DS$-module (see Section~\ref{App}). Hence, using the natural homomorphism $\g \to \DS$, we have that the group $\h^*$ is also a $\g$-module.

The map $\mathbb T$ is bimultiplicative and
invariant under the diagonal action of the group ${\mathop{\rm Diff}}^+(S^1)$. Therefore
this map  induces the homomorphism of $\g$-modules $\h^* \otimes_{\dz}  \h^*  \to \C^*$, where $\g$ acts diagonally on  $\h^* \otimes_{\dz}  \h^*$, and $\C^*$ is the trivial $\g$-module.
We denote this homomorphism by the same letter $\mathbb T$. Hence
 the following definition is correct.
\begin{defin}  \label{bracket}
For any two $1$-cocycles $\lambda_1$ and $\lambda_2$ on the group $\g$ with coefficients in the $\g$-module $\h^*$ define the $2$-cocycle
$$
\langle \lambda_1, \lambda_2   \rangle  = {\mathbb T} \circ (\lambda_1 \cup \lambda_2)
$$
on the group $\g$ with coefficients in the trivial $\g$-module $\C^*$.
Here $\circ$ means the composition of maps  $\lambda_1 \cup \lambda_2 \, : \, \g \times \g \to \h^* \otimes_{\dz}  \h^* $ and $ {\mathbb T} \, :  \, \h^* \otimes_{\dz}  \h^*  \to \C^* $.
\end{defin}

Clearly, if in Definition~\ref{bracket} the $1$-cocycles $\lambda_1$ and $\lambda_2$ are smooth, then the $2$-cocycle $\langle \lambda_1, \lambda_2 \rangle$ is smooth as the map from $\g \times \g$ to $\C^*$.

From the chain rule for the derivative of composition of two functions it is easy to see that the map
$$
\DS \lrto \h^* \; : \; f \longmapsto (f^{\circ -1})' = \frac{1}{f' \circ f^{\circ -1}}
$$
is a smooth $1$-cocycle on the group $\DS$ with coefficients in the $\DS$-module $\h^*$ (where, we recall, ${}'$ means $\frac{d}{dz}$).

Therefore the map
$$
\Omega \; : \; \g \lrto \h^* \, \mbox{,}  \qquad \Omega((d,f)) = (f^{\circ -1})' \, \mbox{,}  \quad \mbox{where} \quad d \in \h^*  \, \mbox{,} \quad f \in \DS
$$
is a {\em smooth  $1$-cocycle} on the group $\G$ with coefficients in the $\G$-module $\h^*$.

Besides, since $\g = \h^* \rtimes \DS$, the map
$$
\Lambda \; : \; \g \lrto \h^* \, \mbox{,}  \qquad \Lambda((d,f))= d \, \mbox{,}  \quad \mbox{where} \quad d \in \h^*  \, \mbox{,} \quad f \in \DS
$$
is also a {\em smooth  $1$-cocycle} on the group $\G$ with coefficients in the $\G$-module $\h^*$.

\medskip
Thus, by definition~\ref{bracket} we have the following {\em smooth $2$-cocycles} on the group $\g$ with coefficients in the trivial $\g$-module $\C^*$:
$$
\langle \Lambda, \Lambda  \rangle  \, \mbox{,} \qquad
\langle   \Lambda , \Omega                 \rangle  \, \mbox{,} \qquad
\langle   \Omega, \Omega                                \rangle     \, \mbox{.}
$$

\begin{nt} \label{Rem-coh} \em
Consider the Abelian group $H_{\rm sm}^2(\g, \C^*)$ that is the quotient group of the group of smooth $2$-cocycles by the subgroup of smooth $2$-coboundaries (we recall that $\C^*$ is the trivial $\g$-module). Then the elements  of the group $H_{\rm sm}^2(\g, \C^*)$
are in one-to-one correspondence with equivalence classes of central extensions of the Lie group $\g$  by the Lie group $\C^*$:
 \begin{equation}   \label{ext-centr}
 1 \lrto \C^*  \lrto \widehat{\g} \lrto \g \lrto 1
 \end{equation}
such that the topological space of $\widehat{\g}$ (without the group structure) is isomorphic to the topological space $\C^{*} \times \g$, and the group structure on the set $\C^{*} \times \g$ is defined in this case as
$$
(c_1, g_1)(c_2, g_2)= (E(g_1, g_2) \, c_1 c_2 , \, g_1 g_2) \, \mbox{,}
$$
where $c_i \in {\mathbb C}^*$, $g_i \in \g$ and $E$ is the corresponding $2$-cocycle.
\end{nt}

\section{Decomposition of determinant central extension}
\label{dec-th}
\subsection{Lie algebra central extensions}
To a central extension of Lie groups~\eqref{ext-centr} one associates the central extension of corresponding tangent Lie algebras. A corresponding $2$-cocycle $E$ on $\G$ with coefficients in $\C^*$
defines the  $2$-cocycle $\Lie E$  on the Lie algebra $\Lie \g $ with the coefficients in $\C$ by a formula (see, e.g., \cite[Ch.~I, Prop.~3.14]{KW}):
\begin{equation}  \label{Lie-coc}
\Lie E (X, Z) = \left. \frac{d^2}{dt \, ds} \right |_{ t=0, s=0} \left( E(g_t, h_s) \, - E(h_s, g_t) \right)  \, \mbox{,}
\end{equation}
where  $X$ and  $Z$ are from   $\Lie \g = \Lie \g^0$, and  $g_t$ and $h_s$ are smooth curves in $\g^0$ such that $\left. \frac{d}{dt} \right |_{t=0} g_t = X$ and $ \left. \frac{d}{ds} \right |_{s=0} h_s = Z$.

This description leads to the calculations of the Lie algebra $2$-cocycles.

For any $X \in \Lie_{\C} \g^0$
consider the block matrix
$$
\begin{pmatrix}
  X_{--} & X_{+-} \\
  X_{-+} & X_{++}
\end{pmatrix}
$$
with respect to the action of $X$ on $\h$ (see formula~\eqref{fo2}) and the  decomposition ${\h = \h_- \oplus \h_+}$ as in~\eqref{block-mat}.

By the same reasoning as in the proof of Theorem~\ref{th-plmi}, the linear operators $X_{+-}$ and  $X_{-+}$  belong to the spaces  $\Homb  ({\mathcal \h_+}, {\mathcal \h_-}) $ and ${\Homb}  ({\mathcal \h_-}, {\mathcal \h_+})$ correspondingly.

The proof of the following proposition is standard (see~\cite[Prop.~8]{O2} in the formal case and \cite[Prop.~6.6.5]{PS} in the case of Hilbert spaces).
\begin{prop}
Recall that $D$ is a $2$-cocycle on $\g^0$ with values in $\C^*$, see formula~\eqref{2-coc}. For any $X$ and  $Y$  from   $ \Lie \g^0$ we have
\begin{equation}  \label{LieD}
\Lie D (X,Z) = \tr \left( Z_{-+} X_{+-} - X_{-+} Z_{+-} \right)  \, \mbox{.}
\end{equation}
\end{prop}

From formula~\eqref{LieD} we see that the $2$-cocycle $\Lie D$ can be extended by $\C$-linearity to the $2$-cocycle on the Lie algebra $\Lie_{\C} \g^0$. We denote this extension by the same notation $\Lie D$. It is given by the same formula~\eqref{LieD}.

\begin{prop}
The $2$-cocycles $\Lie \langle \Lambda, \Lambda  \rangle $, $\Lie \langle   \Lambda  , \Omega                 \rangle$ and $\Lie \langle   \Omega, \Omega                                \rangle  $
can be extended by $\C$-linearity to the $2$-cocycles on the Lie algebra $\Lie_{\C} \g^0$ with values in $\C$. These extended $2$-cocycles
 have the following explicit formulas (in the notation of Proposition~\ref{Lie_alg}) for $s_i, r_i \in \h$:
\begin{gather}
\label{first-form}
\Lie \langle \Lambda, \Lambda  \rangle \left(s_1 - r_1 \frac{\partial}{\partial t}, s_2 - r_2 \frac{\partial}{\partial t} \right) \, = \,  \frac{1}{ \pi i} \oint_{S^1} s_1 d s_2 \, \mbox{,}\\
\label{sec-form}
\Lie \langle \Lambda, \Omega  \rangle \left(s_1 - r_1 \frac{\partial}{\partial t}, s_2 - r_2 \frac{\partial}{\partial t} \right) \, = \, \frac{1}{2 \pi i} \oint_{S^1}
s_1 d r_2' - s_2 dr_1'  \, \mbox{,} \\
\label{third-form}
 \Lie \langle \Omega, \Omega  \rangle \left(s_1 - r_1 \frac{\partial}{\partial t}, s_2 - r_2 \frac{\partial}{\partial t} \right) \, = \,   \frac{1}{ \pi i} \oint_{S^1}  r_1' d r_2'  \, \mbox{.}
\end{gather}
\end{prop}
\begin{proof}
By the results of Section~\ref{Expl_descr}, we have an explicit description of  the Lie group
${ \mathcal G}^0$ and its the tangent space at the identity element.
 Besides, we have   also $ \frac{d}{dx} = 2 \pi i z  \frac{d}{dz}$ when $z = \exp(2 \pi i x)$, and for $\varphi \in \h_{\dr}$ (see Section~\ref{anal_diff}), $t \in \dr$ we have a smooth curve in the Lie group $\DS$:
$$
\exp(2 \pi i (x +  \varphi t))= z \exp(2 \pi i t \varphi)= z(1 + 2 \pi i \varphi t - 2 \pi^2 \varphi^2 t^2 + \ldots)= z + 2 \pi i z \varphi t + \ldots   \, \mbox{.}
$$

Now, with the help of formulas~\eqref{Lie-coc}  and~\eqref{form-CC}, the proof of this proposition is the same as in the formal case in~\cite[Prop. 7]{O2} and~\cite[Prop.~4.1]{O1}.

We only note  that the difference with the mentioned formal case is
that in the left hand sides of formulas~\eqref{first-form}-\eqref{third-form} we have now the minus signs (see Remark~\ref{rem-sign}), and in the right hand sides of these formulas we have $\frac{1}{2 \pi i} \oint_{S^1}$ instead of the residue in the formal case.

\end{proof}

\begin{Th} \label{main-Lie}
We have an equality between  $2$-cocycles on the Lie algebra $\Lie_{\C} \G^0$ with coefficients in $\C$
\begin{equation}  \label{Lie-RR}
12 \Lie D = 6 \Lie \langle \Lambda, \Lambda \rangle  - 6 \Lie \langle \Lambda, \Omega \rangle  + \Lie \langle \Omega, \Omega \rangle   \, \mbox{.}
\end{equation}
\end{Th}
\begin{proof}
The $2$-cocycles given by formulas~\eqref{first-form}-\eqref{third-form} are continuous in each argument, where we consider  the product topology on $\Lie_{\mathbb C} \g^0$, whose topological space is $\h \times \h$.
Besides, the $2$-cocycle $\Lie D$ is continuous in each argument, since in formula~\eqref{LieD} each of the  pairings  $\tr \left( Z_{-+} X_{+-}  \right)$  and $\tr \left( X_{-+} Z_{+-}  \right)$  is continuous in each argument.

Since the linear span of the set of all elements $z^l$ and $- z^{j+1} \frac{d}{dz}$ is dense in $\Lie_{\C} \g^0 $, it is enough to check~\eqref{Lie-RR} on the elements from this set. Using that the  cocycles on Lie algebras
 are antisymmetric, using explicit formulas \eqref{LieD}-\eqref{third-form}, this check is the same as in the formal case in the proof of Theorem~5 from~\cite{O2}.

\end{proof}

\begin{nt}  \label{Lie-dim} \em
 By the similar reasoning as in the proof of Proposition~(2.1) from~\cite{ADKP} we can see that the continuous Lie algebra cohomology $H^2_{\rm c}(\Lie_{\C} \g^0, \C)$ is a three-dimensional vector space over $\C$ with the basis given by the {$2$-cocycles} from
formulas~\eqref{first-form}-\eqref{third-form}. Therefore formula~\eqref{Lie-RR} gives the decomposition of the $2$-cocycle $\Lie D$ with respect to this basis in $H^2_{\rm c}(\Lie_{\C} \g^0, \C)$.
\end{nt}

\subsection{Equivalence of Lie group central extensions}
Recall that  the group ${\rm PSL}(2, \dr)$ is the group of holomorphic automorphisms  of  the unit disk with the  following action.

The group ${\rm PSL}(2, \dr)$ acts on the upper half plane in $\C$ by holomorphic automorphisms:
$$
z \longmapsto \frac{az + b }{cz +d }   \, \mbox{,} \qquad  \mbox{where}  \quad
\begin{pmatrix}
  a & b \\
  c & d
\end{pmatrix}  \in {\rm PSL}(2, \dr)  \, \mbox{.}
$$

The Caley map $z \mapsto \frac{z -i}{z +i}$ gives the holomorphic isomorphism between the upper half plane and the unit disk. Via this isomorphism the group ${\rm PSL}(2, \dr)$ goes to the group ${\rm PSU}(1,1)$ of matrices of form
$ \sigma^{-1} =
 \begin{pmatrix}
  \alpha & \beta \\
  \overline{\beta} & \overline{\alpha}
\end{pmatrix}
$ (factored modulo the subgroup $\{ \pm 1 \}$), where $\alpha, \beta \in \C$ and  $\alpha \overline{\alpha} - \beta \overline{\beta} =1 $. This group preserves the unit circle $S^1$ and acts on $\h$, preserving $\h_+$:
\begin{equation}  \label{act-SL}
\sigma(h)= h \left( \frac{ \alpha z + \beta}{\overline{\beta} z + \overline{\alpha} } \right) \, \mbox{,} \qquad \mbox{where} \quad h \in \h \, \mbox{.}
\end{equation}

\begin{lemma}  \label{sl}
The $\C$-linear span of
the image of the Lie algebra $sl(2,\dr)$ in $\mathop{\rm Vect}\nolimits_{\rm hol}(S^1)_{\C}$ under the map induced by the action of ${\rm PSL}(2, \dr)$ on $S^1$ is the Lie algebra $sl(2, \C)$  generated over $\C$ by $z^j \frac{d}{dz}$, where $0 \le j \le 2$,
with the notation from Proposition~\ref{Lie_alg}.
\end{lemma}
\begin{proof}
The Lie algebra of the group ${\rm PSU}(1,1)$ is generated by the matrices:
$$
J_1 = \begin{pmatrix}
  i & 0 \\
  0 & -i
\end{pmatrix} \, \mbox{,} \qquad
J_2 = \begin{pmatrix}
  0 & 1 \\
  1 & 0
\end{pmatrix} \, \mbox{,} \qquad
J_3 = \begin{pmatrix}
  0 & i \\
  -i & 0
\end{pmatrix} \, \mbox{.}
$$
Now by direct calculations with formula~\eqref{act-SL} we have
\begin{gather*}
(\exp(tJ_1)(h))(z)= h(z) - 2izf'(z)t + \ldots \\
  (\exp(tJ_2)(h))(z)= h(z) + (z^2 -1) f'(z)t + \ldots  \\
  (\exp(tJ_3)(h))(z)= h(z) -i(z^2 +1) f'(z)t + \ldots \, \mbox{.}
\end{gather*}
\end{proof}

\begin{defin}
Define the Lie subgroup $$\g_+ = \h^*_+  \rtimes {\rm PSL}(2, \dr)$$ of the Lie group $\g^0$, where $\h^*_+ = \C^* \times \h^*_{+, 1}$ is the subgroup of $\h^*_0$ (see formula~\eqref{dec2}).
\end{defin}

Through the action of the Lie group $\g^0$ on $\h$ we have the action of the Lie group $\g_+$ on $\h$ that preserves the subspace $\h_+$.

\begin{defin}  \label{prin-bun}
For any Lie group $\mathcal K$  denote   by $\mathop{\rm Ext} ({\mathcal K}, \C^*)$ the Abelian group
of equivalence classes of central extensions
$$
1 \lrto \C^* \lrto \widetilde{ \mathcal K}  \lrto {\mathcal K}  \lrto 1
$$
of ${\mathcal K}$ by $\C^*$,  where (we always suppose that) the Lie group $\widetilde{ \mathcal K}$ is a principal $\C^*$-bundle, and the group structure in $\mathop{\rm Ext} ({\mathcal K}, \C^*)$ is given by the Baer sum of central extensions.
\end{defin}

Let $\Lie {\mathcal K}$ be the Lie algebra of a Lie group ${\mathcal K}$.
Clearly, we have the natural homomorphism from the group $\mathop{\rm Ext} ({\mathcal K}, \C^*)$ to the $\C$-vector space $H^2_{\rm c}(\Lie {\mathcal K}, \C)$  of the continuous Lie algebra cohomology,
where this vector space classifies the central extensions
$$
0 \lrto \C \lrto \widetilde{\Lie {\mathcal K}} \lrto \Lie {\mathcal K} \lrto 0
$$
of the topological Lie algebra $\Lie {\mathcal K}$  by the Abelian Lie algebra $\C$
 for which there exists  a continuous linear section from $\Lie {\mathcal K}$ to $\widetilde{\Lie {\mathcal K}}$.

\begin{Th}  \label{Th7}
Any element from the group $\mathop{\rm Ext} (\g^0, \C^*)$
is uniquely defined by a pair that consists of  the  image of the element in  $H^2_{\rm c}(\Lie {\g^0}, \C)$
and the restriction of the element  to $\mathop{\rm Ext} (\g_+, \C^*)$.
\end{Th}
\begin{proof}
Let $\mathcal K$ be  any connected Lie group  (modelled on a locally convex  topological vector space),
and $\overline{ \mathcal K}$ be its universal covering group.

There is the following exact sequence:
\begin{equation}  \label{monodr}
\Hom(\overline{ \mathcal K}, \C^*)  \lrto \Hom( \pi_1( {\mathcal K}) , \C^* )  \lrto  \mathop{\rm Ext} ( {\mathcal K} , \C^*)  \lrto H^2_{\rm c}(\Lie {\mathcal K}, \C)  \, \mbox{,}
\end{equation}
where the first $\Hom$ is the group of smooth homomorphisms, the first arrow is given by the restriction of homomorphisms to the subgroup $\pi_1( {\mathcal K})  \subset \overline{ \mathcal K}$,
the second arrow assigns to a group homomorphism $\gamma : \pi_1( {\mathcal K}) \to \C^*$ the quotient of $\overline{\mathcal K} \times \C^*$ modulo the group which is the graph of~$\gamma$, the last arrow is the natural homomorphism described above.

Exact sequence~\eqref{monodr} easily follows from the usual reasoning as in~\cite[Prop.~(7.4)]{Se}  and~\cite[\S~4.5]{PS}, where a vector space  splitting of a central extension of tangent Lie algebras is considered as the left invariant connection on the principal $\C^*$-bundle of central extension (when the central extension of Lie algebras comes from a central extension of Lie groups), the Lie algebra $2$-cocycle of the central extension is considered as the curvature of this connection, and when this $2$-cocycle (or curvature) is zero, then the monodromy of this connection gives the homomorphism from
the group $\pi_1( {\mathcal K})$ to the group  $\C^*$.

Note that more general exact sequence than~\eqref{monodr} (when this sequence is extended  to the left and to the right, and when $\C^*$ is replaced by some Abelian and possibly infinite-dimensional Lie group) is obtained
 in~\cite[Theorem~7.12]{Neeb2}.

We will apply~\eqref{monodr} when ${\mathcal K} = \g^0$ and ${\mathcal K} = \g_+ $.

By Theorem~\ref{smooth}, the group $\Hom(\overline{ \g^0}, \C^*)$ is trivial.

We will prove that the group  $\Hom(\overline{ \g_+}, \C^*)$ is also trivial.

By Lemma~\ref{sl}, the Lie algebra $\Lie \g_+$ of the Lie group $\g_+$
is the subalgebra  of the complex Lie algebra $\Lie_{\C}  \g_+ = \h_+  \rtimes sl(2, \C)$.

Consider the dense Lie subalgebra $\Lie_{\C}^{f}  \g_+ = \C[z] \rtimes sl(2, \C)$ of the Lie algebra $\Lie_{\C}  \g_+$.
The Lie algebra $\Lie_{\C}^{f} \g_+$ has the basis which consists of all elements $e_n=z^n$, where $n \ge 0$, and elements $d_m = z^{m+1} \frac{d}{dz}$, where $m= -1, 0, 1$, with the following relations (see formula~\eqref{fo1})
$$
[d_m, d_n]= (m -n)d_{n+m}  \, \mbox{,}  \qquad [d_m, e_n] = - n e_{n+m}  \, \mbox{,}  \qquad [e_n, e_m] =0 \, \mbox{.}
$$
Hence, the Lie algebra $\Lie_{\C}^{f}  \g_+$ is perfect, i.e.~ $\left[\Lie_{\C}^{f} \g_+ , \, \Lie_{\C}^{f} \g_+ \right] = \Lie_{\C}^{f}  \g_+$.

Now let $\tau$ be any element from $\Hom(\overline{ \g_+}, \C^*)$. Then the differential $(d \tau)_e$ at the identity element $e \in \overline{ \g_+}$ is the continuous homomorphism from the Lie algebra $\Lie { \g_+}$ to the Abelian Lie algebra $\C$. As in the proof of Theorem~\ref{smooth},   the  homomorphism $(d \tau )_e$ can be uniquely extended to the $\dr$-linear homomorphism from the Lie algebra $\Lie_{\C}  \g_+$ to the Lie algebra $\C$. But the last homomorphism will be zero, since
restricted to the dense Lie subalgebra $\Lie_{\C}^{f}  \g_+$ it is zero, because $\Lie_{\C}^{f}  \g_+$ is perfect  and $\left[ \C, \C  \right] = 0$.

Hence $(d \tau)_e$ is zero. As in the proof of Theorem~\ref{smooth},  this implies that the differential $(d \tau)_v$ is zero for any point $v \in \overline{ \g_+}$. Since $\overline{ \g_+}$ is connected, we have that the homomorphism
$\tau$ is trivial.

The embedding of Lie groups $\g_+  \hookrightarrow \g^0$ induces isomorphism of their fundamental groups that are equal to  $\Z \times \Z$ (see Section~\ref{Expl_descr} and formula~\eqref{act-SL}). Therefore the statement of the theorem follows from  a commutative  diagram written with the help of  exact sequence~\eqref{monodr}:
$$
 {\xymatrix{
1 \ar[r] &  \Hom ( \pi_1( \g^0) , \C^* )  \ar[d]  \ar[r] & \mathop{\rm Ext} ( \g^0 , \C^*)   \ar[d] \ar[r] & H^2_{\rm c}(\Lie {\g^0}, \C) \ar[d] \\
1 \ar[r] &   \Hom ( \pi_1( \g_+) , \C^* ) \ar[r]   &  \mathop{\rm Ext} ( \g_+ , \C^*)  \ar[r] & H^2_{\rm c}(\Lie {\g_+}, \C)  \, \mbox{,}
 }}
 $$
where
the first vertical arrow is an isomorphism, and the horizontal sequences are exact sequences.
\end{proof}

\begin{Th}  \label{Th-eq}
Consider the multiplicative notation for the group law in the Abelian group~$H_{\rm sm}^2(\g, \C^*)$  (see Remark~\ref{Rem-coh}).
In  $H_{\rm sm}^2(\g, \C^*)$ the $12$th power of the determinant central extension (see formula~\eqref{det-centr}) is equal to
\begin{equation}  \label{fin-th}
  \langle \Lambda, \Lambda \rangle^6 \cdot    \langle \Lambda, \Omega \rangle^{-6} \cdot \langle \Omega, \Omega \rangle  \, \mbox{,}
\end{equation}
In other words, the $12$th power of the determinant central extension is equivalent to the central extension given by the $2$-cocycle~\eqref{fin-th}.
\end{Th}
\begin{proof}
By Theorem~\ref{Th5} it is enough to prove the statement of the theorem
after restriction the central extensions  to the subgroup $\g^0$ of the group $\g$.

Now we will apply Theorem~\ref{Th7}. The determinant central extension restricted to $\g^0$ has the canonical $2$-cocycle $D$ (see Theorem~\ref{th-3}).
By Theorem~\ref{main-Lie}    the corresponding Lie algebra $2$-cocycles (for $D$ and for~\eqref{fin-th})  coincide.

From formula~\eqref{2-coc}
it follows that the restriction of $D$ to the subgroup $\g_+$ is trivial, since any element from $\g_+$  has the  part $\pm$ of the block matrix~\eqref{block-mat} equal to zero.
Besides,
every of $2$-cocycles $\langle \Lambda, \Lambda \rangle$,     $\langle \Lambda, \Omega \rangle $,  $ \langle \Omega, \Omega \rangle$ is trivial after restriction to the subgroup $\g_+$,
because in Definition~\ref{bracket}  we apply formula~\eqref{form-CC} for $\mathbb T$  to functions $f$ and $g$ with the winding number zero (see Proposition~\ref{wind-diff}) and with the property  that the functions $\log f$ and $\log g$ belong to $\h_+$.

\end{proof}

\begin{nt} \label{coc-same} \em
Over the subgroup $\g^0$ of $\g$ there is the canonical smooth section of the determinant central extension that leads to the smooth $2$-cocycle $D$ (see Theorem~\ref{th-3}).
The proof of Theorem~\ref{Th-eq} gives the equality of the $2$-cocycles $D^{12}$ and~\eqref{fin-th} only modulo the possible $2$-coboundary.
Note that by Theorem~\ref{smooth}   there is a unique smooth function  $ \tau : \g^0 \to \C^*$ that defines this $2$-coboundary
$$(g_1, g_2) \longmapsto \frac{ \tau(g_1) \tau(g_2)}{\tau(g_1 g_2)} \, \mbox{,}$$
where $(g_1, g_2) \in \g^0 \times \g^0$.
There is an interesting question: to find the function $\tau$ explicitly.
\end{nt}

\begin{nt}  \em
The formal (algebraic) analog of Theorem~\ref{Th-eq} was proved in~\cite[Theorem~7]{O2}. This is the local Deligne-Riemann-Roch isomorphism for line bundles, see~\cite[Th\'eor\`eme~9.9]{D1} (and also \cite[Section~1]{O2}).
\end{nt}

\begin{nt}  \label{referee}  \em
The interesting formulas were obtained in~\cite[\S~4.10]{Ner1}, \cite{Ner2}, \cite{Ner3} for the central extensions of the group $\mathop{\rm Diff}^+ (S^1) $
of orientation preserving
diffeomorphisms of the circle.
In particular, Yu.~A.~Neretin calculated the canonical $2$-cocycles on the group $\mathop{\rm Diff}^+ (S^1) $ with coefficients in
 the trivial $\mathop{\rm Diff}^+ (S^1) $-module
 $\C^*$. These $2$-cocycles are obtained from the projective  highest weight representations of $\mathop{\rm Diff}^+ (S^1)$. These representations are constructed by the restriction of the projective Weil representation through the
 embeddings of $\mathop{\rm Diff}^+ (S^1)$ to the group ${\rm ASp}(2 \infty, \dr)$ of affine transformations  generated by the elements of the infinite-dimensional symplectic group and some shifts. The projective Weil representation of  ${\rm ASp}(2 \infty, \dr)$
  is the  action by the Gaussian operators on the Fock space related with the infinite-dimensional  Hilbert space. Besides, there is the formula at the end of~\cite[\S~4.10]{Ner1} that relates the Fredholm determinants of operators of type  $1 + I$, where $I$ is the special trace-class operators acting in the infinite-dimensional Hilbert space,  with the formula obtained with the help of functions $\exp$, $\log$, derivatives,  $d \log $ and $\int_{S^1}$. The both parts of this formula use the decomposition of elements of $\mathop{\rm Diff}^+ (S^1)$ via the welding, cf.~Section~\ref{Conf-weld}.

  It is an interesting problem to relate these explicit  $2$-cocycles and formulas with our  $2$-cocycle $D$ (or $D^{12}$) from Theorem~\ref{th-3} and our $2$-cocycle \eqref{fin-th} after restriction to the subgroup $\DS$ of the group $\g$, see also Remark~\ref{coc-same}. Note that the  $2$-cocycles~\eqref{fin-th} and $\langle \Omega, \Omega \rangle$ coincide after restriction to $\DS$.

\end{nt}

\section{Deligne cohomology and the Gysin map}
\label{DelGys}

\subsection{Deligne cohomology}  \label{Del-coh}

We recall the definition and some properties of Deligne cohomology on (finite dimensional) smooth manifolds (see, e.g.,~\cite[\S~1]{EV}, where the corresponding definitions and properties are given for complex analytic manifolds, but what we need is true in our smooth manifolds case too, or cf.~\cite[\S~1.5]{Bry}). (We will always assume further in the article that finite-dimensional smooth manifolds are paracompact.)

Let $Y$ be a (finite dimensional) smooth manifold. By $\oo_Y$ and $\oo_Y^*$ denote the sheaves of smooth $\C$-valued and  $\C^*$-valued functions on $Y$ correspondingly.
For any integer $k \ge 0$ by $\Omega_Y^{k}$ denote the sheaf of smooth $k$-differential forms with complex coefficients on $Y$. For any integer $p \ge 0$ define the Deligne complex $\dz(p)_{D}$ on $Y$ in the following way:
$$
0 \lrto (2  \pi i)^p \dz \lrto \oo_Y \xlrto{d} \Omega^1_Y \xlrto{d} \ldots \xlrto{d} \Omega^{p-1}_Y   \lrto 0  \, \mbox{,}
$$
where $(2  \pi i)^p \dz$ is in degree zero, and the map $d$ is exterior diffrentiation.
The Deligne complex $\dz(0)_{D}$ is just the constant sheaf $\dz$ placed in degree zero.

The Deligne cohomology ${H}^q_D(Y, \dz(p)) = \mathbb{H}^q (Y, \dz(p)_{D}) $ is the $q$-th hypercohomology of the Deligne complex.

There is a multiplication
$$
\cup \; : \;  \dz(p_1)_{D} \otimes \dz(p_2)_D \lrto \dz(p_1 + p_2)_D
$$
which is the usual tensor multiplication with $\dz$ (in all degrees) if $p_1=0$ or $p_2=0$, and when $p_2 >0$ the multiplication $\cup$ is given by
$$
 x \cup y = \left\{
\begin{array}{ll}
x \cdot y & \quad \mbox{if} \quad \deg x =0 \\
x \wedge dy & \quad \mbox{if} \quad \deg x > 0 \quad \mbox{and} \quad \deg y = p_2 \\
0 &  \quad \mbox{otherwise.}
\end{array}
  \right.
$$
The multiplication $\cup$ is a morphism of complexes.
The multiplication $\cup$ is associative and is commutative up to homotopy.
Using methods from homological algebra (or the calculation by the \v{C}ech cohomology), one can see that the multiplication $\cup$ gives an associative ring structure on
$\bigoplus\limits_{p,q} H_D^q(Y, \dz(p))$. The product in this ring is anticommutative, i.e. for $\alpha  \in H_D^{q_1}(Y, \dz(p_1))$ and $\beta  \in H_D^{q_2}(Y, \dz(p_2))$, we have
$\alpha \cup \beta = (-1)^{q_1 q_2} \beta \cup \alpha $.

\bigskip

If $p = 0$, then ${H}^q_D(Y, \dz(0))$ is nothing but singular cohomology $H^q (Y, \dz)$. Moreover, for any integer $p \ge 0$ there is the natural map as the composition of maps
\begin{equation}  \label{pro}
{H}^q_D(Y, \dz(p))  \lrto {H}^q(Y, (2  \pi i)^p \dz) \lrto {H}^q(Y,  \dz)  \, \mbox{,}
\end{equation}
where the first map is induced by the map from $\dz(p)_{D}$ to the complex which is the constant sheaf   $(2  \pi i)^p \dz$ placed in degree zero, and the second map is the isomorphism induced by the multiplication with $(2 \pi i)^{-p}$.

If $p=1$,  then ${H}^q_D(Y, \dz(1))$ is isomorphic to $H^{q-1}(Y, \oo_Y^*)$, since $\dz(1)_D$ is quasi-isomorphic to $\oo_Y^*[-1]$ via the exponential map.

If $p=2$, $q =2$, then $\dz(2)_D$ is quasi-isomorphic to $\left( \oo_Y^*  \xlrto{s \mapsto ds/s} \Omega^1_Y    \right)[-1]$ via the map $x \mapsto \exp(x/(2 \pi i))$ on $\oo_Y$ and the multiplication with $(2\pi i)^{-1}$ on~$\Omega^1_Y$, and therefore
$$
{H}^2_D(Y, \dz(2))  = {\mathbb{H}}^1(Y, \oo_Y^*  \lrto \Omega^1_Y)  \, \mbox{,}
$$
where the last group is canonically isomorphic to the group of complex line bundles on $Y$ with a connection (what is easy to see by the \v{C}ech cohomology).

Via these isomorphisms, the multiplication $\cup$ described above gives the pairing
\begin{equation}  \label{pair}
\cup  \; : \; H^0(Y, \oo_Y^*)  \times H^0 (Y, \oo_Y^*)  \lrto {\mathbb{H}}^1(Y, \oo_Y^*  \lrto \Omega^1_Y)  \, \mbox{.}
\end{equation}
Consider $f$ and $g$ from $H^0(Y, \oo_Y^*)$ and the open cover $\{ U_{\alpha} \}$ of $Y$ such that for any $\alpha$ a branch $\log_{\alpha} f = \log (f |_{U_{\alpha}})$ is defined.
Then the element $f \cup g$ is given by the \v{C}ech cocycle $(\xi_{\alpha \beta}, \omega_{\alpha})$  (see also~\cite[\S~1.4, iv)]{EV}), where
$$
\xi_{\alpha \beta} = g^{\frac{1}{2 \pi i} (\log_{\alpha} f - \log_{\beta} f )}  \in \oo_Y^*(U_{\alpha} \cap U_{\beta}) \, \mbox{,}  \qquad \omega_{\alpha} = \frac{-1}{2 \pi i} \log_{\alpha} f \, \frac{dg}{g}  \in \Omega^1_Y(U_{\alpha})   \, \mbox{.}
$$
The \v{C}ech cocycle $(\xi_{\alpha \beta})$ defines the line bundle on $Y$ such that $s_{\beta} = \xi_{\alpha \beta } s_{\alpha}$, and $1$-forms $\omega_{\alpha}$ define the connection $\nabla$ on this bundle
such that $\nabla(s_{\alpha}) = \omega_{\alpha} s_{\alpha}$.

If $Y = S^1$, take a point $x_0 \in S^1$. Consider the open cover $S^1 = U_0 \cup U_1$, where $U_1 = S^1 \setminus x_0$ and $U_0$ is a very small neighbourhood of $x_0$. From the above description it is easy to see that
the monodromy of the connection associated with $f \cup g$ is given as
$$  \pi_1(S^1) = \dz \ni 1 \longmapsto    {\mathbb T}(f,g)   \, \mbox{,}$$
 see formula~\eqref{expli} (one has also to use that the horizontal section on $U_1$ is $\exp \int_{U_1} (- \omega_1) s_1 $),
 and $1 \in \pi_1(S^1)$ is considered
 in the counterclockwise direction.

\subsection{Gysin map, $\cup$-products and the map $\mathbb T$}  \label{Gysin}

Let $\pi : M \to B$ be a fibration in oriented circles. In other words,  $\pi$ is a proper surjective submersion between (finite dimensional) smooth manifolds with $S^1$ fibers, which is, by Ehresmann's lemma, a locally trivial fibration which has a collection of transitions functions with value in
the group  of orientation preserving diffeomorphisms of $S^1$.

Then the sheaf
$R^1 \pi_* \dz$ is isomorphic to the sheaf $\dz$. This isomorphism can be choosen up to multiplication by $- 1$, and we fix this isomorphism such that the following diagram is commutative:
\begin{equation}   \label{wind-diagr}
 {\xymatrix{
   {H^0(S^1, \oo_S^*)} \ar[d] \ar[rr] &&  {H^0(S^1, 2 \pi i \dz)} \ar[d]  \\
  H^1(S^1, 2 \pi i \dz) \ar[rr]   && 2 \pi i \dz
 }}
 \end{equation}
 where the upper arrow is induced by the map of sheaves $s \mapsto 2 \pi i \nu(s)$ ($\nu$ is the winding number), the left vertical arrow is induced by the exponential sheaf sequence (see also~\eqref{sheaf-exp} below), and the bottom arrow fixes the isomorphism $H^1(S^1,  \dz) \to  \dz$.

There is the Leray spectral sequence $$ E^{pq}_2 = H^p (B, R^q \pi_* \dz) \Rightarrow   H^{p+q}(M , \dz)  \, \mbox{.}$$
From this spectral sequence there is the Gysin map, as the composition of the following maps:
$$
\pi_* \; : \; H^p(M, \dz)  \lrto E_{\infty}^{p-1,1}  \subset E_2^{p-1,1} = H^{p-1}(B, R^1 \pi_* \dz ) \lrto H^{p-1}(B, \dz)  \, \mbox{.}
$$

We will use the multiplicative notation for the group law in an Abelian group (e.g., the cohomology groups and other groups) if the group $\oo_M^*$ is used in the notation of this group, and we will use the additive notation for the group law if the group $\dz$ or the complex~$\dz(p)$ is used in the notation of a group.

There is the exponential sheaf  sequence
\begin{equation}  \label{sheaf-exp}
0 \lrto 2 \pi i \dz  \lrto \oo_M   \xlrto{\exp}   \oo_M^*  \lrto 1  \, \mbox{.}
\end{equation}
Note that the sheaf $\oo_M$ is soft and therefore this sheaf is acyclic.

From these sequences and $\cup$-product in singular cohomology it follows the following composition of maps:
\begin{multline}  \label{first-seq}
H^l(M, \oo_M^*)  \times H^m(M, \oo_M^*) \lrto H^{l+1}(M, \dz)  \times H^{m+1}(M, \dz) \stackrel{\cup}{\lrto}   H^{l+m+2}(M, \dz)  \xlrto{\pi_*}   \\
 \xlrto{\pi_*}   H^{l+m+1}(B, \dz)  \, \mbox{.}
\end{multline}

Since $R^q \pi_* \oo_M^* = \{1\}$ for any integer $q \ge 1$, from the  Leray spectral sequence for the  sheaf $\oo_M^*$ we have for any integer $p \ge 0$ the isomorphism
\begin{equation}  \label{spec-iso}
{H^p(M, \oo_M^*) \lrto H^{p}(B, \pi_* \oo_M^* )}  \, \mbox{.}
\end{equation}
Using $\cup$-products in the cohomology of sheaves and the exponential sheaf sequence on~$B$, we have a composition of maps
\begin{multline}   \label{second-sec}
H^l(M, \oo_M^*)  \times H^m(M, \oo_M^*) \lrto  H^l(B, \pi_* \oo_M^*)  \times H^m(B,  \pi_* \oo_M^*)  \stackrel{\cup}{\lrto}   \\
\stackrel{\cup}{\lrto}
 H^{l+m}(B, \pi_* \oo_M^*  \otimes_{\dz} \pi_* \oo_M^*)   \xlrto{{\mathbb T}^{(-1)^{m+1}}}   H^{l+m}(B, \oo_B^*)  \lrto H^{l+m+1}(B, \dz)   \, \mbox{.}
\end{multline}
Here the map ${{\mathbb T}^{(-1)^{m+1}}} $ is induced by the corresponding map of sheaves, where on every fiber $S^1$ the map ${\mathbb T}^{(-1)^{m+1}}$ (composition of $\mathbb T$ and $a \mapsto a^{(-1)^{m+1}}$) is applied.

\begin{nt}  \em
The map ${{\mathbb T}^{(-1)^{m+1}}} $ in \eqref{second-sec}  is factored through $H^{l+m}(B, K_2^M(\pi_* \oo_M^*))$,  \linebreak where $K_2^M(\pi_* \oo_M^*)$ is the sheaf on $B$ associated with the presheaf of Milnor $K_2$-groups
$U \mapsto   K_2^M(\oo_M^*(\pi^{-1} (U)))$, see Section~\ref{bimult}.
\end{nt}

\medskip

\begin{Th}   \label{th-9}
The composition of maps~\eqref{first-seq} coincides with the composition of maps~\eqref{second-sec}.
\end{Th}
\begin{proof}
{\em Step~1.}
Consider the map
\begin{equation}  \label{one-map}
H^{l+m}(B, \pi_* \oo_M^*  \otimes_{\dz} \pi_* \oo_M^*)
\xlongrightarrow{{\mathbb T}^{-1}}
   H^{l+m}(B, \oo_B^*)  \, \mbox{,}
\end{equation}
Here  the map ${{\mathbb T}^{-1}} $ is induced by the corresponding map of sheaves, where on every fiber $S^1$ the map ${\mathbb T}^{-1}$ (composition of $\mathbb T$ and $a \mapsto a^{-1}$ ) is applied.

Note that  map~\eqref{one-map}
coincides with the composition of maps
\begin{equation}  \label{hyper}
H^{l+m}(B, \pi_* \oo_M^*  \otimes_{\dz} \pi_* \oo_M^*)
 \xlongrightarrow{} H^{l+m}(B, {\mathbb R}^2 \pi_* \dz(2)_D  )
\xlongrightarrow{\text{monodr.}^{-1}}
 H^{l+m}(B, \oo_B^*)  \, \mbox{.}
\end{equation}
Here ${\mathbb R}^2 \pi_*$ is the hyper-derived functor. In particularly, ${\mathbb R}^2 \pi_* \dz(2)_D $ is the sheaf associated with the presheaf $U \mapsto H^2_D(\pi^{-1}(U), \dz(2) )$, where  the group  $H^2_D(\pi^{-1}(U), \dz(2) )$ is also the group of complex line bundles with  connection on
an open set
$\pi^{-1}(U)$.  The first arrow in formula~\eqref{hyper} is induced by the pairing~\eqref{pair}.  The second arrow in formula~\eqref{hyper} is induced by the map of sheaves
${\mathbb R}^2 \pi_* \dz(2)_D \to \oo_B^*$ that
for each fiber $S^1$ gives the non-zero complex number which is
the image of the element
$-1  \in \pi_1(S^1)$ (considered in the clockwise direction) under the monodromy of a connection on a line bundle on the fiber.

There is the Grothendieck spectral sequence, which generalizes the Leray spectral sequence,
 $$ E^{pq}_2 = H^p (B, {\mathbb R}^q \pi_* \dz(l)_D ) \, \Longrightarrow \,   \mathbb{H}^{p+q}(M , \dz(l)_D) \, \mbox{.}$$
 Since ${\mathbb R}^q \pi_* \dz(2)_D = \{ 0\}$ for any integer $q >2$, from this spectral sequence when $l=2$ there is the (new) Gysin map, as the composition of the following maps:
$$
\pi_* \; : \; H^p_D(M, \dz(2)) =  {\mathbb H}^p (M, \dz(2)_D) \lrto E_{\infty}^{p-2,2}  \subset E_2^{p-2,2} = H^{p-2}(B, {\mathbb R}^2 \pi_* \dz(2)_D )   \, \mbox{.}
$$

Analogously, since ${\mathbb R}^q \pi_* \dz(1)_D = \{ 0\}$ for $q = 0$ and any integer $q >1$, there is the isomorphism
$$\pi_* \; : \;  H^{p+1}_D(M, \dz(1))  \lrto H^{p}(B, {\mathbb R}^1 \pi_* \dz(1)_D) $$
which coincides with  isomorphism~\eqref{spec-iso}.

Since the spectral sequences which we use are compatible with  the $\cup$-products up to a certain sign (see, e.g., \cite[Append.~A, \S~3;  Ch.~IV, \S~6.8]{Bre}),   the following diagram is commutative:
$$
 {\xymatrix{
   {H_D^{l+1}(M, \dz(1) ) } \, \, \times  \, \,   {H_D^{m+1}(M, \dz(1) ) } \ar@<1.7cm>[d]^{\pi_*}    \ar@<-1.7cm>[d]^{\pi_*}    \ar[rr]^-{\cup} &&   {H_D^{l+m+2}(M, \dz(2) )} \ar[d]^{\pi_*}  \\
  {H^{l}(B,  {\mathbb R}^1 \pi_*  \dz(1)_D ) }  \, \,  {\times}  \, \,   {H^{m}(B,  {\mathbb R}^1 \pi_*  \dz(1)_D ) }     \ar[rr]^-{(-1)^m \, \cup} &&   {H^{l+m}(B,  {\mathbb R}^2 \pi_*  \dz(1)_D ))}
 }}
 $$
 By construction of the multiplication $\cup$ between Deligne complexes, the following diagram is also commutative:
 $$
 {\xymatrix{
   {H_D^{l+1}(M, \dz(1) ) } \, \, \times  \, \,   {H_D^{l+1}(M, \dz(1) ) } \ar@<1.7cm>[d]    \ar@<-1.7cm>[d]    \ar[r]^-{\cup} &   {H_D^{l+m+2}(M, \dz(2) )} \ar[d]  \\
  {H^{l+1}(M,  \dz })  \, \,  {\times}  \, \,   {H^{m+1}(M,   \dz ) }     \ar[r]^-{\cup} &   {H^{l+m+2}(M,    \dz ))}
 }}
 $$
where the vertical arrows  are maps~\eqref{pro}.

Besides, for any integers $s , r \ge 1$ the following diagram is commutative
\begin{equation}  \label{exp-diagr}
{\xymatrix{
   {H_D^{s}(M, \dz(r) ) }  \ar[d]     \ar[r] &   {H^{s}(M,  \dz )}   \\
  {H_D^{s}(M,  \dz(1) })       \ar[r] &   {H^{s-1}(M,    \oo_M^* )}  \, \mbox{,}  \ar[u]
 }}
\end{equation}
where the upper arrow is the map~\eqref{pro},  the left vertical arrow is induced by the map of complexes $\dz(r)  \to \dz(1)$ given by multiplication with $(2 \pi i)^{1-r}$ and the truncation of the complex,
and the right vertical arrow is induced by the exponential sheaf sequence~\eqref{sheaf-exp}.

Therefore the statement  of the theorem will follow after we will  prove that a composition of maps (when an integer $p \ge 2$)
\begin{equation}  \label{sequ1}
H^p_D(M, \dz(2))  \xlrto{\pi_*}  H^{p-2}(B, {\mathbb R}^2 \pi_* \dz(2)_D )  \xlrto{\text{monodr.}^{-1}}  H^{p-2}(B, \oo_B^*)  \lrto H^{p-1}(B, \dz)
\end{equation}
coincides with a composition of maps
\begin{equation} \label{sequ2}
H^p_D(M, \dz(2)) \lrto H^p(M, \dz)  \xlrto{\pi_*}  H^{p-1}(B, R^1 \pi_* \dz)  \lrto H^{p-1}(B, \dz)  \, \mbox{,}
\end{equation}
where the last arrow in formula~\eqref{sequ1} follows from  the exponential sheaf sequence on~$B$, and the first arrow in formula~\eqref{sequ2} is the map~\eqref{pro}.

\medskip

{\em Step~2.} We will prove that the compositions of maps~\eqref{sequ1} and~\eqref{sequ2} coincide.

 Consider the map $I \colon \oo_M^*  \to \Omega^1_M$, where  $I(s)=    ds/s $. By $I(\oo_M^*)$ denote the image of $\oo_M^*$ under the map $I$, which is a subsheaf of the sheaf $\Omega^1_M$.  From an exact sequence of sheaves
 $$
 1 \lrto \dc^*  \lrto \oo_M^*  \xlrto{I}  I(\oo_M^*)  \lrto 1
  $$
  it follows that $R^q \pi_* I(\oo_M^*) = \{ 1 \}$
for any integer $q \ge 1$ (one has also to use the exponential sheaf sequences~\eqref{sheaf-exp}). Therefore there is the following exact sequence of  sheaves on $B$:
\begin{equation}    \label{I-seq}
1 \lrto \pi_* I(\oo_M^*) \lrto  \pi_* \Omega^1_M  \lrto \pi_* \left( \Omega^1_M/ I(\oo_M^*)  \right)  \lrto 0   \, \mbox{.}
\end{equation}

From exact sequence of complexes of sheaves on $M$
$$
1 \lrto \left( \oo_M^*  \to I(\oo_M^*) \right)  \lrto \left(\oo_M^* \to \Omega^1_M    \right)  \lrto \left( 1 \to  \Omega^1_M / I(\oo_M^*)     \right)  \lrto 1
$$
we have the composition of maps of sheaves on $B$
\begin{equation}   \label{new-comp}
{\mathbb R}^1 \pi_* \left( \oo_M^*  \to \Omega_M^1  \right)  \lrto {\mathbb R}^1 \pi_* \left(1 \to  \Omega^1_M / I(\oo_M^*)      \right)  \simeq \pi_*(\Omega^1_M / I(\oo_M^*) ) \simeq \pi_* \Omega^1_M / \pi_* I(\oo_M^*)  \, \mbox{,}
\end{equation}
where the last isomorphism follows from exact sequence~\eqref{I-seq}.

Since for any point $x \in B$ there is an open neighbourhood $U \ni x$ such that any line bundle on the open set $\pi^{-1}(U)$ is trivial, it is easy to see that the middle arrow in formula~\eqref{sequ1} coincides with a composition of maps
\begin{multline*}
 H^{p-2}(B, {\mathbb R}^2 \pi_* \dz(2)_D )  \simeq H^{p-2}(B, {\mathbb \dr}^1 \pi_* \left(\oo_M^* \to \Omega^1_M    \right)  )  \lrto \\
 \lrto H^{p-2}(B,  \pi_* \Omega^1_M / \pi_* I(\oo_M^*) )  \xlrto{\exp \int_{S^1}}  H^{p-2}(B, \oo_B^*)  \, \mbox{,}
\end{multline*}
where we used the composition of maps~\eqref{new-comp}, and the map  ${\exp \int_{S^1}}$ is induced by the map of sheaves that on every fiber $S^1$
is the composition of the exponential map and the map which is
 the integration (in the counterclockwise direction) of a differential $1$-form.

Therefore, using the evident commutative diagram of sheaves on $B$
$$
 {\xymatrix{
1 \ar[r]  & {\pi_* I(\oo_M^*)}  \ar[r]  \ar[d]^{\int_{S^1}}  & \pi_* \Omega^1_M  \ar[r] \ar[d]^{\int_{S^1}}   &   \pi_* \Omega^1_M / \pi_* I(\oo_M^*)  \ar[r]  \ar[d]^{\exp \int_{S^1}}  & 0 \\
0 \ar[r]  & 2 \pi i \dz  \ar[r]                             & \oo_B                \ar[r]^{\exp}                &  \oo_B^*   \ar[r]                                                           & 1  \, \mbox{,}
 }}
 $$
we have that the composition of maps~\eqref{sequ1} coincides
with the following composition of maps:
\begin{multline}
H^p_D(M, \dz(2))  \xlrto{\pi_*}
 H^{p-2}(B, {\mathbb R}^2 \pi_* \dz(2)_D )  \simeq H^{p-2}(B, {\mathbb \dr}^1 \pi_* \left(\oo_M^* \to \Omega^1_M    \right)  )  \lrto \\  \label{long-seq}
 \lrto H^{p-2}(B,  \pi_* \Omega^1_M / \pi_* I(\oo_M^*) )  \lrto  H^{p-1}(B, \pi_* I(\oo_M^*))  \xlrto{\int_{S^1}} H^{p-1}(B, 2 \pi i \dz)  \simeq H^{p-1}(B, \dz)  \, \mbox{.}
\end{multline}

Now the statement of the theorem will follow after we will prove that the composition of maps~\eqref{long-seq} coincides with the composition of maps~\eqref{sequ2}.

\medskip

{\em Step 3.} We will prove that the compositions of maps~\eqref{long-seq} and~\eqref{sequ2} coincide.

From the construction of the spectral sequence (using the Godement resolutions) it follows that a  piece of composition of maps~\eqref{long-seq}
\begin{multline*}
H^p_D(M, \dz(2))  \xlrto{\pi_*}
 H^{p-2}(B, {\mathbb R}^2 \pi_* \dz(2)_D )  \simeq H^{p-2}(B, {\mathbb \dr}^1 \pi_* \left(\oo_M^* \to \Omega^1_M    \right)  )  \lrto \\
 \lrto H^{p-2}(B,  \pi_* \Omega^1_M / \pi_* I(\oo_M^*) )  \lrto  H^{p-1}(B, \pi_* I(\oo_M^*))
\end{multline*}
coincides with a composition of maps
\begin{multline}
H^p_D(M, \dz(2))
\simeq H^{p-1} (M, \oo_M^* \to \Omega^1_M)
\lrto H^{p-1}(M, \oo_M^*) \lrto  \\ \lrto  H^{p-1}(M, I(\oo_M^*))  \simeq H^{p-1}(B, \pi_* I(\oo_M^*))  \, \mbox{,}  \label{last-iso}
\end{multline}
where the last isomorphism in formula~\eqref{last-iso} follows from the Leray spectral sequence for the sheaf $I(\oo_M^*)$, since
$R^q \pi_* I(\oo_M^*) = \{ 1 \}$
for any integer $q \ge 1$.

Now from the construction of the spectral sequence (using the Godement resolutions) it follows that the following diagram is commutative
\begin{equation}  \label{exp-diagr2}
{\xymatrix{
   {H^{p-1}(M, \oo_M^* ) }  \ar[d]     \ar[r] &   {H^{p}(M,  \dz )}  \ar[d]^{\pi_*} \\
  {H^{p-1}(B,  \pi_* \oo_M^* })       \ar[r] &   {H^{p-1}(B,    R^1 \pi_* \dz )}  \, \, \mbox{,}
  }}
\end{equation}
where the upper and bottom arrows are induced by the exponential sheaf sequence~\eqref{sheaf-exp}, the left vertical arrow is an isomorphisms as in formula~\eqref{spec-iso}.

Besides, it is easy to see that the following diagram of sheaves on $B$, which generalizes the diagram~\eqref{wind-diagr}, is commutative:
 \begin{equation}  \label{gen-diagr}
{\xymatrix{
   {\pi_* \oo_M^* }  \ar[d]     \ar[rr] &&   {R^1 \pi_*   \dz }  \ar[d] \\
  { \pi_* I( \oo_M^*) }       \ar[rr]^-{\frac{1}{2 \pi i}  \int_{S^1}} & &  {\, \, \dz }  \, \, \mbox{.}
  }}
\end{equation}

Now from formulas~\eqref{last-iso}--\eqref{gen-diagr} and diagram~\eqref{exp-diagr} (when $s=p$ and $r=2$) it follows that the compositions of maps~\eqref{long-seq}  and~\eqref{sequ2} coincide.

\end{proof}

\section{Circle fibrations and  topological Riemann-Roch theorem}
\label{sec-last}

\subsection{Group cohomology, gerbes and central extensions}   \label{gerbes}

Let $Y$ be a smooth (finite-dimensional) manifold. We will use notation from Section~\ref{Del-coh}.

Using the exponential sheaf sequence on $Y$, for any integer $p \ge 1$ we have an isomorphism
$$H^p(Y, \oo_Y^*)  \simeq H^{p+1}(Y, \dz)  \, \mbox{.}$$

Let $\mathcal K$ and $\mathcal N$ be   smooth Lie groups (possibly infinite-dimensional)  such that $\mathcal N$ is an Abelian group and the Lie group  $\mathcal K$ acts on the Lie group $\mathcal N$, i.e. $\mathcal N$ is
a $\mathcal K$-module.

Recall the definition of  Van Est smooth group  cohomology. They are  the cohomology of the following complex
\begin{equation} \label{van-est}
C^0 ({\mathcal K},{\mathcal N})  \stackrel{\delta_0}{\lrto} C^{1}({\mathcal K}, {\mathcal N}) \stackrel{\delta_1}{\lrto} \ldots \xrightarrow{\delta_{k-1}} C^{k}({\mathcal K}, {\mathcal N}) \stackrel{\delta_k}{\lrto} \ldots  \, \mbox{,}
\end{equation}
where $C^0({\mathcal K}, {\mathcal N})= {\mathcal N}$, $C^k({\mathcal K}, {\mathcal N})$ when $k \ge 1$ is the group of smooth maps from
the smooth (possibly infinite-dimensional) manifold
${\mathcal K}^{\times k}$ to the Abelian Lie group ${\mathcal N}$, and the differentials in complex~\eqref{van-est} are given as follows
\begin{multline*}
\mathop{\delta_{q}c} \, (g_1, \ldots, g_{q+1}) = g_1  \mathop{c}  (g_2, \ldots, g_{q+1}) \, \cdot \,  \prod_{i=1}^q  \mathop{c}  (g_1, \ldots, g_i g_{i+1}, \ldots, g_{q+1})^{(-1)^i} \,  \cdot\\
\cdot \,   \mathop{c}  (g_1, \ldots, g_q)^{(-1)^{q+1}}  \, \mbox{,}
\end{multline*}
where integers $q \ge 0$, elements $c \in C^q({\mathcal K},{\mathcal N})$, $g_j \in {\mathcal K}$ with $1 \le j \le q+1$. Besides, we use the multiplicative notation for the group laws in the  Abelian groups $\mathcal K$,  ${\mathcal N}$ and~$C^{q}({\mathcal K},{\mathcal N})$.

In particularly, we have from complex~\eqref{van-est} the definitions of smooth $1$-cocycles and $2$-cocycles, and also the definition of the second smooth cohomology group, which we already  used and  defined  in Section~\ref{cup-sect}.

\medskip

Consider an open cover $\{ U_{\alpha} \}$ of $Y$. Fix a \v{C}ech $1$-cocycle $\{ \psi_{\alpha \beta} \}$ of the sheaf of smooth $\mathcal K$-valued functions on $Y$, i.e. $\psi_{\alpha \beta}$
is a smooth function from $U_{\alpha} \cap U_{\beta}$ to $\mathcal K$ for any $\alpha, \beta$, and $\psi_{\alpha \beta} \psi_{\beta \gamma} = \psi_{\alpha \gamma}  $ on $U_{\alpha} \cap U_{\beta} \cap U_{\gamma}$ for any $\alpha, \beta, \gamma$.

The following lemma easily follows by direct calculation.
\begin{lemma}  \label{lemma-last}
Any \v{C}ech $1$-cocycle $\{ \psi_{\alpha \beta } \}$ of the sheaf of smooth  $\mathcal K$-valued functions on $Y$ for an open cover $\{ U_{\alpha} \}$ of $Y$
defines the cochain map from complex~\eqref{van-est} when ${\mathcal N} = \C^*$ to the \v{C}ech complex  of the sheaf $\oo_Y^*$ on $Y$ for an open cover $\{ U_{\alpha} \}$ of $Y$ by a rule
$$
c \longmapsto c (\psi_{{\alpha}_0 {\alpha}_1} |_{U_{{\alpha}_0 \ldots {\alpha}_k}}  , \psi_{{\alpha}_1 {\alpha}_2}  |_{U_{{\alpha}_0 \ldots {\alpha}_k}}  , \ldots , \psi_{{\alpha}_{k-1} {\alpha}_k} |_{U_{{\alpha}_0 \ldots {\alpha}_k}} ) \, \in \, \oo_Y^*(U_{{\alpha}_0 \ldots {\alpha}_k}) \,  \mbox{,}
$$
where an open set $U_{{\alpha}_0 \ldots {\alpha}_k} = U_{{\alpha}_0} \cap \ldots \cap U_{{\alpha}_k}$, elements
$c \in C^k({\mathcal K}, \C^*)$,  $k \ge 1$. For $k=0$ we consider the obvious identity map.
\end{lemma}

\medskip

Now we say some words about $\oo_Y^*$-gerbes on $Y$.

By definition, $\oo_Y^*$-gerbe on $Y$
is a sheaf of $\oo_Y^*$-groupoids on $Y$  that satisfies some properties, see, e.g.,~\cite[\S~5.2]{Bry}
(more exactly, an  $\oo_Y^*$-gerbe is a locally connected sheaf of $\oo_Y^*$-groupoids with the descent properties).
 The set formed by $\oo_Y^*$-gerbes up to equivalence is
 identified with
 $$\check{H}^2(Y, \oo_Y^*)  \simeq  H^2(Y, \oo_Y^*) \simeq H^3(Y, \dz) \, \mbox{,}$$
 where $\check{H}^2(Y, \oo_Y^*)$ is the second \v{C}ech cohomology group, see~\cite[Theorem~5.2.8]{Bry}.

Let ${\mathcal P}$ be a principal  $\mathcal K$-bundle over $Y$. Consider a central extension of Lie groups (see Definition~\ref{prin-bun})
\begin{equation}  \label{centr-ex}
1 \lrto \C^* \lrto \widetilde{ \mathcal K}  \xlrto{\sigma} {\mathcal K}  \lrto 1   \mbox{.}
\end{equation}
Then there is a canonical $\oo_Y^*$-gerbe on $Y$ (which is also  called { \em the lifting gerbe}) whose class in $\check{H}^2(Y, \oo_Y^*)$ is an obstruction to find a principal $\widetilde{\mathcal K}$-bundle $\widetilde{\mathcal P}$ over $Y$ such that the principal $\mathcal K$-bundles
$\widetilde{\mathcal P} / \C^*$ and $\mathcal P$ are isomorphic, see~\cite[Prop.~5.2.3]{Bry}. In other words, this class  is an obstruction  to lift $\mathcal P$ to a principal $\widetilde{\mathcal K}$-bundle.  More exactly, the corresponding sheaf of $\oo_Y^*$-groupoids (which is the lifting gerbe) consists of local  lifts  on $Y$
of the structure of the principal $\mathcal K$-bundle
 to  a structure of a principal $\widetilde{\mathcal K}$-bundle just described, and locally such a lift always exists, since   ${ \mathcal P}$ is a locally trivial bundle.

The lifting gerbe  can be described by more explicit data.
Consider an open cover $\{ U_{\alpha} \} $ of $Y$ such that the restriction ${\mathcal P} |_{U_{\alpha}}$ is the trivial principal $\mathcal K$-bundle for any $\alpha$.
Let  $\{ \psi_{\alpha \beta } \}$ be a \v{C}ech $1$-cocycle of the sheaf of smooth  $\mathcal K$-valued functions on $Y$ that defines $\mathcal P$ with respect to this cover. Consider the principal $\C^*$-bundle
(or $\oo_{U_{\alpha \beta}}^*$-torsor)
$T_{\alpha \beta} = \sigma^{-1}(\psi_{\alpha \beta })$ on $U_{\alpha \beta}$ for any $\alpha, \beta$. Then the collection $\{ T_{\alpha \beta} \}$ glues the lifting gerbe  on $Y$ from the trivial  $\oo_{U_{\alpha}}^*$-gerbes
on every $U_{\alpha}$.
Besides, for any $\alpha$, $\beta$ and $\gamma$ there is a canonical isomorphism of $\oo_{U_{\alpha \beta \gamma}}^*$-torsors on $U_{\alpha \beta \gamma}$
$$
T_{\alpha \beta} |_{U_{\alpha \beta \gamma}} \otimes_{\oo_{U_{\alpha \beta \gamma}}^*}  T_{\beta \gamma }  |_{U_{\alpha \beta \gamma}} \otimes_{\oo_{U_{\alpha \beta \gamma}}^*}  T_{\gamma \alpha} |_{U_{\alpha \beta \gamma}} \, \simeq  \, \oo_{U_{\alpha \beta \gamma}}^*  \, \mbox{.}
$$
 When every $T_{\alpha \beta}$ has a section $s_{\alpha \beta}$ (this will   always be the case when we consider a good open cover of $Y$ that is   a refinement of the open cover $\{ U_{\alpha} \} $), then
the collection of functions
\begin{equation}  \label{cech-coc}
s_{\alpha \beta} |_{U_{\alpha \beta \gamma}} \otimes s_{\beta \gamma } |_{U_{\alpha \beta \gamma}} \otimes s_{\gamma \alpha} |_{U_{\alpha \beta \gamma}}  \, \in \, \oo_{Y}^*(U_{\alpha \beta \gamma})
\, \mbox{,}\quad \mbox{where} \quad \alpha , \beta , \gamma \quad \mbox{are any,}
\end{equation}
is a \v{C}ech $2$-cocycle that defines the lifting gerbe.

Besides, when there is a smooth section (in general, non-group) $\tau \, : \,  {\mathcal K} \to \widetilde{ \mathcal K}$ of the
map $\sigma$ in the
central extension~\eqref{centr-ex}, then
this section defines the  sections $s_{\alpha \beta}$ of  $T_{\alpha \beta}$ for any $\alpha, \beta$.
On the other hand, the section $\tau$ defines the smooth $2$-cocycle on the group $\mathcal K$ with coefficients in the group $\C^*$  (as in Theorem~\ref{th-3}).
Now it is easy to see that
the image of this smooth group $2$-cocycle under the map in Lemma~\ref{lemma-last}
coincides with
the  \v{C}ech $2$-cocycle~\eqref{cech-coc} for the constructed sections $s_{\alpha \beta}$.

\subsection{Topological Riemann-Roch theorem}
\label{tRR}

Let $Y$ be a smooth (finite-dimensional) manifold. We will use notation from previous Section~\ref{gerbes}.

Let $L$ be a complex line bundle on $Y$.  The isomorphism class of $L$ corresponds to the element from $H^1(Y, \oo_Y^*)$.
Using the exponential sheaf sequence on $Y$ we have a map (which is an isomorphism)  $H^1(Y, \oo_Y^*)  \to H^2(Y, \dz)$. The first Chern class
$$c_1(L)  \in H^2(Y, \dz)$$
of $L$  is the image of the class of  $L$ under this map.

Denote by $\overline{L}$ the  complex conjugate line bundle on $Y$, which is given by the complex conjugate \v{C}ech $1$-cocycle for $L$. From the exponential sheaf sequence on $Y$ it is easy to see that
$c_1(\overline{L}) = - c_1(L)$.

Let $\kappa$ be a real line bundle on $Y$. Denote by $\kappa_{\C}$ the complex line bundle on $Y$ which is the complexification of $\kappa$.  It is clear that $\overline{\kappa_{\C}}  \simeq \kappa_{\C}$.
Since $c_1(\overline{\kappa_{\C}}) = - c_1(\kappa_{\C})$, we have $2 c_1(\kappa_{\C}) =0 $ in $H^2(Y, \dz)$. (This is part of  more general statement that the odd Chern classes of the complexification of a real vector bundle are  elements of order $2$, see~\cite[\S~15]{MS}.)

\bigskip

As in Section~\ref{Gysin}  we fix  a fibration in oriented circles
$\pi : M \to B$. Let $L$ be a complex line bundle on $M$.

Then $M$ is a locally trivial fibration over $B$, and hence $L$ is also a locally trivial fibration over $B$.
Moreover, there is a unique analytic structure on $B$ compatible with the given smooth structure, and the same is true for $L$; in particularly,
$L$ admits a collection of the analytic transition function on $B$, see~\cite[\S~2]{Sh} (after the works of
H.~Whitney, H.~Grauert and C.~B.~Morrey).

This means that $L$ as a locally trivial fibration over $B$  can be given by a collection of transition functions with value in the Lie group $\G =\h^*  \rtimes \DS  $ (recall Definition~\ref{d1}). {\em We will assume this further.}

Construct by $L$ the (right) principal $\G$-bundle  ${\mathcal P}_L$ over $B$ by  changing every fiber $L |_{\pi^{-1}(b)}$ of $L$ over the base $B$ to the space of analytic isomorphisms  from  $\C^* \times S^1$, where ${S^1 = \{z \in \mathbb{C} \, \mid  \,  |z|  =1   \}}$, to $L |_{\pi^{-1}(b)} \setminus \{\mbox{zero section} \}$ such that these isomorphisms maps every fiber to a fiber  in the fibrations
$$\C^* \times S^1  \lrto S^1 \qquad \mbox{and}  \qquad
L |_{\pi^{-1}(b)} \setminus \{\mbox{zero section} \}  \lrto \pi^{-1}(b)$$
 and commute with the action of $\C^*$ on fibers.

 The principal
$\G$-bundle ${\mathcal P}_L$ is given by the same transitions functions  with value in $\G$ as the locally trivial fibration $L$ over the base $B$.

As in Section~\ref{gerbes}, consider the lifting gerbe  of the principal $\G$-bundle ${\mathcal P}_L$ and
the determinant central extension of $\G$ by $\C^*$ (recall the central extension~\eqref{det-centr}). We will call this gerbe on $B$ as {\em the determinant gerbe} of $L$ and denote it by ${\mathcal Det}(L)$.  Denote by $[{\mathcal Det}(L)]$  the class of the determinant gerbe ${\mathcal Det}(L)$ in $H^3(B, \dz)$.

\medskip

Now we have the topological Riemann-Roch theorem.
\begin{Th}  \label{th-last}
Let $\pi : M \to B$ be  a fibration in oriented circles and $L$ be a complex line bundle on $M$. In the group $H^3(B, \dz)$ we have
$$
12 \, [{\mathcal Det}(L)] = 6  \, \pi_* (c_1(L) \cup c_1(L))  \, \mbox{.}
$$
\end{Th}
\begin{proof}
Define the complex line bundle $\Omega_{M/B} = (T^*M / \pi^*(T^* B))_{\C} $ on $M$ which is the complexification of the real line bundle $T^*M / \pi^*(T^* B)$,
where $T^*M$ and $T^*B$ are cotangent vector bundles on $M$ and $B$ correspondingly. The sheaf of sections of  $\Omega_{M/B}$ is the sheaf of relative differential $1$-forms with complex coefficients.

First we prove the following equality in $H^3(B, \dz)$:
\begin{equation}  \label{top-1}
12 \, [{\mathcal Det}(L)] = 6  \, \pi_* (c_1(L) \cup c_1(L)) - 6 \, \pi_* (c_1(L) \cup c_1(\Omega_{M/B})) + \pi_* (c_1(\Omega_{M/B}) \cup c_1(\Omega_{M/B}))  \, \mbox{.}
\end{equation}

 Consider   an open cover $\{ U_{\alpha} \}$ of $B$ such that the principal $\g$-bundle ${\mathcal P}_{L}$ is given by a \v{C}ech   $1$-cocycle $\{ \psi_{\alpha \beta} \}$ with respect to this cover.
 Fix isomorphisms
 $$\Phi_{\alpha} \, : \, L |_{\pi^{-1}(U_{\alpha)}}  \lrto \C \times S^1 \times U_{\alpha}$$
 which lead to this \v{C}ech $1$-cocycle. These isomorphisms induce the following isomorphisms
 $$
 \widetilde{\Phi}_{\alpha} \, : \,  \pi^{-1}(U_{\alpha})  \lrto  S^1 \times U_{\alpha}  \, \mbox{,}  \qquad  \widetilde{\Phi}_{\alpha}^*   \, : \, \oo_{S^1 \times U_{\alpha}}  \lrto \oo_{ \pi^{-1}(U_{\alpha})}  \, \mbox{.}
 $$

 Let $\lambda : \G \to \h^*$ be a smooth group $1$-cocycle (see Sections~\ref{cup-sect} and~\ref{gerbes}). Then $\left\{ \widetilde{\Phi}_{\alpha}^* \left( \lambda \circ \psi_{\alpha \beta} \right)   \right\}$
is a \v{C}ech $1$-cocycle for the sheaf $\pi_* \oo_M^*$ and the open cover $\{ U_{\alpha} \}$ of $B$,
 where $\lambda \circ \psi_{\alpha \beta}$ is the composition of maps $\lambda$ and $\psi_{\alpha \beta}$.
 Since $H^1(B, \pi_* \oo_M^*) = H^1(M, \oo_M^*)$, this $1$-cocycle
 $\left\{ \widetilde{\Phi}_{\alpha}^* \left( \lambda \circ \psi_{\alpha \beta} \right)   \right\}$
  corresponds
to a certain complex line bundle $Q_{\lambda}$ on $M$.

 Recall that  the $\cup$-product in \v{C}ech cohomology
$$
\check{H}^1(\{ U_{\alpha} \} , \pi_* \oo_M^*) \, \cup \, \check{H}^1( \{ U_{\alpha} \} , \pi_* \oo_M^*)  \lrto \check{H}^2( \{ U_{\alpha} \}, \pi_* \oo_M^*  \otimes_{\dz} \pi_* \oo_M^*)
$$
can be given on \v{C}ech $1$-cocycles $\{ \rho_{\alpha \beta}   \}$ and $\{ \mu_{\alpha \beta}    \}$  in the following way:
$$
\{ \rho_{\alpha \beta}    \}    \, \cup \, \{ \mu_{\alpha \beta}    \}  \longmapsto  \{  \chi_{\alpha \beta \gamma}  \}  \, \mbox{,} \quad  \mbox{where} \quad
\chi_{\alpha \beta \gamma} = \rho_{\alpha \beta} |_{U_{\alpha \beta \gamma}}  \otimes \mu_{\beta  \gamma} |_{U_{\alpha \beta \gamma}} $$
 for any $\alpha$, $\beta$ and $\gamma$  (here $\chi_{\alpha \beta \gamma}  \in H^0(U_{\alpha \beta \gamma},  \pi_* \oo_M^*  |_{U_{\alpha \beta \gamma}}  \otimes_{\dz} \pi_* \oo_M^*  |_{U_{\alpha \beta \gamma}} )$).

Let $\lambda_i : \G \to \h^*$, where $i=1$ and $i=2$, be two   smooth group $1$-cocycles.
Recall that in Definition~\ref{bracket} we defined the group $2$-cocycle
 $\langle \lambda_1, \lambda_2   \rangle $ on the group $\G$ with coefficients in the group $\C^*$.

 Now, using Theorem~\ref{th-9}, it is easy to see that  the  element $\pi_* (c_1(Q_{\lambda_1}) \cup c_1(Q_{\lambda_2}) )$ in the group $H^3(B, \dz)$ coincides with the image of the group $2$-cocycle
 $\langle \lambda_1, \lambda_2   \rangle $ in ${H^2(B, \oo_B^*) \simeq H^3(B, \dz)}$ under the map in Lemma~\ref{lemma-last}.

In Section~\ref{cup-sect} we considered the special smooth  $1$-cocycles  on the group $\G$ with coefficients in the group $\h^*$:  $\Lambda$  and $\Omega$.

 Note that if $\lambda = \Lambda$, then $Q_{\lambda} \simeq L$.  If $\lambda = \Omega$, then $Q_{\lambda} \simeq \Omega_{M/B}$, besides  $d \left(\widetilde{\Phi}_{\alpha}^*(z) \right)$ is a local section of the last line bundle.

 From the reasoning at the end of Section~\ref{gerbes}  it follows that the element $[{\mathcal Det}(L)]$ in $H^3(B, \dz)$ coincides with the image in $H^2(B, \oo_B^*) \simeq H^3(B, \dz)$ under the map in Lemma~\ref{lemma-last}  of a group $2$-cocycle that defines the determinant central extension  (see Remark~\ref{sect-g}).

 Now   equality~\eqref{top-1} follows from Theorem~\ref{Th-eq}  and Lemma~\ref{lemma-last}.

Now we prove the statement of the theorem.

Note that $2 \, \pi_* (c_1(L) \cup c_1(\Omega_{M/B}))= 0$, since $c_1(\Omega_{M/B})$ is an element of order~$2$.

We will prove that
\begin{equation}  \label{fin-form}
\pi_* (c_1(\Omega_{M/B}) \cup c_1(\Omega_{M/B})) = 0  \, \mbox{.}
\end{equation}

Denote by $ \mathop{\pi_{*, 0}}  \oo_{M}^*  \subset  \pi_* \oo_M^*$ the subsheaf of those functions from the sheaf $\pi_* \oo_M^*$ that after the restriction to every fiber of the map $\pi$ have the winding number $\nu$ equal to $0$.
By Proposition~\ref{wind-diff}, the class of $\Omega_{M/B}$ in $H^1(B, \pi_* \oo_M^*)$
comes from the element in $H^1(B, \mathop{\pi_{*, 0}} \oo_{M}^* )$  under the sheaf  embedding $ \mathop{\pi_{*, 0}} \oo_{M, 0}^*  \subset \pi_* \oo_M^*$.

Now by Theorem~\ref{th-9}, it is enough to prove that the map of sheaves
$$ \pi_* \oo_M^*  \otimes_{\dz} \pi_* \oo_M^*  \xlrto{\mathbb T} \oo_B^*  $$
restricted to the subsheaf $\mathop{\pi_{*, 0}} \oo_{M}^*  \otimes_{\dz} \mathop{\pi_{*, 0}} \oo_{M}^*$
induces the trivial homomorphism from $H^2(B, \mathop{\pi_{*, 0}} \oo_{M}^*  \otimes_{\dz} \mathop{\pi_{*, 0}} \oo_{M}^*)$ to $H^2(B, \oo_B^*)$.

From formula~\eqref{expli} for the map $\mathbb T$
it follows  that the map
$$\mathop{\pi_{*, 0}} \oo_{M}^*  \otimes_{\dz} \mathop{\pi_{*, 0}} \oo_{M}^*  \xlrto{\mathbb T} \oo_B^*$$
is the composition of the following maps:
$$
\mathop{\pi_{*, 0}} \oo_{M}^*  \otimes_{\dz} \mathop{\pi_{*, 0}} \oo_{M}^*  \xlrto{\widetilde{\mathbb T}} \oo_B  \xlrto{\exp} \oo_B^*  \,  \mbox{,}
$$
where the map $\widetilde{\mathbb T}$ is induced by the map (denoted by the same letter) $\widetilde{\mathbb T}$ on   every fiber $S^1$ of the map $\pi$:
\begin{equation}  \label{wT}
\widetilde{\mathbb T}(f,g)=  \frac{1}{2 \pi i} \oint_{S^1}  \log f \,  d \log g   \, \mbox{,}
\end{equation}
where the  functions $f$ and  $g$ are
from the group $C^{\infty}(S^1, \C^*)$, and $\nu(f)= \nu(g)=0$.
The map~\eqref{wT} is well-defined, it does not depend on the choice of the branches $\log f$ and~$\log g$.

Now we use that $H^2(B, \oo_B) =0$, and thus we proved  formula~\eqref{fin-form}.

\end{proof}

\begin{nt} \em
In the group $H^3(B, \C)$ the equality $2  [{\mathcal Det}(L)] =    \pi_* (c_1(L) \cup c_1(L))$ (see notation in Theorem~\ref{th-last}) was proved by another methods in~\cite{BKTV}.
\end{nt}

\vspace{0.3cm}

\noindent Steklov Mathematical Institute of Russsian Academy of Sciences, 8 Gubkina St., Moscow 119991, Russia,
 {\em and}

 \noindent National Research University Higher School of Economics, Laboratory of Mirror Symmetry,  6 Usacheva str., Moscow 119048, Russia,
{\em and}

\noindent National University of Science and Technology ``MISiS'',  Leninsky Prospekt 4, Moscow  119049, Russia

\noindent {\it E-mail:}  ${d}_{-} osipov@mi{-}ras.ru$


\begin{thebibliography}{99}

\bibitem{AST}
A.~Alekseev,  S.~Shatashvili,  L.~Takhtajan, {\em
 Berezin quantization, conformal welding and the Bott-Virasoro group},
Ann. Henri Poincar\'{e} 25 (2024), no. 1, 35--64.


\bibitem{ADKP}
E.~Arbarello, C.~De Concini, V.~G.~Kac, C.~Procesi, {\em  Moduli spaces of curves and representation theory,} Comm. Math. Phys. 117 (1988), no.~1, 1--36.

\bibitem{Be} A.~A.~Beilinson, {\em
Higher regulators and values of $L$-functions of curves.}
Funktsional. Anal. i Prilozhen.   14 (1980), no. 2, 46--47; english transl. in
Functional Analysis and Its Applications, 14 (1980),  no. 2,  116--118.


\bibitem{Bl}
S.~Bloch, {\em
The dilogarithm and extensions of Lie algebras.} Algebraic  K -theory, Evanston 1980 (Proc. Conf., Northwestern Univ., Evanston, Ill., 1980),  1--23.
Lecture Notes in Math., 854
Springer, Berlin, 1981.


\bibitem{BM}
S.~Bochner, W.~T.~Martin, {\em
Several Complex Variables},
Princeton Math. Ser., vol. 10
Princeton University Press, Princeton, NJ, 1948, ix+216 pp.


\bibitem{Bre} G.~E.~Bredon, { \em
Sheaf theory.}
Second edition.
Grad. Texts in Math., 170
Springer-Verlag, New York, 1997. xii+502 pp.

\bibitem{BKTV}
P.~Bressler,  M.~Kapranov, B.~Tsygan, E.~Vasserot,
 {\em
 Riemann-Roch for real varieties},
Progr. Math., 269,
Birkh\"auser Boston, Ltd., Boston, MA, 2009, 125--164.



\bibitem{Bro}
K. S. Brown, {\em Cohomology of groups,} Graduate Texts in Mathematics, 87. Springer-Verlag, New York-Berlin, 1982.

\bibitem{Bry}
J.-L. Brylinski,  {\em Loop spaces, characteristic classes and geometric quantization.} Progress in Mathematics, 107. Birkhäuser Boston, Inc., Boston, MA, 1993.





\bibitem{BD}
J.-L. Brylinski, P. Deligne, {\em Central extensions of reductive groups by $K_2$,} Publ. Math.
Inst. Hautes \'{E}tudes Sci. No. 94 (2001),  5--85.




\bibitem{CC1}
C. Contou-Carr\`{e}re,  {\em Jacobienne locale, groupe de bivecteurs de Witt universel, et symbole mod\'{e}r\'{e},}
C. R. Acad. Sci. Paris S\'er. I Math., {\bf 318}:8 (1994), 743--746.


\bibitem{D1}
P.~Deligne, {\em
Le d\'{e}terminant de la cohomologie.} (French)  Current trends in arithmetical algebraic geometry (Arcata, Calif., 1985), 93--177,
Contemp. Math., 67, Amer. Math. Soc., Providence, RI, 1987.



\bibitem{D2}
P.~Deligne, {\em Le symbole mod\'{e}r\'{e},}   Inst. Hautes \'{E}tudes Sci. Publ. Math. No. 73 (1991), 147--181.

\bibitem{dual}
Duality in complex analysis. {\em Encyclopedia of Mathematics.} URL: \href{http://encyclopediaofmath.org/index.php?title=Duality_in_complex_analysis&oldid=50700}{https://encyclopediaofmath.org/wiki/Duality\_in\_complex\_analysis}


\bibitem{EV}
H.~Esnault, E.~Viehweg, {\em
Deligne-Beilinson cohomology}, Beilinson's conjectures on special values of  L-functions, 43--91.
Perspect. Math., 4
Academic Press, Inc., Boston, MA, 1988.



\bibitem{F} K.~Floret, {\em Lokalkonvexe Sequenzen mit kompakten Abbildungen}, J. Reine Angew. Math. 247 (1971), 155--195.


\bibitem{FW}
K.~Floret, J.~Wloka, {\em
Einf\"uhrung in die Theorie der lokalkonvexen R\"aume},
Lecture Notes in Math., No.~56,
Springer-Verlag, Berlin-New York, 1968, vii+194 pp.


\bibitem{GO} S.~O.~Gorchinskiy, D.~V.~Osipov, {\em A higher-dimensional Contou-Carr\`{e}re symbol: local
 theory.} (Russian) Mat. Sb. 206 (2015), no. 9, 21--98; translation in Sb. Math. 206
 (2015), no. 9--10, 1191--1259.


\bibitem{Gro0}  A.~Grothendieck, {\em Produits tensoriels topologiques et espaces nucl\'{e}aires.}
Mem. Amer. Math. Soc. 16 (1955).

\bibitem{Gro} A.~Grothendieck,  {\em La th\'{e}orie de Fredholm},
Bull. Soc. Math. France 84 (1956), 319--384.

\bibitem{H}
F.~Hartogs, {\em Zur Theorie der analytischen Funktionen mehrerer unabh\"angiger Ver\"anderlichen, insbesondere \"uber die Darstellung derselben durch Reihen, welche nach Potenzen einer Ver\"anderlichen fortschreiten}, Math. Ann., 62 (1906), 1--88.





\bibitem{KW}
B. Khesin, R. Wendt, {\em The geometry of infinite-dimensional groups.} Ergebnisse der Mathematik und ihrer Grenzgebiete. 3. Folge. A Series of Modern Surveys in Mathematics [Results in Mathematics and Related Areas. 3rd Series. A Series of Modern Surveys in Mathematics], 51. Springer-Verlag, Berlin, 2009.

\bibitem{K}
A.~A.~Kirillov,  {\em K\"ahler structures on $K$-orbits of the group of diffeomorphisms of a circle.}
Functional Analysis and Its Applications, 1987, Volume 21, Issue 2,  122--125.


\bibitem{LR}
 D.~H.~Luecking, L.~A.~Rubel, {\em Complex Analysis: A Functional Analysis Approach.}
 Universitext,
Springer-Verlag, New York, 1984, vii+176 pp.


\bibitem{MS}
J. W. Milnor, J. D. Stasheff, {\em
 Characteristic classes.}
Ann. of Math. Stud., No. 76
Princeton University Press, Princeton, NJ; University of Tokyo Press, Tokyo, 1974, vii+331 pp.


\bibitem{Mor} M.~Morimoto, {\em
An introduction to Sato's hyperfunctions},
Transl. Math. Monogr., 129
American Mathematical Society, Providence, RI, 1993, xii+273 pp.

\bibitem{Neeb2}  K.-H.~Neeb, {\em Central extensions of infinite-dimensional Lie groups}, Ann. Inst. Fourier (Grenoble) 52 (2002), no. 5, 1365--1442.


\bibitem{Neeb1} K.-H.~Neeb, {\em
Towards a Lie theory of locally convex groups},
Jpn. J. Math. 1 (2006), no. 2, 291--468.

\bibitem{Ner1}
Yu.~A.~Neretin, {\em
Holomorphic continuations of representations of the group of diffeomorphisms of the circle.}(Russian)
Mat. Sb. 180 (1989), no. 5, 635--657; translation in
Math. USSR-Sb. 67 (1990), no. 1, 75--97.


\bibitem{Ner2}
Yu.~A.~Neretin, {\em
Categories of symmetries and infinite-dimensional groups.}
Translated from the Russian by G.~G.~Gould.
London Math. Soc. Monogr. (N.S.), 16
Oxford Sci. Publ.
The Clarendon Press, Oxford University Press, New York, 1996. xiv+417 pp.

\bibitem{Ner3}
Yu.~A.~Neretin, {\em
The group of diffeomorphisms of the circle: reproducing kernels and analogs of spherical functions.}
J. Geom. Mech. 9 (2017), no. 2, 207--225.


\bibitem{O1}
D. V. Osipov, {\em Formal Bott-Thurston cocycle and part of a formal Riemann-Roch theorem,}  Algebra and Arithmetic, Algebraic, and Complex Geometry, Collected papers. Dedicated to the memory of Academician Aleksei Nikolaevich Parshin, Tr. Mat. Inst. Steklova 320 (2023), 243--277; English translation in  Proc. Steklov Inst. Math. 320 (2023), 226 --257; e-print arXiv: 2211.15932.



\bibitem{O2} D.~V.~Osipov, {\em Local analog of the Deligne--Riemann--Roch isomorphism for line bundles in relative dimension $1$}, Izvestiya: Mathematics, 2024, Volume 88, Issue 5,  930--976.

\bibitem{OZ}
D.~Osipov, X.~Zhu, {\em  The two-dimensional Contou-Carr\`{e}re symbol and reciprocity laws,} J. Algebraic Geom. 25 (2016), no. 4, 703--774.


\bibitem{PS} A.~Pressley, G.~Segal, {\em
Loop groups},
Oxford Mathematical Monographs, Oxford Science Publications, The Clarendon Press, Oxford University Press, New York, Revised ed. edition 1988.


\bibitem{Se} G.~Segal, {\em Unitary representations of some infinite-dimensional groups.} Comm. Math. Phys. 80 (1981), no. 3, 301--342.


\bibitem{Sh} K.~Shiga, {\em  Some aspects of real-analytic manifolds and differentiable manifolds},
J. Math. Soc. Japan 16 (1964), 128--142.

\bibitem{Si} B.~Simon, {\em Trace ideals and their applications}.
Math. Surveys Monogr., 120
American Mathematical Society, Providence, RI, 2005, viii+150 pp.









\end{thebibliography}
\end{document}